\documentclass[11pt,letterpaper,reqno]{amsart}
\usepackage[english]{babel}
\usepackage[T1]{fontenc}
\usepackage{verbatim}
\usepackage{palatino}
\usepackage{amsmath}
\usepackage{mathabx}
\usepackage{amssymb}
\usepackage{amsthm}
\usepackage{amsfonts}
\usepackage{overpic}
\usepackage{esint}
\usepackage{xcolor}
\usepackage{mathtools}
\usepackage{enumerate}

\DeclarePairedDelimiter\floor{\lfloor}{\rfloor}
\usepackage[colorlinks = true, linkcolor={red!80!black}, citecolor = {blue!60!black}]{hyperref}
\author{Tuomas Orponen}
\address{Department of Mathematics and Statistics\\ University of Jyv\"askyl\"a,
P.O. Box 35 (MaD)\\
FI-40014 University of Jyv\"askyl\"a\\
Finland}
\email{tuomas.t.orponen@jyu.fi}

\author{Pablo Shmerkin}
\address{Department of Mathematics \\
The University of British Columbia \\
1984 Mathematics Road\\
Vancouver, BC, V6T 1Z2\\
Canada}
\email{pshmerkin@math.ubc.ca}

\title[Furstenberg sets and projections]{On the Hausdorff dimension of Furstenberg sets and orthogonal projections in the plane}
\date{\today}
\thanks{T.O. is supported by the Academy of Finland via the projects \emph{Quantitative rectifiability in Euclidean and non-Euclidean spaces} and \emph{Incidences on Fractals}, grant Nos. 309365, 314172, 321896. P.S. was supported by an NSERC discovery grant}

\newcommand{\R}{\mathbb{R}}

\newcommand{\N}{\mathbb{N}}

\newcommand{\Z}{\mathbb{Z}}

\newcommand{\calT}{\mathcal{T}}

\newcommand{\calD}{\mathcal{D}}
\newcommand{\calH}{\mathcal{H}}

\newcommand{\calQ}{\mathcal{Q}}

\newcommand{\spt}{\operatorname{spt}}
\newcommand{\Hd}{\dim_{\mathrm{H}}}

\newcommand{\spa}{\operatorname{span}}
\newcommand{\diam}{\operatorname{diam}}

\newcommand{\dist}{\operatorname{dist}}

\newcommand{\e}{\epsilon}
\newcommand{\wt}{\widetilde}

\def\Barint_#1{\mathchoice
          {\mathop{\vrule width 6pt height 3 pt depth -2.5pt
                  \kern -8pt \intop}\nolimits_{#1}}%
          {\mathop{\vrule width 5pt height 3 pt depth -2.6pt
                  \kern -6pt \intop}\nolimits_{#1}}%
          {\mathop{\vrule width 5pt height 3 pt depth -2.6pt
                  \kern -6pt \intop}\nolimits_{#1}}%
          {\mathop{\vrule width 5pt height 3 pt depth -2.6pt
                  \kern -6pt \intop}\nolimits_{#1}}}

\theoremstyle{plain}
\newtheorem{thm}{Theorem}
\newtheorem*{"thm"}{"Theorem"}

\newtheorem{lemma}[thm]{Lemma}

\newtheorem{cor}[thm]{Corollary}
\newtheorem{proposition}[thm]{Proposition}

\theoremstyle{definition}

\newtheorem{definition}[thm]{Definition}

\theoremstyle{remark}
\newtheorem{remark}[thm]{Remark}

\numberwithin{equation}{section}
\numberwithin{thm}{section}

\newcommand{\nref}[1]{(\hyperref[#1]{#1})}

\DeclareMathSymbol{\intop}  {\mathop}{mathx}{"B3}

\begin{document}

\begin{abstract} Let $0 \leq s \leq 1$ and $0 \leq t \leq 2$. An \emph{$(s,t)$-Furstenberg set} is a set $K \subset \R^{2}$ with the following property: there exists a line set $\mathcal{L}$ of Hausdorff dimension $\Hd \mathcal{L} \geq t$ such that $\Hd (K \cap \ell) \geq s$ for all $\ell \in \mathcal{L}$. We prove that for $s\in (0,1)$, and $t \in (s,2]$, the Hausdorff dimension of $(s,t)$-Furstenberg sets in $\R^{2}$ is no smaller than $2s + \epsilon$, where $\epsilon > 0$ depends only on $s$ and $t$. For $s>1/2$ and $t = 1$, this is an $\epsilon$-improvement over a result of Wolff from 1999.

The same method also yields an $\epsilon$-improvement to Kaufman's projection theorem from 1968. We show that if $s \in (0,1)$, $t \in (s,2]$ and $K \subset \R^{2}$ is an analytic set with $\Hd K = t$, then
\begin{displaymath} \Hd \{e \in S^{1} : \Hd \pi_{e}(K) \leq s\} \leq s - \epsilon, \end{displaymath}
where $\epsilon > 0$ only depends on $s$ and $t$. Here $\pi_{e}$ is the orthogonal projection to the line in direction $e$. \end{abstract}

\maketitle

\tableofcontents

\section{Introduction}

\subsection{Main results}

The purpose of this paper is to prove the following two closely related theorems:

\begin{thm}\label{mainFurstenberg} For every $s \in (0,1)$ and $t\in (s,2]$, there exists $\epsilon = \epsilon(s,t) > 0$ such that the following holds. Let $K \subset \R^{2}$, let $\mathcal{L}$ be a family of lines with $\Hd \mathcal{L} \geq t$, and assume that $\Hd (K \cap \ell) \geq s$ for all $\ell \in \mathcal{L}$. Then $\Hd K \geq 2s + \epsilon$.
\end{thm}
The notation "$\Hd$" stands for Hausdorff dimension. The Hausdorff dimension of line families is defined via point-line duality, see Section \ref{s:prelim} for details. No measurability is assumed on either $\mathcal{L}$ or $K$.

In the next theorem, $\pi_{e} \colon \R^{2} \to \R$, $e \in S^{1}$, stands for the orthogonal projection to the line $\spa(e)$, identified with $\R$. In other words, $\pi_{e}(x) := x \cdot e$ for $x \in \R^{2}$.
\begin{thm}\label{mainKaufman} For every $s \in (0,1)$ and $t \in (s,2]$, there exists a constant $\epsilon = \epsilon(s,t) > 0$ such that the following holds. If $K \subset \R^{2}$ be an analytic set with $\Hd K = t$, then
\begin{displaymath} \Hd \{e \in S^{1} : \Hd \pi_{e}(K) \leq s\} \leq s - \epsilon. \end{displaymath} \end{thm}

Theorems \ref{mainFurstenberg} and \ref{mainKaufman} are so closely related that they can be both deduced from a single $\delta$-discretised statement, as follows:

\begin{thm}\label{t:mainTechnical} For $s \in (0,1)$ and $t\in (s,2]$, there exists $\epsilon(s,t) > 0$ such that the following holds for all small enough $\delta \in 2^{-\N}$, depending only on $s$ and $t$. Let $\mathcal{P} \subset \mathcal{D}_{\delta}$ be a $(\delta,t,\delta^{-\epsilon})$-set with $\cup \mathcal{P} \subset [0,1)^{2}$, and let $\mathcal{T} \subset \mathcal{T}^{\delta}$ be a family of dyadic $\delta$-tubes. Assume that for every $p \in \mathcal{P}$, there exists a $(\delta,s,\delta^{-\epsilon})$-set $\mathcal{T}(p) \subset \mathcal{T}$ such that $T \cap p \neq \emptyset$ for all $T \in \mathcal{T}(p)$. Then $|\mathcal{T}| \geq \delta^{-2s - \epsilon}$. \end{thm}

For the definitions of the concepts appearing in Theorem \ref{t:mainTechnical}, see Section \ref{s:prelim}.  Theorems \ref{mainFurstenberg} and \ref{mainKaufman} will be reduced to Theorem \ref{t:mainTechnical} in Section \ref{s:discretisation}. See also Theorem \ref{t:mainDual} for a dual formulation of Theorem \ref{t:mainTechnical}.

\begin{remark}
In all three theorems stated above, the positive constant $\epsilon(s,t)$ can be taken uniform in any compact subset of $\{ (s,t): 0<s<\min(1,t)<2\}$. This follows since in Theorem \ref{t:mainTechnical}, if $\e>0$ works for a fixed pair $(s,t)$, then $\e/2$ works in the $\e/2$-neighbourhood of $(s,t)$.
\end{remark}

\begin{remark} The assumptions that $\cup \mathcal{P} \subset [0,1)^{2}$ and the tubes are dyadic can be relaxed, see Theorem \ref{t:mainTechnicalGeneral} for the details. \end{remark}

\subsection{Background on Furstenberg sets} In 1999, T. Wolff \cite{Wolff99} introduced what has become known as the Furstenberg set problem. Given a parameter $s\in (0,1)$, an $s$-Furstenberg set $K$ is a compact planar set such that for all directions $e\in S^1$ there is a line $\ell_e$ in direction $e$ such that $\Hd(K\cap\ell_e) \ge s$. The \emph{$s$-Furstenberg set problem} asks for the smallest possible Hausdorff dimension of an $s$-Furstenberg set. This problem has several motivations. It is related to investigations of H. Furstenberg \cite{Furstenberg70} on intersections of Cantor sets defined in terms of expansions to bases $2,3$ (the conjecture that motivated this connection has since been resolved by the second author \cite{Sh} and, independently, by M. Wu \cite{Wu19}). The Furstenberg set problem is also a ``fractal'' version of the Kakeya problem, and finally a ``discretized'' analog of point-line incidence problems in geometric combinatorics, related to the Szemer\'{e}di-Trotter theorem.

Using elementary arguments, Wolff showed that if $K$ is an $s$-Furstenberg set, then
\begin{equation}\label{wolffBound} \Hd(K)\ge \max(\tfrac{1}{2}+s,2s), \end{equation}
and constructed a $s$-Furstenberg sets of dimension $\tfrac{1}{2} +\tfrac{3s}{2}$, for $s \in (0,1]$. It is worth noting that the bounds \eqref{wolffBound} agree, and equal $1$, for $s = \tfrac{1}{2}$. Soon after Wolff's work, N. Katz and T. Tao \cite{MR1856956} connected the $\tfrac{1}{2}$-Furstenberg set problem to two other outstanding problems in geometric measure theory: the Falconer distance set problem and the (discretized) ring conjecture. In 2003, J. Bourgain \cite{Bourgain03} solved the discretized ring  conjecture which, in conjunction with the work of Katz-Tao, established that $\tfrac{1}{2}$-Furstenberg sets have Hausdorff dimension at least $1+\epsilon$ for a small universal $\epsilon>0$. A formal consequence of this result is also an $\epsilon$-improvement over Wolff's bounds for the $(\tfrac{1}{2} + \epsilon)$-Furstenberg set problem.

For $s > \tfrac{1}{2} + \epsilon$, Wolff's bounds \eqref{wolffBound} remained the strongest results in the $s$-Furstenberg set problem, until now. As an immediate corollary to Theorem \ref{mainFurstenberg}, we obtain the following $\epsilon$-improvement for all $s\in (1/2,1)$:
\begin{cor}\label{mainFurstenbergCor}
For every $s\in (\tfrac{1}{2},1)$ there exists $\epsilon=\epsilon(s)>0$ such that the Hausdorff dimension of every $s$-Furstenberg set is at least $2s+\epsilon$.
\end{cor}
We remark that, previously, the first author had shown in \cite{Orponen20} that the \emph{packing} dimension of an $s$-Furstenberg sets is at least $2s+\epsilon$, for $s \in (\tfrac{1}{2},1)$. This result, or more precisely a $\delta$-discretised version of it, plays an important role in this article. A similar improvement for the packing dimension of Furstenberg sets in the regime $s\in (0,\tfrac{1}{2})$ was obtained by the second author \cite{Shmerkin22}. Unfortunately, the methods of this article do not appear to be sufficient to obtain a similar Hausdorff dimension improvement for $s\in (0,\tfrac{1}{2})$.

\subsubsection{$(s,t)$-Furstenberg sets} Theorem \ref{mainFurstenberg} concerns a generalised notion of Furstenberg sets, in which the family of lines is arbitrary, and its Hausdorff dimension, typically denoted "$t$" in this paper, is an additional parameter. The study of such generalised Furstenberg sets was initiated by U. Molter and E. Rela \cite{MolterRela12}. See \cite{MR3973547,MR4002667,LutzStull17, Shmerkin22} for other results in this direction. To be precise, let us say that $K \subset \R^{2}$ is an \emph{$(s,t)$-Furstenberg set} if it satisfies the assumptions of Theorem \ref{mainFurstenberg}. With this terminology, an $(s,1)$-Furstenberg set is a (mild) generalisation of an $s$-Furstenberg set, in the sense of Wolff.

An argument similar to Wolff's proof of the bound $\Hd(K)\ge 2s$ for $s$-Furstenberg sets (see \cite[Appendix A]{HSY22}) shows that $\Hd(K)\ge 2s$ for every $(s,s)$-Furstenberg set $K$. It is also easy to see using a ``Cantor target'' construction that this is sharp. Thus, in Wolff's statement that every $s$-Furstenberg has dimension $\ge 2s$, it is not necessary to use the full force of the assumption: it suffices that there exists a subset $S \subset S^{1}$ of dimension $\Hd S \geq s$ such that $\Hd [K \cap \ell_{e}] \geq s$ for some line $\ell_{e}$ parallel to every vector $e \in S$.

It is natural to expect a stronger lower bound as soon as the assumption $\Hd (S) \geq s$ is upgraded to $\Hd (S) \geq t$ for some $t > s$. Until now, however, this has only been known for $t \ge 2s$. For $s = \tfrac{1}{2}$, this follows from the works of Bourgain \cite{Bourgain03} and Katz-Tao \cite{MR1856956}, already discussed above. For general $s \in (0,1)$, it was shown more recently by K.~H\'era, the second author, and A.~Yavicoli \cite{HSY22} that every $(s,2s)$-Furstenberg set $K \subset \R^{2}$ has $\Hd (K) \geq 2s + \epsilon(s)$. The value of the constant $\epsilon(s)$ was quantified very recently by D.~Di Benedetto and J.~Zahl \cite{BZ21}: for example, it follows from their results that a classical $1/2$-Furstenberg set has Hausdorff dimension at least $1+1/4536$. The case $(s,2s)$ of the Furstenberg set problem is special due to its close connection with the $\delta$-discretized version of the Erd\H{o}s-Szemer\'edi sum-product problem, see \cite{BZ21} for more information.

With this background, we may emphasise the key novelty of Theorem \ref{mainFurstenberg}: for any $t>s$, it gives an $\epsilon$-improvement over "elementary" bounds, whereas earlier works only accomplished this for $t = 2s$. Equivalently, if $t \in (0,2)$ is fixed, earlier works only gave an $\epsilon$-improvement over elementary bounds if $s = t/2$: the case $t = 1$ is particularly relevant for Corollary \ref{mainFurstenbergCor}.

Finally, we mention other recent results related to the Furstenberg set problem. R. Zhang \cite{Zhang17} completely resolved a discrete analog of the problem. More recently, L.~Guth, N.~Solomon and H.~Wang \cite{GSW19} proved incidence theorems for "well spaced" tubes. These can be seen to imply very strong bounds for the size of $\delta$-discretised Furstenberg-type sets, where lines are replaced by "well-spaced" $\delta$-tubes. Some further developments took place after a preprint of this article was first posted to the arXiv. In the article \cite{MR4452675}, D.~D{\k a}browski, the first author and M.~Villa  obtained a new explicit bound for $(s,t)$-Furstenberg sets for $t>1$, namely $\Hd K \geq 2s + (t - 1)(1 - s)$, which improves the bound in Theorem \ref{mainFurstenberg} when $t>1+\e$ for some absolute $\e>0$. This bound was further superseded by Y.~Fu-K.~Ren \cite{2021arXiv211105093F}, who proved that $\Hd K \geq \min\{2s + t - 1,1 + s\}$. The paper \cite{FuGanRen} of Y.~Fu-S.~Gan-K.~Ren also studies a closely related problem. The results and techniques in these works are very different from the ones in this article. Roughly speaking, the dimensional or spacing assumptions in \cite{MR4452675, FuGanRen, 2021arXiv211105093F,GSW19} are such that the discretised sum-product problem (or approximate Borel sub-ring obstruction) is not relevant. This, in particular, allows for more quantitative estimates for $(s,t)$-Furstenberg sets for $t > 1$.

\subsection{Background on projections} The study of the effect of projections on dimension is one of the oldest and most intensively studied problems in fractal geometry. See the surveys \cite{FFJ15, Mattila19} for an introduction to this vast area. J.~Marstrand's celebrated projection theorem \cite{Mar} asserts that if $K\subset\R^2$ is an analytic set of dimension $s\le 1$, then $\Hd \pi_e(K)=s$ for $\mathcal{H}^{1}$ almost all $e\in S^1$. In 1968, R.~Kaufman \cite{Ka} found a simple proof of Marstrand's theorem that also gave the following sharpening: if the analytic set $K \subset\mathbb{R}^2$ has Hausdorff dimension $s \in [0,1]$, then
\begin{equation} \label{eq:kaufman}
\Hd(\{ e\in S^1: \Hd\pi_e (K)< s \}) \leq s.
\end{equation}
It is natural to ask whether the right-hand side can be improved if one assumes, instead, that $\Hd(K)=t$ for some $t>s$. Kaufman's method is not able to exploit this ``extra largeness'' of the set $K$. This is similar (not coincidentally) to the phenomenon we described for $(s,t)$-Furstenberg sets. A different argument that can be perhaps considered ``folklore'' shows that for a planar analytic set $K$ with $\Hd(K)=t\in (0,2]$,
\begin{equation}\label{babyBourgain}
\Hd(\{ e\in S^1: \Hd \pi_e(K) \leq s\})=0, \qquad 0 \leq s < \tfrac{t}{2}. \end{equation}
Thus, Kaufman's bound \eqref{eq:kaufman} is far from sharp for $s<t/2$. An influential result of J. Bourgain \cite{Bourgain03, Bourgain10} extends \eqref{babyBourgain} to the case $s = t/2$, and even gives the following information for $s > t/2$ "close" to $t/2$: for every $t\in (0,2)$ and $\epsilon>0$, there is $\kappa=\kappa(\epsilon,t)>0$ such that
\[
\Hd(\{ e\in S^1: \Hd \pi_e(K) < t/2 + \kappa\})\le \epsilon.
\]
Taking $\epsilon<s$, we see that Kaufman's estimate is also not sharp for $s \leq t/2 + \kappa(\epsilon,t)$. Theorem \ref{mainKaufman} provides the first improvement over Kaufman's bound for arbitrary values of $s<t$. This is new for example in the case $t=1$, $s=3/4$. A variant of Theorem \ref{mainKaufman} for packing dimension, in the case $t = 1$, was proved in \cite{Orponen20}.

While Theorem \ref{mainKaufman} has been stated for $t \in (s,2]$, it is only news for $t \in (s,1 + \epsilon(s)]$. For larger "$t$", the following estimate due to K. Falconer \cite{MR673510} (case $s = 1$) and Y. Peres-W. Schlag \cite{MR1749437}, which can be obtained by a Fourier analytic method, is stronger than Theorem \ref{mainKaufman}:
\begin{equation}\label{falconer} \Hd \{e \in S^{1} : \Hd \pi_{e}(K) < s\} \leq \max\{0,1 + s - t\}. \end{equation}
Theorem \ref{mainKaufman} also provides a new improvement over the estimate \eqref{falconer} for $s\in (0,1)$ and $t\in [1,1+\epsilon(s)]$.

\subsection{Connections with other problems in geometric measure theory}

Beyond their intrinsic interest, a further motivation for the results in this paper lies in their application to other problems in geometric measure theory. Even though the "$\epsilon$" from our results is not explicit, and in any case would be very small, the second author and H.~Wang \cite{ShmerkingWang21} used a bootstrapping mechanism based on Theorem \ref{t:mainTechnical} to obtain new explicit estimates in the Falconer distance set problem \cite[Theorems 1.1--1.4]{ShmerkingWang21} and the dimension of radial projections \cite[Theorems 1.6 and 1.7]{ShmerkingWang21}. In particular, they obtained a full resolution of the Falconer distance set conjecture for Ahlfors-regular sets (and more generally, sets of equal Hausdorff and packing dimension). While the proofs in \cite{ShmerkingWang21} involve several old and new ideas, Theorem \ref{t:mainTechnical} is a crucial ingredient in all the proofs.

\subsection{Proof strategy}

We discuss some of the ideas involved in the proof of Theorem \ref{t:mainTechnical}. As mentioned already, it follows by combining results of Katz-Tao \cite{MR1856956} and Bourgain \cite{Bourgain03} that $s$-Furstenberg sets have Hausdorff dimension at least $1+\epsilon$ for $s$ "very close" to $\tfrac{1}{2}$. This assumption on $s$ is essential for the approach, as it forces a rather rigid ``product structure'' for the Furstenberg set (after projective transformation), that enables the application of Bourgain's discretized sum-product theorem.

Our proof also ultimately relies on Bourgain's theorem, but it follows a rather different path. The starting point is the main technical result of \cite{Orponen20} (it is there that Bourgain's projection theorem gets used). In the setting of Theorem \ref{t:mainTechnical}, it states that if $s \in (0,1)$, $t \in (s,2]$, and additionally $\mathcal{P}$ is \emph{$t$-regular} (in a fairly weak sense, see Definition \ref{def:regularity}), then either
\begin{equation} \label{eq:alternative}
|\mathcal{T}| \ge \delta^{-2s-\epsilon} \quad\text{or}\quad |\mathcal{T}|_{\delta^{1/2}} \ge \delta^{-s-\epsilon}.
\end{equation}
Here $|\mathcal{T}|_{\delta^{1/2}}$ denotes the number of dyadic $\delta^{1/2}$-tubes required to cover $\mathcal{T}$. Thus, if $P$ is $t$-regular, then one gets a gain in the size of $\mathcal{T}$ on (at least) one of the scales $\delta^{1/2},\delta$.

It turns out that the alternative \eqref{eq:alternative} is not necessary: under the $t$-regularity assumption on $\mathcal{P}$, one always gets $|\mathcal{T}| \ge \delta^{-2s-\epsilon}$, and thus Theorem \ref{t:mainTechnical} is established for $t$-regular sets $\mathcal{P}$. For the details, see Theorem \ref{t:improvedIncidence}. The reduction from Theorem \ref{t:improvedIncidence} to \eqref{eq:alternative} is based on a new ``induction on scales'' scheme for incidence counting, see Proposition \ref{p:induction-on-scales}, which is crucial also in a later part of the argument. Roughly speaking, under certain assumptions, Proposition \ref{p:induction-on-scales} bounds incidences at scale $\delta$ in terms of incidences at the coarser scales $\Delta$ and $\delta/\Delta$, where $\Delta \in (\delta,1]$ is a free parameter. Moreover, the families of squares and tubes arising at the coarser scales are related in a natural way to the original families.

A technical hurdle in the proof of Theorem \ref{t:improvedIncidence} is that \eqref{eq:alternative} was only stated in \cite{Orponen20} in the case $t = 1$: the general case $t \in (s,2]$ is fairly similar, but requires sufficiently many changes to merit writing up in Appendix \ref{appA} (the proof given in the appendix is also streamlined compared to \cite{Orponen20}).

A general $(\delta,t,\delta^{-\epsilon})$-set $\mathcal{P}$ can fail to be $t$-regular, and hence additional arguments are needed to complete the proof of Theorem \ref{t:mainTechnical}. A key innovation is a multi-scale decomposition of an arbitrary $(\delta,t,\delta^{-\epsilon})$-set $\mathcal{P}$ that enables us to exploit the gain for $t$-regular sets, see Proposition \ref{p:multiscaledecomp-kaufman}. This proposition provides a sequence of scales $\delta=\Delta_n < \Delta_{n-1}< \ldots < \Delta_0 = 1$ (depending on $\mathcal{P}$) such that, roughly speaking, one of the following three alternatives holds for each $j\in \{1,\ldots,n\}$:
\begin{itemize}
  \item $\mathcal{P}$ is $s$-dimensional between scales $\Delta_{j}$ and $\Delta_{j - 1}$.
  \item $\mathcal{P}$ is $t_j$-regular between scales $\Delta_{j}$ and $\Delta_{j - 1}$ for some $t_j \ge s$.
   \item (The ``bad'' case.) No information on $\mathcal{P}$ between scales $\Delta_{j}$ and $\Delta_{j - 1}$ is available.
\end{itemize}

The bad scales form a "negligible" proportion of all scales, so we ignore them here. If the first alternative occurs, we use elementary bounds (see Corollary \ref{prop5}). No gain is achieved, but also there is no loss. If the second alternative occurs for some $t_j > s + \eta$, then the special case of $t_j$-regular sets, discussed above, yields an $\epsilon$-gain (depending on $\eta$). The fact that $\mathcal{P}$ is a $(\delta,t,\delta^{-\epsilon})$-set is finally used to show that a positive proportion of the indices $j \in \{1,\ldots,n\}$ satisfy the second alternative, with $t_j> s+\eta$ for some $\eta>0$ depending only on $s$ and $t$.

The estimates from different scales are eventually combined via the ``induction on scales'' Proposition \ref{p:induction-on-scales}. The details are contained in Proposition \ref{prop3}. All these ingredients are put together in Section \ref{s:proof-main-thm}, where the proof of Theorem \ref{t:mainTechnical} is concluded.

The multiscale decomposition is inspired by similar ones in \cite{KeletiShmerkin19} and, especially, \cite{Shmerkin20}, although the details differ. The main difference to these papers is not in the multiscale decomposition itself, but rather in the way it is applied. The results of \cite{KeletiShmerkin19, Shmerkin20} concern projections, and the information from different scales is put together by means of entropy. The use of entropy seems challenging in the context of incidence counting, as in Theorem \ref{t:mainTechnical}. Thus, Proposition \ref{p:induction-on-scales} can be seen as a substitute for the multiscale entropy formulas that are a key element in many recent papers in the area.

Finally, we mention that induction-on-scales arguments have certainly been used before in problems involving incidence counting (see L. Guth's proof of the multilinear Kakeya inequality \cite{Guth15} for a very clean example). One novelty in this paper is that the scales in the inductive process are not arbitrary, but are chosen carefully in terms of the geometry of the set under consideration. A second novelty is that the "dimensions" $s,t$ of the families $\mathcal{T},\mathcal{P}$ are baked into the induction in a way which we have not seen before.

We close the introduction by a summary of the structure of the paper. Section \ref{s:prelim} contains preliminaries on point-line duality, and "classical" incidence bounds that do not require sum-product theory. Section \ref{s:discretisation} contains reductions of our main results, Theorems \ref{mainFurstenberg} and \ref{mainKaufman}, to their (common) $\delta$-discretised counterpart, Theorem \ref{t:mainTechnical}. Section \ref{s:thickTubeCover} consists of an auxiliary result, Proposition \ref{prop2}, whose motivation is thoroughly discussed at the head of Section \ref{s:thickTubeCover}. Section \ref{s:induction-on-scales} contains a key technical result, Proposition \ref{p:induction-on-scales}, which is needed to recombine estimates from various "scale blocks" $[\Delta_{j},\Delta_{j - 1}]$, as explained above. Section \ref{s:improvedIncidence} contains a special case of Theorem \ref{t:mainTechnical} where the set $\mathcal{P}$ is assumed to be $t$-regular. To be precise, Section \ref{s:improvedIncidence} only contains (the new) part of the proof, while a bulk of the work, heavily based on \cite{Orponen20}, is postponed to Appendix \ref{appA}.

Section \ref{s:combiningEstimates} contains Proposition \ref{prop3}, which is an application Proposition \ref{p:induction-on-scales}. Both Proposition \ref{p:induction-on-scales} and Proposition \ref{prop3} have the flair of "combining estimates from different scales", but Proposition \ref{prop3} does this in a way directly applicable in the proof of Theorem \ref{t:mainTechnical}. The hypotheses of Proposition \ref{prop3} accommodate data about $\mathcal{P}$ "looking $s$-dimensional" or "looking $t_{j}$-regular" between the scales $\Delta_{j - 1}$ and $\Delta_{j}$. Proposition \ref{prop3} can be applied to any $(\delta,t)$-set with such a structure: one has \emph{a priori} information that $\mathcal{P}$ "looks $s$-dimensional" or "looks $t_{j}$-regular" between the scales $\Delta_{j - 1}$ and $\Delta_{j}$. In Section \ref{s:multiScaleDecomposition}, the main task is to verify that a (rather) general $(\delta,t)$-set has such a structure if the lengths of the blocks $[\Delta_{j},\Delta_{j - 1}]$ are chosen appropriately. This idea of "finding good scale block decompositions for $(\delta,t)$-sets" stems from previous work of the second author \cite{KeletiShmerkin19,Shmerkin20}. Finally, Section \ref{s:9} puts all the pieces together to prove Theorem \ref{t:mainTechnical}.

 \subsection*{Acknowledgements} We are grateful to the reviewers for reading the manuscript carefully, and for making a large number of helpful suggestions.

\section{Preliminaries and elementary incidence estimates}\label{s:prelim}

\subsection{Notation}

We adopt some standard notational conventions. If $A,B>0$, we use the notation $A\lesssim B$ to mean $A\le C B$ for some universal $C>0$. Likewise, $A\sim B$ is a shortcut for $A\lesssim B\lesssim A$. Sometimes the implicit constant $C$ will be allowed to depend on certain parameters; these will be denoted by subscripts. Occasionally we will use the notation $A=O(B)$ for $A\lesssim B$ (and likewise with subscripts).

We will sometimes use the notation $A\lessapprox B$, $A\approx B$. The specific meaning will be specified each time, but it will generally be used to ``hide'' slowly growing functions of a small scale $\delta$, such as $\log(1/\delta)$, or $\delta^{-\epsilon}$.

Given a dyadic number $\delta\in 2^{-\mathbb{N}}$, we denote the family of half-open dyadic sub-cubes of $\R^d$ of side-length $\delta$ by $\mathcal{D}_{\delta}(\R^d)$. More generally, if $A\subset\R^d$, then we denote $\mathcal{D}_{\delta}(A) = \{ p\in\mathcal{D}_{\delta}(\R^d): p\cap A\neq\emptyset\}$. We will very often take $A=[0,1)^2$; in this case we simply write $\mathcal{D}_{\delta} = \mathcal{D}_{\delta}([0,1)^2)$. We also write $|A|_{\delta} := |\mathcal{D}_{\delta}(A)|$.

We finally extend the previous notations from subsets of $\R^{d}$ to sub-families of $\mathcal{D}_{\delta}(\R^{d})$, where $\delta \in 2^{-\N}$. For $\mathcal{P} \subset \mathcal{D}_{\delta}(\R^{d})$ and $\Delta \in 2^{-\N}$ with $\Delta \geq \delta$, we write $\mathcal{D}_{\Delta}(\mathcal{P}) := \mathcal{D}_{\Delta}(P)$, where $P := \cup \mathcal{P}$ is the union of the elements in $\mathcal{P}$. With the same notation, we also write $|\mathcal{P}|_{\Delta} := |P|_{\Delta}$.

\subsection{$(\delta,s)$-sets}

Even though the main results of this paper concern "infinitesimal" quantities like Hausdorff dimension, the proofs will be effective. Most of the technical and auxiliary results below will be phrased in terms of the following notion of $(\delta,s,C)$-sets, which are a kind of discretization of $s$-dimensional sets at scale $\delta$.

\begin{definition}[$(\delta,s,C)$-set]\label{def1} Let $P \subset \R^{d}$ be a bounded nonempty set, $d \geq 1$. Let $\delta > 0$ be a dyadic number, and let $0 \leq s \leq d$ and $C > 0$. We say that $P$ is a \emph{$(\delta,s,C)$-set} if
\begin{equation}\label{form53} |P \cap Q|_{\delta} \leq C \cdot |P|_{\delta} \cdot r^{s}, \qquad Q \in \mathcal{D}_{r}(\R^{d}), \, \delta \leq r \leq 1. \end{equation}
\end{definition}

We emphasize that this definition differs from related ones in e.g. \cite{MR1856956, Orponen20}, notably in that they have $\delta^{-s}$ in the right-hand side of \eqref{form53} in place of $|P|_{\delta}$. The ambient dimension $d$ will be clear from context; in this article, $d$ will always be either $1$ or $2$.

\begin{definition}[Families of $\delta$-cubes]
We extend the definition of $(\delta,s,C)$-sets to subsets of $\mathcal{D}_{\delta}(\R^{d})$: a finite family $\mathcal{P} \subset \mathcal{D}_{\delta}$ is called a $(\delta,s,C)$-set if $P :=\cup \mathcal{P}$ satisfies \eqref{form53}. \end{definition}

When the constant $C$ is not too important or assumed to be small, we will drop it from the notation and speak of $(\delta,s)$-sets.

\begin{remark} We will often use the following (easy) observations without further remark:
\begin{itemize}
\item If $P$ is a $(\delta, s, C)$-set and $P'\subset P$ has $|P'|_\delta \ge K^{-1} |P|_{\delta}$, then $P'$ is a $(\delta,s, CK)$-set.
\item If $P \subset \R^{d}$ is a $(\delta,s,C)$-set, we have the lower bound $|P|_{\delta} \geq C^{-1} \cdot \delta^{-s}$. This follows from applying \eqref{form53} to $Q\in\mathcal{D}_{\delta}(\R^d)$ such that $P\cap Q\neq\emptyset$.
\end{itemize}
\end{remark}

The following proposition is often useful for finding $(\delta,s,C)$-sets. Let $\mathcal{H}_{\infty}^s$ denote Hausdorff content, that is,
\[
\mathcal{H}_{\infty}^s(B) = \inf \left\{ \sum_i \diam(B_i)^s: B\subset \bigcup_i B_i \right\}.
\]

\begin{proposition}\label{deltasSet} Let $\delta \in 2^{-\N}$, and let $B \subset [-2,2]^{d}$ be a set with $\calH^{s}_{\infty}(B) =: \kappa > 0$. Then, there exists a $\delta$-separated $(\delta,s,C/\kappa)$-set $P \subset B$, where $C \geq 1$ is an absolute constant. Moreover, one can choose $P$ so that $|P| \leq \delta^{-s}$. \end{proposition}
The details can be found in \cite[Lemma 3.13]{FasslerOrponen14}.

\begin{remark} We remark that a $(\delta,s,C)$-set $P \subset \R^{d}$ need not be a finite set, so one may not replace the left hand side of the defining inequality \eqref{form53} by the cardinality of $P \cap Q$. In fact, our sets will often be the collections, or unions, of squares in $\mathcal{D}_{\delta}$. However, if $P \subset \R^{d}$ is a $(\delta,s,C)$-set then every, equivalently any, maximal $\delta$-separated subset $P_{\delta} \subset P$ satisfies the cardinality estimate
\begin{equation}\label{form55} |P_{\delta} \cap B(x,r)| \lesssim_d C \cdot |P_{\delta}| \cdot r^{s}, \qquad x \in \R^{d}, \, \delta \leq r \leq 1. \end{equation}
Indeed, it is easy to check that if $x \in \R^{d}$ and $\delta \leq r \leq 1$, then there exists a dyadic square "$Q$" of side-length $\sim r$ with the property $|P_{\delta} \cap B(x,r)| \sim_d |P \cap Q|_{\delta}$, and this implies \eqref{form55} when combined with the $(\delta,s,C)$-property of $P$. The converse implication also holds and is easy to check, but we will not need it, so we omit the details. \end{remark}

\begin{remark} \label{r:delta-set-via-measure}
Let $P\subset \R^{d}$ be a $(\delta,s,C)$-set. If we define
\[
\mu = |P|_\delta^{-1} \delta^{-d} \sum_{Q\in\mathcal{D}_{\delta}(P)} \mathcal{L}^{d}|_Q,
\]
where $\mathcal{L}^d|_Q$ is the restriction of Lebesgue measure to the cube $Q$, then it is easy to check that
\begin{equation} \label{eq:frostman}
\mu(B(x,r)) \lesssim C r^s
\end{equation}
for all $x\in\R^d$ and $r\in (0,1]$. For $r\in [\delta,1]$ this follows from the fact that $P$ is a $(\delta,s,C)$-set, and for $r<\delta$ an even better estimate follows from the fact that inside $\delta$-cubes $\mu$ is a multiple of Lebesgue measure. Conversely, if $P\subset\R^d$ is bounded and $\mu$ defined as above satisfies \eqref{eq:frostman}, then $P$ is a $(\delta,s,O(C))$-set.
\end{remark}

A consequence of \eqref{eq:frostman} is that if $P \subset \R^{d}$ is a $(\delta,s,C)$-set, and $P_{\delta}=\cup\mathcal{D}_{\delta}(P)$, then $\mathcal{H}_\infty^s(P_{\delta})\gtrsim C^{-1}$, since for each cover $(B_i)_i$ of $P_{\delta}$ we can apply  \eqref{eq:frostman} to a ball of radius $\diam(B_i)$ containing $B_i$. Therefore, by Proposition \ref{deltasSet}, the set $P_{\delta}$ contains a $\delta$-separated $(\delta,s,O(C))$-set $P'_{\delta}$ with $|P'_{\delta}| \leq \delta^{-s}$. By picking a point of $P$ from each $Q \in \mathcal{D}_{\delta}(P'_{\delta})$, we arrive at the following lemma.
\begin{lemma}\label{lem:thin-delta-subset}
Let $P\subset [-2,2]^{d}$ be a $(\delta,s,C)$-set. Then $P$ contains a $\delta$-separated $(\delta,s,O_d(C))$-subset $P'$ with $|P'| \leq \delta^{-s}$.
\end{lemma}

\subsection{Duality and $(\delta,s)$-sets of lines}

In addition to $(\delta,s)$-sets of points, we also wish to talk about $(\delta,s)$-sets of lines. These will be defined as the images of $(\delta,s)$-sets of points under the following "duality" map:

\begin{definition}[Dual sets] For $(a,b) \in \R^{2}$, define the \emph{dual line}
\begin{displaymath} \mathbf{D}(a,b) := \{(x,y)\in R^{2} : y = ax + b\} . \end{displaymath}
Then $\mathbf{D}$ is a one-to-one map between the points in $\R^{2}$, and the \emph{non-vertical} lines in $\R^{2}$. If $P \subset \R^{2}$ is a set, we write
\begin{displaymath} \mathbf{D}(P) := \{\mathbf{D}(a,b) : (a,b) \in P\}. \end{displaymath}
Thus $\mathbf{D}(P)$ is defined as a collection of lines, but we will often abuse notation by identifying $\mathbf{D}(P)$ with its union. If $P \subset \R^{2}$ is a bounded set, we say that a collection of lines $\mathbf{D}(P)$ is a $(\delta,s,C)$-set if $P$ is a $(\delta,s,C)$-set in the sense of Definition \ref{def1}.
\end{definition}

If $\mathcal{L}$ is a collection of non-vertical lines, then $\mathcal{L} = \mathbf{D}(P)$ for a unique set $P \subset \R^{2}$, and hence it is well-defined to ask whether $\mathcal{L}$ is a $(\delta,s,C)$-set of lines. The duality map $\mathbf{D}$ also allows us to define
\begin{displaymath} |\mathcal{L}|_{\delta} := |\mathbf{D}(P)|_{\delta} := |P|_{\delta}. \end{displaymath}
\begin{definition}[Slope set]\label{def2} The \emph{slope} of a non-vertical line $\ell = \mathbf{D}(a,b)$ is defined to be the number $a \in \R$; we will write $\sigma(\ell) := a$. More generally, the \emph{slope set} of a line family $\mathcal{L} = \mathbf{D}(P)$ is defined as $\sigma(\mathcal{L}) := \pi_{1}(P)$, where $\pi_{1}(a,b) = a$. \end{definition}

\begin{definition}[Dyadic $\delta$-tubes]\label{def:dyadicTubes} Let $\delta \in 2^{-\N}$. A \emph{dyadic $\delta$-tube} is a set of the form $T = \cup \mathbf{D}(p)$, where $p \in \mathcal{D}_{\delta}([-1,1)\times \R)$. In this context, we abbreviate $\mathbf{D}(p) := \cup \mathbf{D}(p)$. The collection of all dyadic $\delta$-tubes is denoted
\[
\mathcal{T}^{\delta} := \{\mathbf{D}(p) : p \in \mathcal{D}_{\delta}([-1,1)\times \R)\}.
\]
A finite collection of dyadic $\delta$-tubes $\{\mathbf{D}(p)\}_{p \in \mathcal{P}}$ is called a $(\delta,s,C)$-set if $\mathcal{P}$ is a $(\delta,s,C)$-set in the sense of Definition \ref{def1}.

We remark that a dyadic $\delta$-tube is not a tube in the usual sense (a $\delta$-neighbourhood of some line). However, $T':=\mathbf{D}([a,a+\delta]\times [b,b+\delta])\cap [-1,1)^2$ satisfies
\begin{align*}
\{(x,y)\in [-1,1)^2: &|y-(ax+b)|\in [0,\delta] \}\subset T'\\
&\subset  \{ (x,y)\in [-1,1)^2: |y-(ax+b)|\in [-\delta,2\delta]\}.
\end{align*}
Since for the most part we will only care about what happens inside $[-1,1)^2$, it is safe to think of a dyadic $\delta$-tube as comparable to a tube of width $\sim\delta$.
We also note that if $p,Q$ are dyadic cubes (of possibly different sizes), then $p\subset Q$ if and only if $\mathbf{D}(p)\subset \mathbf{D}(Q)$. However, unlike dyadic cubes, different tubes in $\mathcal{T}^{\delta}$ may intersect.

If $T = \mathbf{D}(p)$ is a dyadic $\delta$-tube, we define the slope $\sigma(T)$ as the left endpoint of the interval $\pi_{1}(p) \in \mathcal{D}_{\delta}([-1,1))$. Thus $\sigma(T) \in (\delta \cdot \Z) \cap [-1,1)$. (This notation is slightly inconsistent with Definition \ref{def2}, but we will make sure that this will not cause confusion.) If $\mathcal{T}$ is a collection of dyadic $\delta$-tubes, we often write $\sigma(\mathcal{T}) := \{\sigma(T) : T \in \mathcal{T}\} \subset \delta \cdot \Z$. \end{definition}

The choice of the rectangle $[-1,1) \times \R$ is somewhat arbitrary: the main purpose is to avoid dealing with "nearly vertical" tubes. We have now defined all the concepts appearing in our $\delta$-discretised main result, Theorem \ref{t:mainTechnical}. The precise choices of parameters for dyadic $\delta$-squares and $\delta$-tubes play no role for the validity of Theorem \ref{t:mainTechnical}, even if such normalisations are convenient in the proofs. This fact will be formalised in Theorem \ref{t:mainTechnicalGeneral}.

We then make some further notational conventions. From now on, sets of ``points'' will typically be subsets of $[0,1)^2$ while, as explained, dyadic $\delta$-tubes will be of the form $\mathbf{D}(p)$ with $p \in \mathcal{D}_{\delta}([-1,1) \times \R)$.  As above, squares in $\mathcal{D}_{\delta}(\R^{2})$ are often denoted by the letter "$p$", and we remind the reader that $\mathcal{D}_{\delta} := \mathcal{D}_{\delta}([0,1)^{2})$. Our slightly non-standard use of "$p$" leaves the notation "$Q$" applicable for squares in $\mathcal{D}_{\Delta}(\R^{2})$, with $\delta < \Delta \leq 1$.

As stated, $\mathcal{T}^\delta$ will denote the family of \emph{all} dyadic $\delta$-tubes. In contrast, "specific" families of $\delta$-tubes are denoted $\mathcal{T},\mathcal{T}_{\delta},\mathcal{T}_{\Delta}$. The notation $\mathcal{T}(p)$ will always denote a collection of dyadic tubes in $\mathcal{T}^{\delta}$ intersecting the dyadic square $p\in\mathcal{D}_{\delta}(\R^2)$. When dealing with two dyadic scales $\delta<\Delta$, as will often be the case, we will denote ``thin'' $\delta$-tubes by $T, T_j$, and ``thick'' $\Delta$-tubes by $\mathbf{T},\mathbf{T}_j$.

In general, a collection "$\mathcal{T}$" of dyadic $\delta$-tubes can have very different separation properties from that of its slope set "$\sigma(\mathcal{T})$", but if all the tubes in $\mathcal{T}$ intersect a common square $p \in \mathcal{D}_{\delta}$, things are different:
\begin{lemma}\label{tubesSlopes} Let $p \in \mathcal{D}_{\delta}$, and let $\mathcal{T}(p)\subset \mathcal{T}^{\delta}$ be a collection of dyadic $\delta$-tubes, all of which intersect $p$. Then the map $T \mapsto \sigma(T)$ is at most $10$-to-$1$ on $\mathcal{T}(p)$. In particular,
\begin{equation}\label{form63} |\{T \in \mathcal{T}(p) : \sigma(T) \in I\}| \leq 10|\sigma(\mathcal{T}(p)) \cap I|, \qquad I \subset \R. \end{equation}
\end{lemma}

\begin{proof} Let $a \in \sigma(\mathcal{T}(p))$ be a fixed slope, and assume that $T_{j} = \mathbf{D}([a,a + \delta) \times [b_{j},b_{j} + \delta)) \in \mathcal{T}(p)$ for $j \in \{1,2\}$, so in particular there exist points $(x_{1},y_{1}) \in T_{1} \cap p$ and $(x_{2},y_{2}) \in T_{2} \cap p$. This implies that there exist numbers $a',a'',b_{1}',b_{2}''$ with $\max\{|a' - a|,|a'' - a|,|b_{1}' - b_{1}|,|b_{2}'' - b_{2}|\} < \delta$ such that
\begin{displaymath} y_{1} = a' x_{1} + b_{1}' \quad \text{and} \quad y_{2} = a''x_{2} + b_{2}''. \end{displaymath}
Using $|y_{1} - y_{2}|,|x_{1} - x_{2}| \leq \delta$, and $|a'|, |x_{2}| \leq 1$ (since $p \subset [0,1)^{2}$, and $T \in \mathcal{T}^{\delta}$), it follows from the triangle inequality that
\begin{displaymath} |b_{1} - b_{2}| \leq |b_{1}' - b_{2}''| + 2\delta  = |(y_{1} - a'x_{1}) - (y_{2} - a''x_{2})| + 2\delta \leq 5\delta. \end{displaymath}
This implies the claim. \end{proof}
\begin{cor}\label{cor1} Let $p_{0} \in \mathcal{D}_{\delta}$, and let $\mathcal{T}\subset\mathcal{T}^{\delta}$ be a collection of dyadic $\delta$-tubes, all of which intersect $p_{0}$. If the slope set $\sigma(\mathcal{T})$ is a $(\delta,s,C)$-set for some $s \geq 0$ and $C > 0$, then also $\mathcal{T}$ is a $(\delta,s,10C)$-set. Conversely, if $\mathcal{T}$ is a $(\delta,s,C)$-set, then $\sigma(\mathcal{T})$ is a $(\delta,s,C')$-set for some $C' \sim C$.
\end{cor}

\begin{proof} Write $\mathcal{T} = \{\mathbf{D}(p)\}_{p \in \mathcal{P}}$, where $\mathcal{P} \subset \mathcal{D}_{\delta}$. Fix a square $Q_{r} \in \mathcal{D}_{r}$ with $\delta \leq r \leq 1$, and note that whenever $p \subset Q_{r}$, the slope $\sigma(T)$ lies on the interval $I := \pi_{1}(Q_{r}) \in \mathcal{D}_{r}(\R)$ of length $r \geq \delta$. Therefore,
\begin{displaymath} |\{p \in \mathcal{P} : p \subset Q_{r}\}| \leq |\{T \in \mathcal{T} : \sigma(T) \in I\}| \stackrel{\eqref{form63}}{\leq} 10|\sigma(\mathcal{T}) \cap I| \leq 10C \cdot |\sigma(\mathcal{T})| \cdot r^{s}.  \end{displaymath}
Since clearly $|\sigma(\mathcal{T})| \leq |\mathcal{P}| = |\mathcal{T}|$, we have now shown that $\mathcal{T}$ is a $(\delta,s,10C)$-set.

For the converse implication, we first observe that $|\mathcal{T}| \leq 10|\sigma(\mathcal{T})|$ by \eqref{form63} applied with $I = \R$. So, it suffices to show that if $I \in \mathcal{D}_{r}$, $\delta \leq r \leq 1$, then $|\sigma(\mathcal{T}) \cap I| \lesssim C \cdot |\mathcal{T}| \cdot r^{s}$. To this end, fix $I \in \mathcal{D}_{r}$ as above, and write $\{a_{1},\ldots,a_{M}\} := \sigma(\mathcal{T}) \cap I$. Since $a_{j} \in \sigma(\mathcal{T})$, there is at least one corresponding number $b_{j} \in \delta \cdot \Z$ such that
\begin{displaymath} [a_{j},a_{j} + \delta) \times [b_{j},b_{j} + \delta) \in \mathcal{P}. \end{displaymath}
(If there are several, just pick one.) We claim that all such numbers "$b_{j}$" lie on a single interval $J \subset \R$ of length $|J| \sim |I|$, determined by $I$ and the square $p_{0}$. To see this, recall that $\mathbf{D}([a_{j},a_{j} + \delta) \times [b_{j},b_{j} + \delta))$ intersects $p_{0} \in \mathcal{D}_{\delta}$. Spelling out what this means, for every $1 \leq j \leq M$ there exist $(x_{j},y_{j}) \in p_{0}$, $a_{j}' \in [a_{j}a_{j} + \delta)$, and $b_{j}' \in [b_{j},b_{j} + \delta)$ satisfying $y_{j} = a_{j}'x_{j} + b_{j}'$. The numbers $a_{j}'$ here range in the $\delta$-neighbourhood of the interval $I$, and the numbers $x_{j},y_{j}$ lie at distance $\delta$ from $x_{0},y_{0}$, the coordinates of the left corner of $p_{0}$. Therefore, the numbers $b_{j}$ range in the $C\delta$-neighbourhood of $y_{0} - I \cdot x_{0}$, which is an interval of length $\lesssim |I|$. This establishes the claim.

Consequently, we may pick an interval $I' \in \mathcal{D}_{r}(\R)$ such that
\begin{displaymath} |\sigma(\mathcal{T}) \cap I| = M \lesssim |\{1 \leq j \leq M : [a_{j},a_{j} + \delta) \times [b_{j},b_{j} + \delta) \subset I \times I'\}| \leq C \cdot |\mathcal{T}| \cdot r^{s}. \end{displaymath}
This completes the proof. \end{proof}

\subsection{An elementary incidence bound}

We next record an elementary incidence estimate, which is reminiscent of Theorem \ref{t:mainTechnical}, but without the "$\epsilon$"-gain (the parallel is even clearer in Corollary \ref{prop5}). In this section, the notation $A \lessapprox_{\delta} B$ will mean that there exists an absolute constant $C \geq 1$ such that
\begin{displaymath} A \leq C \cdot \log \left(\tfrac{1}{\delta} \right)^{C} B. \end{displaymath}

\begin{proposition}\label{incidenceProp} Let $0 \leq s \leq t \leq 1$, and let $C_{P},C_{T} \geq 1$. Let $\mathcal{P} \subset \mathcal{D}_{\delta}$ be a $(\delta,t,C_{P})$-set. Assume that for every $p \in \mathcal{P}$ there exists a $(\delta,s,C_{T})$-family $\mathcal{T}(p) \subset \mathcal{T}^{\delta}$ of dyadic $\delta$-tubes with the property that $T \cap p \neq \emptyset$ for all $T \in \mathcal{T}(p)$, and $|\mathcal{T}(p)|=M$ for some $M\ge 1$.

Let $\mathcal{T} \subset \mathcal{T}^{\delta}$ be arbitrary, and define $\mathcal{I}(\mathcal{P},\mathcal{T}) := \{(p,T) \in \mathcal{P} \times \mathcal{T} : T \in \mathcal{T}(p)\}$. Then
\begin{equation}\label{incidenceBound} |\mathcal{I}(\mathcal{P},\mathcal{T})| \lessapprox_{\delta} \max\left\{ \sqrt{C_{P}C_{T}} \cdot (M\delta^{s})^{\theta/2} \cdot  |\mathcal{T}|^{1/2}|\mathcal{P}|, |\mathcal{T}| \right\}, \end{equation}
where $\theta = \theta(s,t) := (1 - t)/(1 - s) \in [0,1]$. (If $s = t = 1$, then $\theta(s,t) := 0$.)
\end{proposition}

\begin{proof} Using Cauchy-Schwarz, we first estimate as follows:
\begin{align*} |\mathcal{I}(\mathcal{P},\mathcal{T})| &= \sum_{T \in \mathcal{T}} |\{p \in \mathcal{P} : T \in \mathcal{T}(p)\}| \\
&\leq |\mathcal{T}|^{1/2} \left|  \{(T,P,P'): T\in\mathcal{T}(p)\cap \mathcal{T}(p')  \}\right|^{1/2}  \\
&\leq |\mathcal{T}|^{1/2} \Bigg( |\mathcal{I}(\mathcal{P},\mathcal{T})| + \sum_{p \neq p'} |\mathcal{T}(p) \cap \mathcal{T}(p')| \Bigg)^{1/2}.
\end{align*}
Assume first that $|\mathcal{I}(\mathcal{P},\mathcal{T})| \geq 2 \sum_{p \neq p'} |\mathcal{T}(p) \cap \mathcal{T}(p')|$. Then
\begin{equation}\label{form95} |\mathcal{I}(\mathcal{P},\mathcal{T})| \leq |\mathcal{T}|^{1/2} \sqrt{3/2} \cdot |\mathcal{I}(\mathcal{P},\mathcal{T})|^{1/2} \quad \Longrightarrow \quad |\mathcal{I}(\mathcal{P},\mathcal{T})| \leq \tfrac{3}{2}|\mathcal{T}|. \end{equation}
To make progress in the opposite case, we will interpolate between the following upper bounds for $|\mathcal{T}(p) \cap \mathcal{T}(p')|$:
\begin{equation}\label{form94} |\mathcal{T}(p) \cap \mathcal{T}(p')| \lesssim \min \left\{C_{T} \cdot M \cdot \left(\tfrac{\delta}{d(p,p')} \right)^{s}, \tfrac{1}{d(p,p')} \right\}, \end{equation}
where $d(p,p')$ stands for the distance of the midpoints of $p$ and $p'$. Both bounds in \eqref{form94} follow by observing that if $T \in \mathcal{T}(p) \cap \mathcal{T}(p')$, then in particular $T \cap p \neq \emptyset \neq T \cap p'$, which forces $\sigma(T)$ to lie on a certain interval $I \subset [-1,1)$ of length $|I| \lesssim \delta/d(p,p')$, and of course also in $\sigma(\mathcal{T}(p))$. Therefore,
\begin{displaymath} |\mathcal{T}(p) \cap \mathcal{T}(p')| \leq |\{T \in \mathcal{T}(p) : \sigma(T) \in I\}| \stackrel{\eqref{form63}}{\lesssim} |\sigma(\mathcal{T}(p)) \cap I|. \end{displaymath}
Now \eqref{form94} follows from  Corollary \ref{cor1}, and the fact that $\sigma(\mathcal{T}(p))$ is a $\delta$-separated $(\delta,s,O(C_{T}))$-set of cardinality $\le M$. Write $\theta := \theta(s,t) := (1 - t)/(1 - s) \in [0,1]$. (If $s = t = 1$, we set $\theta := 0$.) The parameter $\theta$ is chosen so that $t = s\theta+(1-\theta)$. Then \eqref{form94} and the inequality $\min\{a,b\} \leq a^{\theta}b^{1 - \theta}$ imply that
\begin{displaymath} |\mathcal{T}(p) \cap \mathcal{T}(p')| \lesssim (C_{T}M\delta^{s})^{\theta} \cdot d(p,p')^{-t}. \end{displaymath}
By the $(\delta,t,C_{P})$-hypothesis of $\mathcal{P}$, for fixed $p\in\mathcal{P}$ we have
\begin{displaymath}
\sum_{p'\neq p} d(p,p')^{-t} \lesssim \sum_{ \sqrt{2}\cdot\delta \le 2^{-j} \le \sqrt{2} } 2^{t j}|\{p'\in \mathcal{P}: d(p,p') \le 2^{-j}\}| \lessapprox_{\delta} C_{P}\cdot |\mathcal{P}|.
\end{displaymath}
We deduce that
\begin{displaymath} \sum_{p \neq p'} |\mathcal{T}(p) \cap \mathcal{T}(p')| \lesssim (C_{T}M\delta^{s})^{\theta} \sum_{p \neq p'} d(p,p')^{-t} \lessapprox_{\delta} C_{P}(C_{T}M\delta^{s})^{\theta} \cdot |\mathcal{P}|^{2}. \end{displaymath}
Therefore, in the case $|\mathcal{I}(\mathcal{P},\mathcal{T})| < 2\sum_{p \neq p'} |\mathcal{T}(p) \cap \mathcal{T}(p')|$, we obtain
\begin{displaymath} |\mathcal{I}(\mathcal{P},\mathcal{T})| \lessapprox_{\delta} C_{P}^{1/2}(C_{T}M\delta^{s})^{\theta/2} \cdot  |\mathcal{T}|^{1/2}|\mathcal{P}| \leq \sqrt{C_{P}C_{T}} \cdot (M\delta^{s})^{\theta/2} \cdot  |\mathcal{T}|^{1/2}|\mathcal{P}|. \end{displaymath}
Combining this estimate with \eqref{form95} completes the proof. \end{proof}

We record a corollary, which is the form we will use.
\begin{cor}\label{prop5} Let $0 \leq s \leq t \leq 1$, and let $C_{P},C_{T} \geq 1$. Let $\mathcal{P} \subset \mathcal{D}_{\delta}$ be a $(\delta,t,C_{P})$-set. Assume that for every $p \in \mathcal{P}$ there exists a $(\delta,s,C_{T})$-set $\mathcal{T}(p) \subset \mathcal{T}^{\delta}$ of dyadic $\delta$-tubes with the properties that $T \cap p \neq \emptyset$ for all $T \in \mathcal{T}(p)$, and $|\mathcal{T}(p)| \sim M$ for some $M \geq 1$. Then,
\begin{equation} \label{form61}
 |\mathcal{T}| \gtrapprox_{\delta} (C_{P}C_{T})^{-1} \cdot M\delta^{-s} \cdot  (M\delta^s)^{\tfrac{t-s}{1-s}}, \end{equation}
 where $\mathcal{T} = \bigcup_{p \in \mathcal{P}} \mathcal{T}(p)$. (If $s = t = 1$, we interpret $(t - s)/(1 - s) = 1$). \end{cor}

 \begin{proof} Write $\mathcal{I}(\mathcal{P},\mathcal{T}) = \{(p,\mathcal{T}) \in \mathcal{P} \times \mathcal{T} : T \in \mathcal{T}(p)\}$, as before. Under the current hypotheses, $|\mathcal{I}(\mathcal{P},\mathcal{T})| \sim M|\mathcal{P}|$, and on the other hand \eqref{incidenceBound} implies that
 \begin{equation}\label{form58} |\mathcal{I}(\mathcal{P},\mathcal{T})| \lessapprox_{\delta} \max\{\sqrt{C_{P}C_{T}} \cdot (M\delta^{s})^{\theta/2} \cdot  |\mathcal{T}|^{1/2}|\mathcal{P}|, |\mathcal{T}| \}, \end{equation}
 where $\theta = (1 - t)/(1 - s)$. If the second term dominates, then
 \begin{displaymath} |\mathcal{T}| \gtrapprox_\delta |\mathcal{I}(\mathcal{P},\mathcal{T})| \sim M|\mathcal{P}| \geq M \cdot C_{P}^{-1} \cdot \delta^{-t}. \end{displaymath}
If the first term on the right of \eqref{form58} dominates, then
\begin{align*}  |\mathcal{T}| & \gtrapprox_{\delta} |\mathcal{I}(\mathcal{P},\mathcal{T})|^{2} \cdot |\mathcal{P}|^{-2} \cdot \left(C_{P}C_{T} \cdot ( M\delta^s)^{\theta}\right)^{-1} \\
&\sim (C_{P}C_{T})^{-1} \cdot M \delta^{-s} \cdot  (M\delta^s)^{\tfrac{t-s}{1-s}}.
\end{align*}
A calculation using that $M\lesssim \delta^{-1}$ shows that the second lower bound is smaller.  \end{proof}

\section{Discretising the main results}\label{s:discretisation}

We start the section by formulating a superficially stronger version of Theorem \ref{t:mainTechnical}, which is more comfortable to apply, but can be easily reduced to Theorem \ref{t:mainTechnical}. Only for this  statement and proof, we set
\begin{displaymath}
\mathcal{T}^{\delta}_R := \{\mathbf{D}(p) : p \in \mathcal{D}_{\delta}([-R,R)\times\R)\}.
\end{displaymath}

\begin{thm}\label{t:mainTechnicalGeneral} Let $R \in 2^{\N}$. For $s \in (0,1)$ and $t\in (s,2]$, there exists $\epsilon(s,t) > 0$ such that the following holds for all small enough $\delta \in 2^{-\N}$, depending only on $s,t,R$. Let $\mathcal{P} \subset \mathcal{D}_{\delta}([-R,R)^{2})$ be a $(\delta,t,\delta^{-\epsilon})$-set, and let $\mathcal{T} \subset \mathcal{T}^{\delta}_R$. Assume that for every $p \in \mathcal{P}$, there exists a $(\delta,s,\delta^{-\epsilon})$-set $\mathcal{T}(p) \subset \mathcal{T}$ such that $\overline{T} \cap \bar{p} \neq \emptyset$ for all $T \in \mathcal{T}(p)$. Then $|\mathcal{T}| \geq \delta^{-2s - \epsilon}$. \end{thm}

\begin{proof} Let $\epsilon := \epsilon_{1}/2$, where $\epsilon_{1}(s,t) > 0$ is the constant of Theorem \ref{t:mainTechnical}. Let
\begin{displaymath}
 S(x,y) := (\tfrac{1}{2},\tfrac{1}{2}) + (\tfrac{x}{4R},\tfrac{y}{(4R)^2}), \quad (x,y)\in\R^2.
\end{displaymath}
The map $S$ evidently sends $[-R,R)^{2}$ inside $[\tfrac{1}{4},\tfrac{3}{4})^{2} \subset [0,1)^{2}$, but it also reduces the slopes of tubes $T \in \mathcal{T}^{\delta}_R$ by a factor of $4R$. Indeed, it is easy to check that if $\ell = \mathbf{D}(z)$ with $z \in [-R,R] \times \R$, then $\sigma(S(\ell)) \in [-1/4,1/4]$. One can also check that if $T \in \mathcal{T}^{\delta}_R$, then $S(\overline{T})$ can be covered by a family $\mathcal{S}(T) \subset \mathcal{T}^{\delta}$ of (standard) dyadic $\delta$-tubes with $|\mathcal{S}(T)| \sim 1$. For every $p \in \mathcal{D}_{\delta}([R,R)^{2})$, we also choose a collection $\mathcal{S}(p) \subset \mathcal{D}_{\delta}$ of (standard) dyadic $\delta$-squares such that $S(\bar{p}) \subset \cup \mathcal{S}(p)$, and $|\mathcal{S}(p)| \sim 1$.

Now the following facts need a little checking, which we leave to the reader:
\begin{itemize}
\item $\mathcal{P}' := \bigcup_{p \in \mathcal{P}} \mathcal{S}(p) \subset \mathcal{D}_{\delta}$ is a $(\delta,t,C\delta^{-\epsilon})$-set with $C \sim_{R} 1$.
\item $\mathcal{T}' := \bigcup_{T \in \mathcal{T}} \mathcal{S}(T) \subset \mathcal{T}^{\delta}$ satisfies $|\mathcal{T}'| \sim_{R} |\mathcal{T}|$.
\item The sets $\mathcal{S}(\mathcal{T}(p)) := \bigcup_{T \in \mathcal{T}(p)} \mathcal{S}(T)$ are $(\delta,s,C\delta^{-\epsilon})$-sets for $p \in \mathcal{P}$, with $C \sim_{R} 1$.
\end{itemize}
As we will see in a moment, the pair $(\mathcal{P}',\mathcal{T}')$ "almost" satisfies the hypotheses of Theorem \ref{t:mainTechnical}, and hence $|\mathcal{T}| \sim_{R} |\mathcal{T}'| \gtrsim \delta^{-2s - \epsilon_{1}}$. Recalling that $\epsilon = \epsilon_{1}/2$, this will conclude the proof.

The word "almost" still calls for an explanation. Recall that for all $p \in \mathcal{P}$ and $T \in \mathcal{T}(p)$, we are assuming that $\overline{T} \cap \bar{p} \neq \emptyset$. Consequently $S(\overline{T}) \cap S(\bar{p}) \neq \emptyset$. Since
\begin{displaymath}  S(\bar{p}) \subset \cup \mathcal{S}(p) \quad \text{and} \quad S(\overline{T}) \subset \cup \mathcal{S}(T) \subset \cup \mathcal{S}(\mathcal{T}(p)), \end{displaymath}
for each pair $(p,T) \in \mathcal{P} \times \mathcal{T}(p)$ we may choose a representative $(p',T') \in (\mathcal{S}(p),\mathcal{S}(\mathcal{T}(p)))$ such that $T' \cap p' \neq \emptyset$. Here both $p',T'$ depend on $p,T$, but there are only $\sim 1$ different choices of $(p',T')$ for each $(p,T)$. In particular, for $p \in \mathcal{P}$ fixed, we may choose $p'' \in \mathcal{S}(p)$ in such a way that $T' \cap p'' \neq \emptyset$ for $\sim |\mathcal{T}(p)|$ choices of $T' \in \mathcal{S}(\mathcal{T}(p))$.

Now we reduce $\mathcal{P}'$ to the subset $\mathcal{P}''$ of squares $p''$ obtained by the pigeonholing procedure above. Evidently $\mathcal{P''}$ remains a $(\delta,t,C\delta^{-\epsilon})$-set. Also, for every $p'' \in \mathcal{P}''$, there exists a $(\delta,s,C\delta^{-\epsilon})$-subset of $\mathcal{S}(\mathcal{T}(p)) \subset \mathcal{T}'$, all tubes in which intersect $p''$. Therefore $(\mathcal{P}'',\mathcal{T}')$ satisfies the assumptions of Theorem \ref{t:mainTechnical}, taking $\delta > 0$ so small that $C \leq \delta^{-\epsilon} = \delta^{-\epsilon_{1}/2}$. Now the proof of Theorem \ref{t:mainTechnicalGeneral} can be completed as discussed above. \end{proof}

We then proceed to show how to reduce the proofs of Theorems \ref{mainFurstenberg} and \ref{mainKaufman} to the statement above.
\subsection{Discretising Furstenberg sets} This section contains the proof of Theorem \ref{mainFurstenberg}.

We  formulate a dual version of Theorem \ref{t:mainTechnical}, which is more suited for the application to Furstenberg sets:
\begin{thm}\label{t:mainDual} For every $s \in (0,1)$ and $t \in (s,2]$, there exists $\epsilon = \epsilon(s,t) > 0$ such that the following holds for all small enough $\delta\in 2^{-\N}$ depending only on $s,t$. Let $\mathcal{T} \subset \mathcal{T}^{\delta}$ be a $(\delta,t,\delta^{-\epsilon})$-set of dyadic $\delta$-tubes. Assume that for every $T \in \mathcal{T}$, there exists a $(\delta,s,\delta^{-\epsilon})$-set $\mathcal{P}(T) \subset \mathcal{D}_{\delta}$ such that $T \cap p \neq \emptyset$ for all $p \in \mathcal{P}(T)$. Then $|\mathcal{P}| \geq \delta^{-2s - \epsilon}$, where $\mathcal{P} := \bigcup_{T \in \mathcal{T}} \mathcal{P}(T)$. \end{thm}

Theorem \ref{t:mainDual}  follows from Theorem \ref{t:mainTechnical}, or rather from the generalised version in Theorem \ref{t:mainTechnicalGeneral}, by swapping the roles of dyadic squares and dyadic tubes. More precisely, let $\mathcal{T}$ and $\{\mathcal{P}(T)\}_{T \in \mathcal{T}}$ be the families specified in Theorem \ref{t:mainDual}. We define the map "$\mathbf{D}^{\ast}$" on $\mathcal{T}^{\delta}$ as follows: if $T = \mathbf{D}(p) \in \mathcal{T}^{\delta}$, then
\begin{displaymath} \mathbf{D}^{\ast}(T) := \{(-a,b) : (a,b) \in p\}. \end{displaymath}
Now we let $\mathcal{P}^{\ast} := \{ \mathbf{D}^{*}(T): T\in \mathcal{T}\}$, and $\mathcal{T}^{\ast} := \cup \{\mathbf{D}(\mathcal{P}(T)) : T \in \mathcal{T}\}$. First, note that $\mathcal{P}^{\ast}$ is a collection of squares of side-length $\delta$. These squares are ``almost dyadic'', except that they are not half-open in the standard way. Without changing notation, we rearrange the boundaries of the squares in $\mathcal{P}^{\ast}$ so that they become dyadic. Note that $\mathcal{P}^{\ast}\subset [-1,1]\times\R$ by definition. In fact this can be sharpened: since every tube $T = \mathbf{D}(p) \in \mathcal{T}$ intersects $[0,1)^{2}$ by assumption, one has $p \subset [-1,1) \times [-1,2)$, and consequently $\cup \mathcal{P}^{\ast} \subset [-2,2]^{2}$. Further, $\cup \mathcal{P}^{\ast}$ is a $(\delta,t,\delta^{-\epsilon})$-set, since $\mathcal{T}$ is.

Second, notice that $\mathcal{T}^{\ast}$ contains a $(\delta,s,\delta^{-\epsilon})$-set $\mathcal{T}^{\ast}(p^{\ast}) := \mathbf{D}(\mathcal{P}(T))$ for each $p^{\ast} := \mathbf{D}^{\ast}(T) \in \mathcal{P}^{\ast}$. We will check in Lemma \ref{lemma2} below that the $T^{\ast} \cap p^{\ast} \neq \emptyset$ for all $T^{\ast} \in \mathcal{T}^{\ast}(p^{\ast})$. It follows that the pair $(\mathcal{P}^{\ast},\mathcal{T}^{\ast})$ satisfies the hypotheses of Theorem \ref{t:mainTechnicalGeneral} with $R = 2$, and the conclusion of Theorem \ref{t:mainDual} follows.

It remains to check that every square $p^{\ast} \in \mathcal{P}^{\ast}$ satisfies $T^{\ast} \cap p^{\ast} \neq \emptyset$ for all $T^{\ast} \in \mathcal{T}^{\ast}(p^{\ast})$. This follows from the next lemma, which explains our need to introduce the map "$\mathbf{D}^{\ast}$":

\begin{lemma}\label{lemma2} Let $p \in \mathcal{D}_{\delta}$ and $T \in \mathcal{T}^{\delta}$ with $T \cap p \neq \emptyset$. Then $\mathbf{D}(p) \cap \mathbf{D}^{\ast}(T) \neq \emptyset$.  \end{lemma}

\begin{proof} Since $T \in \mathcal{T}^{\delta}$, we may write $T = \mathbf{D}(p_{T})$ for some $p_{T} \in \mathcal{D}_{\delta}([0,1)\times [-1,2))$. Let $(c,d) \in T \cap p$. Since $(c,d) \in T$, we have $d = ac + b$ for some $(a,b) \in p_{T}$. Equivalently $b = c(-a) + d$, which implies that $(-a,b) \in \mathbf{D}(p)$. Also, $(-a,b) \in \mathbf{D}^{\ast}(T)$ by definition, so $(-a,b) \in \mathbf{D}(p) \cap \mathbf{D}^{\ast}(T)$. \end{proof}

To obtain the lower bound for the Hausdorff dimension of Furstenberg sets claimed in Theorem \ref{mainFurstenberg}, we appeal to the following auxiliary result, which is Lemma 3.3 from \cite{HSY22}.
\begin{lemma}\label{HSYLemma} Assume that every discretised $(\delta,C,s,t)$-Furstenberg set has Lebesgue measure $\gtrsim_{C} \delta^{2 - u}$, for some $u \geq 0$. Then every $(s,t)$-Furstenberg set has Hausdorff dimension at least $u$.  \end{lemma}
Let us explain the terminology. In the language of \cite{HSY22}, an $(s,t)$-Furstenberg set is a set $K \subset \R^{2}$ with the property that there exists a line set $\mathcal{L} \subset \R^{2}$ with $\mathcal{H}^{t}(\mathcal{L}) > 0$ such that $\mathcal{H}^{s}(K \cap \ell) > 0$ for all $\ell \in \mathcal{L}$. Evidently, to prove Theorem \ref{mainFurstenberg}, it suffices to show that if $0 < s < t \leq 2$, then every $(s,t)$-Furstenberg set $K \subset \R^{2}$ satisfies $\Hd K \geq 2s + \epsilon$, where $\epsilon > 0$ only depends on $s,t$ and is bounded away from $0$ in a small neighborhood of $(s,t)$.

A \emph{discretised $(\delta,C,s,t)$-Furstenberg set}, on the other hand, is a set $F \subset B(2)$ of the following kind. The set $F$ can be expressed as
\begin{displaymath} F = \bigcup_{\ell \in \mathcal{L}} F(\ell), \end{displaymath}
where $\mathcal{L}$ is a $(\delta,t,C)$-set of lines, and for each $\ell \in \mathcal{L}$, the set $F(\ell)$ is a $(\delta,s,C)$-set of the form $F(\ell) = \cup \mathcal{P}(\ell)$, where
\begin{displaymath} \mathcal{P}(\ell) \subset \mathcal{D}_{\delta} \quad \text{and} \quad F(\ell) \subset \ell_{2\delta} = \{x \in \R^{2} : \dist(x,\ell) < 2\delta\}. \end{displaymath}
According to Lemma \ref{HSYLemma}, our main theorem on Furstenberg sets, Theorem \ref{mainFurstenberg}, will follow once we manage to show that every $(\delta,C,s,t)$-Furstenberg set has Lebesgue measure $\gtrsim_{C} \delta^{2 - (2s + \epsilon)}$, for some $\epsilon = \epsilon(s,t) > 0$ which is bounded away from $0$ in a neighborhood of $s,t$ (this robustness is needed in order to replace ``Hausdorff dimension $u$'' by ``positive $u$-dimensional measure'' for $u\in\{s,t\}$). Equivalently, $|\mathcal{P}| \gtrsim_{C} \delta^{-2s - \epsilon}$, where $\mathcal{P} = \bigcup_{\ell \in \mathcal{L}} \mathcal{P}(\ell)$.

This claim easily follows from Theorem \ref{t:mainDual}. Indeed, for every $\ell \in \mathcal{L}$, one first selects a representative dyadic $\delta$-tube $T = T(\ell) \in \mathcal{T}^{\delta}$ with the property
\begin{displaymath} |\{p \in \mathcal{P}(\ell) : T \cap p \neq \emptyset\}| \sim |\mathcal{P}(\ell)|. \end{displaymath}
This can be done, because $\cup \mathcal{P}(\ell) \subset \ell_{2\delta} \cap B(2)$ can be covered by $\sim 1$ dyadic $\delta$-tubes. (To be precise, for this we need to assume that the lines $\ell \in \mathcal{L}$ all have slope between $(0,1)$, and they intersect the $x$-axis on the segment $\{0\} \times (0,1)$. It is easy to reduce the study of general Furstenberg sets to ones with these constraints.)  Then, the $(\delta,t,C)$-property of $\mathcal{L}$ translates into the $(\delta,t,C')$-set property of $\mathcal{T} = \{T(\ell) : \ell \in \mathcal{L}\}$, and the families $\mathcal{P}(T) := \{p \in \mathcal{P}(\ell) : T \cap p \neq \emptyset\}$ are $(\delta,s,C')$-sets, for some $C' \sim C$. Therefore, Theorem \ref{t:mainDual} implies that $|\mathcal{P}| \geq \delta^{-2s - \epsilon}$, as desired. While this argument was detailed for a specific pair $(s,t)$, using that a $(\delta,u,\delta^{-\epsilon})$-set is automatically a $(\delta,u-\epsilon/2,\delta^{-\epsilon/2})$-set, we see that in Theorem \ref{t:mainDual} the value of "$\epsilon$" can be taken uniform in a neighborhood of $(s,t)$. The proof of Theorem \ref{mainFurstenberg} is complete.

\subsection{Discretising projections} This section contains the proof of Theorem \ref{mainKaufman}. We start with some standard reductions. It is enough to consider directions $e=(e_1,e_2)$ in a given $\pi/4$ arc; after a suitable rotation, we may furthermore assume that $e_1\in (-1,0]$. Since scaling does not change dimension, we may thus redefine the family of projections as $\pi_{\sigma} \colon \R^{2} \to \R$, $\sigma \in (-1,0]$, where $\pi_{\sigma}(x,y) = \sigma x + y$.  Likewise, we may assume that the set $K$ in question is contained in $[0,1)^2$, at the cost of weakening the assumption to $\Hd(K)>t-\epsilon/2$, where $\epsilon>0$ is arbitrarily small.

Fix, then, $s \in (0,1)$ and $t \in (s,2]$ and an analytic set $K\subset [0,1)^2$ with $\Hd(K)>t-\epsilon/2$.  Our re-defined goal is then to show that
\begin{equation}\label{form99} \Hd\{\sigma \in (-1,0] : \Hd \pi_{\sigma}(K) \leq s\} \leq s - \epsilon, \end{equation}
where $\epsilon=\epsilon(s,t)>0$.  This is equivalent to showing that
\begin{equation}\label{form98} \Hd \{\sigma \in (-1,0] : \Hd \pi_{\sigma}(K) < s\} \leq s - \epsilon. \end{equation}
To be precise, to deduce \eqref{form99} from \eqref{form98}, we need to know that the value of the constant "$\epsilon(s,t) > 0$" is bounded away from zero in a neighbourhood of $s$. This will follow from the application of Theorem \ref{t:mainTechnical}, in moment, where the analogous constant "$\epsilon(s,t) > 0$" has this property.

Use Frostman's lemma to find a probability measure $\mu$ with $\spt \mu \subset K$, which satisfies $\mu(B(x,r)) \lesssim r^{t-\epsilon}$ for all $x \in \R^{2}$ and $r > 0$. To reach a contradiction, assume that \eqref{form98} fails: thus $\mathcal{H}^{s - \epsilon}(\Sigma) > 0$, where
\begin{displaymath} \Sigma := \{\sigma \in (-1,0] : \Hd \pi_{\sigma}(K) < s\}, \end{displaymath}
and $\epsilon > 0$ is a parameter to be fixed at the end of the argument. The value of $\epsilon$ will only depend on $s,t$. Let $\mu_\sigma=\pi_\sigma\mu$ be the push-down of $\mu$ under $\pi_{\sigma}$. Since $\mu_{\sigma}$ is supported on $\pi_{\sigma}(K)$, we have $\mathcal{H}^{s}(\spt \mu_{\sigma}) = 0$. Therefore, given any threshold $\delta_{0} \in 2^{-\N}$, there exists a collection $\mathcal{I}_{\sigma}$ of dyadic intervals of $\R$ such that $|I| \leq \delta_{0}$ for all $I \in \mathcal{I}_{\sigma}$,
\begin{displaymath} \sum_{I \in \mathcal{I}_{\sigma}} \mu_{\sigma}(I) = 1 \quad \text{and} \quad \sum_{I \in \mathcal{I}_{\sigma}} |I|^{s} \leq 1. \end{displaymath}
By the pigeonhole principle, we can then find a dyadic number $\delta_{\sigma} = 2^{-j(\sigma)}$, $j(\sigma) \geq 0$, such that $\delta_{\sigma} \leq \delta_{0}$ and, denoting $\mathcal{I}_{\sigma}(\delta_{\sigma}) := \{I \in \mathcal{I} : |I| = \delta_{\sigma}\}$, we have
\begin{equation}\label{form96} \sum_{I \in \mathcal{I}_{\sigma}(\delta_{\sigma})} \mu_{\sigma}(I) \geq c/j^{2} = c \cdot [\log(1/\delta_{\sigma})]^{-2} \quad \text{and} \quad |\mathcal{I}_{\sigma}(\delta_{\sigma})| \leq \delta_{\sigma}^{-s}. \end{equation}
The choice of "$j(\sigma)$" depends on $\sigma \in \Sigma$, but this dependence can be essentially eliminated by another application of the pigeonhole principle: recalling that $\mathcal{H}^{s - \epsilon}(\Sigma) > 0$, there exists a fixed index $j_{0} \geq 0$, and a subset $\overline{\Sigma} \subset \Sigma$ with $\mathcal{H}^{s - \epsilon}(\overline{\Sigma}) \gtrsim 1/j_{0}^{2}$, such that
\begin{displaymath} \delta_{\sigma} = 2^{-j(\sigma)} = 2^{-j_{0}} =: \delta, \qquad \sigma \in \overline{\Sigma}. \end{displaymath}
At this point, we discretise everything at the scale $\delta > 0$ we located above. In particular, by the lower bound $\mathcal{H}^{s - \epsilon}_{\infty}(\overline{\Sigma}) \gtrsim 1/j_{0}^{2} =  [\log(1/\delta)]^{-2}$, there exists by Lemma \ref{deltasSet} a $(\delta,s - \epsilon,C \cdot [\log(1/\delta)]^{2})$-set
\begin{displaymath} \Sigma' \subset (\delta \cdot \Z) \cap \overline{\Sigma}(\delta) \quad \text{with} \quad |\Sigma'| \leq \delta^{-s + \epsilon}. \end{displaymath}
The notation $\overline{\Sigma}(\delta)$ refers to the $\delta$-neighbourhood of $\overline{\Sigma}$. For every $\sigma \in \Sigma'$, it follows from \eqref{form96} that a certain fairly large subset of $\spt \mu$ may be covered by $\lesssim \delta^{-s}$ dyadic tubes with common slope $-\sigma \in (\delta \cdot \Z) \cap [0,1)$ (it is easy to check that  the pre-images $\pi_{\sigma}^{-1}\{r\}$, $r \in \R$, are lines with slope $-\sigma \in [0,1)$.) More precisely, for every $\sigma \in \Sigma'$, there exists a family of dyadic tubes $\mathcal{T}_{\sigma}$ with the properties
\begin{displaymath} |\mathcal{T}_{\sigma}| \lesssim \delta^{-s} \quad \text{and} \quad \mu(\cup \mathcal{T}_{\sigma}) \gtrsim [\log(1/\delta)]^{-2}. \end{displaymath}
We write $\mathcal{T} := \bigcup_{\sigma \in \Sigma'} \mathcal{T}_{\sigma}$, and we record that $|\mathcal{T}| \lesssim \delta^{-2s + \epsilon}$, since $|\Sigma'| \leq \delta^{-s + \epsilon}$. Let $\mathcal{P}_{\sigma} \subset \mathcal{D}_{\delta}$ be the collection of dyadic sub-squares of $[0,1)^{2}$ which have non-empty intersection with at least one of the tubes from $\mathcal{T}_{\sigma}$. We note that $\mu(\cup \mathcal{P}_{\sigma}) \geq \mu(\cup \mathcal{T}_{\sigma})$, hence
\begin{displaymath} \sum_{p \in \mathcal{D}_{\delta}} \mu(p) \cdot |\{\sigma \in \Sigma' : p \in \mathcal{P}_{\sigma}\}| = \sum_{\sigma \in \Sigma'} \mu(\cup \mathcal{P}_{\sigma}) \gtrsim |\Sigma'| \cdot [\log(1/\delta)]^{-2} \gtrapprox_{\delta} \delta^{-s + \epsilon}, \end{displaymath}
also recalling that $\Sigma'$ is a $(\delta,s - \epsilon,C \cdot [\log(1/\delta)]^{2})$-set. Since on the other hand
\begin{displaymath}
|\{\sigma \in \Sigma' : p \in \mathcal{P}_{\sigma}\}| \leq |\Sigma'| \leq \delta^{-s + \epsilon}
\end{displaymath}
for all $p \in \mathcal{D}_{\delta}$, and $\mu$ is a probability measure, we may infer the existence of a family $\mathcal{P} \subset \mathcal{D}_{\delta}$ with the following properties:
\begin{equation}\label{form97} \mu(\cup \mathcal{P}) \approx_{\delta} 1 \quad \text{and} \quad |\{\sigma \in \Sigma' : p \in \mathcal{P}_{\sigma}\}| \approx_{\delta} \delta^{-s + \epsilon} \text{ for all } p \in \mathcal{P}. \end{equation}
Since $\spt \mu \subset K \subset [0,1)^{2}$, we may also assume that $\cup \mathcal{P} \subset [0,1)^{2}$. The family $\mathcal{P}$ may not be a $(\delta,t)$-set, but since $\mu(\cup \mathcal{P}) \approx_{\delta} 1$, and $\mu$ is a $(t-\epsilon)$-dimensional Frostman measure, we have $\mathcal{H}^{t-\epsilon}_{\infty}(\cup \mathcal{P}) \approx_{\delta} 1$. Therefore, by another application of Lemma \ref{deltasSet}, there exists a $(\delta,t-\epsilon,C)$-set (and thus also a $(\delta,t, C\delta^{-\epsilon})$-set)  $\overline{\mathcal{P}} \subset \mathcal{P}$ with $C \approx_{\delta} 1$. For every $p \in \overline{\mathcal{P}}$, we recall from \eqref{form97} that there correspond $\approx_{\delta} \delta^{-s + \epsilon}$ choices of $\sigma \in \Sigma'$ such that at least one tube from $\mathcal{T}_{\sigma}$ intersects $p$. Since $\Sigma'$ is a $(\delta,s - \epsilon,C)$-set, also with $C \approx_{\delta} 1$, the family of these tubes forms a $(\delta,s,C \delta^{-\epsilon})$-set $\mathcal{T}(p) \subset \mathcal{T}$.

Therefore, we have now constructed a $(\delta,t,C \delta^{-\epsilon})$-set $\overline{\mathcal{P}} \subset \mathcal{D}_{\delta}$, and for every $p \in \overline{\mathcal{P}}$ a $(\delta,s,C\delta^{-\epsilon})$-set $\mathcal{T}(p)$ of dyadic $\delta$-tubes, all intersecting $p$. By Theorem \ref{t:mainTechnical}, these facts should imply that $|\mathcal{T}| \geq \delta^{-2s - \epsilon}$, assuming that $\epsilon = \epsilon(s,t) > 0$ was chosen small enough. However, in fact $|\mathcal{T}| \leq |\Sigma'| \cdot \max_{\sigma \in \Sigma'} |\mathcal{T}_{\sigma}| \lesssim \delta^{-2s + \epsilon}$, and a contradiction ensues for $\delta > 0$ small enough. This completes the proof of Theorem \ref{mainKaufman}.

\section{Covering thin tubes with thick tubes without losing separation}\label{s:thickTubeCover}

Before formulating the main result in this section, let us briefly explain what it achieves. Let $0 < \delta \leq \Delta \leq 1$ be dyadic numbers. Let $Q \in \mathcal{D}_{\Delta}$, and let $\mathcal{P} \subset \mathcal{D}_{\delta}(Q)$ be a family of sub-squares of $Q$. Assume that for every $p \in \mathcal{P}$, there is a $(\delta,s)$-set $\mathcal{T}(p) \subset \mathcal{T}^{\delta}$ such that $T \cap p \neq \emptyset$ for all $T \in \mathcal{T}(p)$. Then all the $\delta$-tubes in the family $\mathcal{T} = \bigcup_{p \in P} \mathcal{T}(p)$ intersect $Q$. Now, suppose that we cover $\cup \mathcal{T}$ by some minimal collection of dyadic $\Delta$-tubes, say $\mathcal{T}_{\Delta}(Q)$. Then all the tubes in $\mathcal{T}_{\Delta}(Q)$ intersect $Q$. Is $\mathcal{T}_{\Delta}(Q)$ a $(\Delta,s)$-set? Certainly not: even a fixed family $\mathcal{T}(p)$ need not be a $(\Delta,s)$-set. What is worse, the families $\mathcal{T}(p)$ can be so different from each other that $\mathcal{T}_{\Delta}(Q)$ may not enjoy any properties of the individual families $\mathcal{T}(p)$. Regardless: Proposition \ref{prop2} will imply that after refining both $\mathcal{P}$ and the families $\mathcal{T}(p)$ appropriately, the minimal cover $\mathcal{T}_{\Delta}(Q)$ is, in fact, a $(\Delta,s)$-set. In fact, Proposition \ref{prop2} will not make any reference to the fixed square "$Q$", but in practice, we will apply it in situations as described above.

In Proposition \ref{prop2}, the notation $A \lessapprox_{\Delta} B$ means that there exists an absolute constant $C \geq 1$ such that $A \leq C \cdot [\log (1/\Delta)]^{C} B$.

\begin{proposition}\label{prop2} Let $0 < \delta \leq \Delta \leq 1$ be dyadic numbers, and let $C_{1},M \geq 1$. Let $\mathcal{P}$ be a finite set, and assume that for every $p \in \mathcal{P}$, there is an associated $(\delta,s,C_{1})$-set $\mathcal{T}(p) \subset \mathcal{T}^{\delta}$ with $\tfrac{M}{2} < |\mathcal{T}(p)| \leq M$, and such that $T \cap [0,1)^{2} \neq \emptyset$ for all $T \in \mathcal{T}(p)$ and all $p \in \mathcal{P}$.

Then, there exist a subset $\overline{\mathcal{P}} \subset \mathcal{P}$ of cardinality $|\overline{\mathcal{P}}| \approx_{\Delta} |\mathcal{P}|$, and a collection $\overline{\mathcal{T}}_{\Delta} \subset \mathcal{T}^{\Delta}$ of dyadic $\Delta$-tubes intersecting $[0,1)^2$ with the following properties:
\begin{enumerate}
\item \label{prop2i} $\overline{\mathcal{T}}_{\Delta}$ is a $(\Delta,s,C_{2})$-set with $C_{2} \lessapprox_{\Delta} C_{1}$,
\item \label{prop2ii} There exists a constant $H \approx_{\Delta} M \cdot |\mathcal{P}|/|\overline{\mathcal{T}}_{\Delta}|$ such that
\begin{displaymath} |\{(p,T) \in \overline{\mathcal{P}} \times \mathcal{T}^{\delta} : T \in \mathcal{T}(p) \text{ and } T \subset \mathbf{T}\}| \gtrsim H, \qquad \mathbf{T} \in \overline{\mathcal{T}}_{\Delta}. \end{displaymath}
\end{enumerate}
\end{proposition}

\begin{remark} To understand the numerology, note that
\begin{align*} \frac{1}{|\overline{\mathcal{T}}_{\Delta}|} & \sum_{\mathbf{T} \in \overline{\mathcal{T}}_{\Delta}} |\{(p,T) \in \overline{\mathcal{P}} \times \mathcal{T}^{\delta} : T \in \mathcal{T}(p) \text{ and } T \subset \mathbf{T}\}|\\
& \qquad \leq \frac{|\{(p,T) \in \mathcal{P} \times \mathcal{T}^{\delta} : T \in \mathcal{T}(p)\}|}{|\overline{\mathcal{T}}_{\Delta}|} \le \frac{M \cdot |\mathcal{P}|}{|\overline{\mathcal{T}}_{\Delta}|} \approx_{\Delta} H, \end{align*}
so the uniform lower bound in \eqref{prop2ii} essentially matches the upper bound for the average. \end{remark}

\begin{remark} In our concrete applications of Proposition \ref{prop2}, $\mathcal{P}$ will be a collection of $\delta$-sub-squares of a $\Delta$-square $Q$, but this plays no role in the proof so we chose to state it for an arbitrary finite set indexing a collection of $(\delta,s,C_{1})$-sets of dyadic $\delta$-tubes.
\end{remark}

\begin{proof}[Proof of Proposition \ref{prop2}] Let $\mathcal{T}_{\Delta}$ be a minimal cover of $\mathcal{T} := \bigcup_{p \in \mathcal{P}} \mathcal{T}(p)$ by dyadic $\Delta$-tubes. Then each tube in $\mathcal{T}_{\Delta}$ intersects $[0,1)^{2}$, hence $|\mathcal{T}_{\Delta}| \leq 100\Delta^{-2}$. For $p \in \mathcal{P}$ fixed, different tubes in $\mathcal{T}_{\Delta}$ may contain different numbers of tubes from $\mathcal{T}(p)$, and we need to perform an initial pigeonholing to fix this. Let
\begin{displaymath} \mathcal{T}_{\Delta,j}(p) := \{\mathbf{T} \in \mathcal{T}_{\Delta} : 2^{j - 1} < |\{T \in \mathcal{T}(p) : T \subset \mathbf{T}\}| \leq 2^{j}\}, \end{displaymath}
and note that
\begin{displaymath} M \sim |\mathcal{T}(p)| \leq \sum_{2^{j} \leq M} 2^{j} \cdot |\mathcal{T}_{\Delta,j}(p)|. \end{displaymath}
Since $|\mathcal{T}_{\Delta}| \leq 100\Delta^{-2}$, the sum over those indices $j \geq 0$ with $2^{j} \leq M\Delta^{2}/200$ cannot dominate the left hand side. Thus,
\begin{displaymath} M \lesssim \sum_{M\Delta^{2}/200 \leq 2^{j} \leq M} 2^{j} \cdot |\mathcal{T}_{\Delta,j}(p)|. \end{displaymath}
Now, the number of terms in the sum is bounded by $\lessapprox_{\Delta} 1$, so there exists an index $j = j(p)$ such that $2^{j} \cdot |\mathcal{T}_{\Delta,j}(p)| \approx_{\Delta} M$. We write
\begin{equation}\label{form66} m_{1}(p) := 2^{j(p)}, \quad \mathcal{T}_{\Delta}(p) := \mathcal{T}_{\Delta,j(p)}(p), \quad \text{and} \quad m_{2}(p) := 2^{-j(p)} \cdot M \approx_{\Delta} |\mathcal{T}_{\Delta}(p)|. \end{equation}
Thus $|\{T \in \mathcal{T}(p) : T \subset \mathbf{T}\}| \sim m_{1}(p)$ for all $\mathbf{T} \in \mathcal{T}_{\Delta}(p)$, and the union of the tubes in $\mathcal{T}_{\Delta}(p)$ contains $\approx_{\Delta} M$ distinct tubes from $\mathcal{T}(p)$.

The next trouble is that $m_{1}(p),m_{2}(p)$ depend on $p \in \mathcal{P}$, and another pigeonholing is needed to fix that. Noting that there are $\lessapprox_{\Delta} 1$ possible choices for the pair $(m_{1}(p),m_{2}(p))$ (which is in fact determined by $m_{2}(p)\lessapprox_{\Delta} 1$ alone), there exists a fixed pair $(m_{1},m_{2})$ such that
\begin{displaymath} |\{p \in \mathcal{P} : (m_{1}(p),m_{2}(p)) = (m_{1},m_{2})\}| \approx_{\Delta} |\mathcal{P}|. \end{displaymath}
We let $\overline{\mathcal{P}}$ be the subset of $\mathcal{P}$ defined above. For $p \in \overline{\mathcal{P}}$, we let
\begin{displaymath} \overline{\mathcal{T}}(p) := \{T \in \mathcal{T}(p) : T \subset \mathbf{T} \text{ for some }\mathbf{T} \in \mathcal{T}_{\Delta}(p)\}. \end{displaymath}
Then $|\overline{\mathcal{T}}(p)| \approx_{\Delta} M$ by the choice of $\mathcal{T}_{\Delta}(p)$ in \eqref{form66}. After these initial reductions, the tube families $\overline{\mathcal{T}}(p)$ have gained a small amount of uniformity: each $\overline{\mathcal{T}}(p)$, $p \in \overline{\mathcal{P}}$, is covered by the $\approx_{\Delta} m_{2}$ dyadic $\Delta$-tubes in $\mathcal{T}_{\Delta}(p)$, and
\begin{equation}\label{form43} |\{T \in \overline{\mathcal{T}}(p) : T \subset \mathbf{T}\}| \sim \begin{cases} m_{1}, & \text{if } \mathbf{T} \in \mathcal{T}_{\Delta}(p), \\ 0, & \text{if } \mathbf{T} \in \mathcal{T}_{\Delta} \, \setminus \, \mathcal{T}_{\Delta}(p). \end{cases} \end{equation}
We continue the proof by calculating that
\begin{equation} \label{form28}
\begin{split}  M \cdot |\overline{\mathcal{P}}| & \approx_{\Delta} \sum_{p \in \overline{\mathcal{P}}} \sum_{\mathbf{T} \in \mathcal{T}_{\Delta}(p)} |\{T \in \overline{\mathcal{T}}(p) : T \subset \mathbf{T}\}| \\
& \sim m_{1} \sum_{\mathbf{T} \in \mathcal{T}_{\Delta}} |\{p \in \overline{\mathcal{P}} : \mathbf{T} \in \mathcal{T}_{\Delta}(p)\}| \sim m_{1} \sum_{2^{j} \leq |\overline{\mathcal{P}}|} 2^{j} \cdot |\mathcal{T}_{\Delta,j}|,
\end{split}
\end{equation}
where
\begin{equation}\label{form64} \mathcal{T}_{\Delta,j} := \{\mathbf{T} \in \mathcal{T}_{\Delta} : 2^{j - 1} < |\{p \in \overline{\mathcal{P}} :  \mathbf{T} \in \mathcal{T}_{\Delta}(p) \}| \leq 2^{j}\}. \end{equation}
In fact, the sum on the right hand side of \eqref{form28} can be restricted to those $j \geq 0$ with $2^{j} \geq c|\overline{\mathcal{P}}|\Delta^{2}$ for some $c \approx_{\Delta} 1$, since $m_1 \le M$ and $|\mathcal{T}_{\Delta,j}| \leq |\mathcal{T}_{\Delta}| \leq 100\Delta^{-2}$, and consequently the sum over $2^{j} < c|\overline{\mathcal{P}}|\Delta^{2}$ cannot dominate the left hand side of \eqref{form28}. This observation is used to infer that there are only $\approx_{\Delta} 1$ choices of "$j$" one needs to consider. From \eqref{form28}, one may now deduce the existence of
\begin{equation}\label{form67} c|\overline{\mathcal{P}}|\Delta \leq 2^{j} \leq |\overline{\mathcal{P}}| \quad \text{such that} \quad 2^{j} \cdot |\mathcal{T}_{\Delta,j}| \approx_{\Delta} \frac{M \cdot |\overline{\mathcal{P}}|}{m_{1}} = m_{2} \cdot |\overline{\mathcal{P}}|. \end{equation}
We then define $\overline{\mathcal{T}}_{\Delta} := \mathcal{T}_{\Delta,j}$ for the index "$j$" located above. We record that, by the definition of $\mathcal{T}_{\Delta,j}$ in \eqref{form64}, we have
\begin{equation}\label{form29a} |\{p \in \overline{\mathcal{P}} : \mathbf{T} \in \mathcal{T}_{\Delta}(p)\}| \sim 2^{j} \approx_{\Delta} |\overline{\mathcal{T}}_{\Delta}|^{-1} \cdot m_{2} \cdot |\overline{\mathcal{P}}|, \qquad \mathbf{T} \in \overline{\mathcal{T}}_{\Delta}. \end{equation}
It remains to verify the properties \eqref{prop2i}-\eqref{prop2ii}. We claim that \eqref{prop2ii} is valid with the constant $H := 2^{j} \cdot m_{1}$, which indeed satisfies $H \approx_{\Delta} M \cdot |\mathcal{P}|/|\overline{\mathcal{T}}_{\Delta}|$ by \eqref{form67}, and $|\overline{\mathcal{P}}| \approx_{\Delta} |\mathcal{P}|$. To prove \eqref{prop2ii} with this choice of "$H$", fix $\mathbf{T} \in \overline{\mathcal{T}}_{\Delta}$, and note that
\begin{align*} |\{(p,T) \in \overline{\mathcal{P}} \times \mathcal{T}^{\delta} : T \in \mathcal{T}(p) \text{ and } T \subset \mathbf{T}\}| = \sum_{p \in \overline{\mathcal{P}}} |\{T \in \mathcal{T}(p) : T \subset \mathbf{T}\}|. \end{align*}
Recalling \eqref{form43}, one has $|\{T \in \mathcal{T}(p) : T \subset \mathbf{T}\}| \gtrsim m_1$ for all $\mathbf{T} \in \mathcal{T}_{\Delta}(p)$. Therefore,
\begin{displaymath} |\{(p,T) \in \overline{\mathcal{P}} \times \mathcal{T}^{\delta} : T \in \mathcal{T}(p) \text{ and } T \subset \mathbf{T}\}| \gtrsim m_1 \cdot |\{p \in \overline{\mathcal{P}} : \mathbf{T} \in \mathcal{T}_{\Delta}(p)\}|.\end{displaymath}
In combination with \eqref{form29a}, this shows that $|\{(p,T) \in \overline{\mathcal{P}} \times \mathcal{T}^{\delta} : T \in \mathcal{T}(p) \text{ and } T \subset \mathbf{T}\}| \gtrsim 2^{j} \cdot m_{1} = H$, as claimed.

To prove the claim \eqref{prop2i}, fix a dyadic tube $\mathbf{T}_{r} \in \mathcal{T}^{r}$ with $r \geq \Delta$. Write
\begin{displaymath}
 \overline{\mathcal{T}}_{\Delta}(\mathbf{T}_{r}) := \{\mathbf{T} \in \overline{\mathcal{T}}_{\Delta} : \mathbf{T} \subset \mathbf{T}_{r}\}.
\end{displaymath}
To finish the proof, we must show that $|\overline{\mathcal{T}}_{\Delta}(\mathbf{T}_{r})| \lessapprox_{\Delta} C_{1} \cdot |\overline{\mathcal{T}}_{\Delta}| \cdot r^{s}$. Start by observing that
\begin{align*} |\overline{\mathcal{T}}_{\Delta}(\mathbf{T}_{r})| \cdot |\overline{\mathcal{T}}_{\Delta}|^{-1} \cdot m_{2} \cdot |\overline{\mathcal{P}}| &\stackrel{\eqref{form29a}}{\approx_{\Delta}} \sum_{\mathbf{T} \in \overline{\mathcal{T}}_{\Delta}(\mathbf{T}_{r})} |\{p \in \overline{\mathcal{P}} : \mathbf{T} \in \mathcal{T}_{\Delta}(p)\}| \\
&= \sum_{p \in \overline{\mathcal{P}}} |\overline{\mathcal{T}}_{\Delta}(\mathbf{T}_{r}) \cap \mathcal{T}_{\Delta}(p)|.
\end{align*}
Dividing by $|\overline{\mathcal{P}}|$, we find that there exists $p_{0} \in \overline{\mathcal{P}}$ with
\begin{displaymath} |\overline{\mathcal{T}}_{\Delta}(\mathbf{T}_{r}) \cap \mathcal{T}_{\Delta}(p_{0})| \gtrapprox_{\Delta} |\overline{\mathcal{T}}_{\Delta}(\mathbf{T}_{r})| \cdot |\overline{\mathcal{T}}_{\Delta}|^{-1} \cdot m_{2}. \end{displaymath}
Recall from \eqref{form43} that $|\{T \in \mathcal{T}(p_{0}) : T \subset \mathbf{T}\}| \gtrsim m_{1}$ for $\mathbf{T} \in \mathcal{T}_{\Delta}(p_{0})$. Consequently,
\begin{displaymath} |\{T \in \mathcal{T}(p_0) : T \subset \mathbf{T}_{r}\}| \gtrsim m_{1} \cdot |\overline{\mathcal{T}}_{\Delta}(\mathbf{T}_{r}) \cap \mathcal{T}_{\Delta}(p_{0})| \gtrapprox_{\Delta} |\overline{\mathcal{T}}_{\Delta}(\mathbf{T}_{r})| \cdot |\overline{\mathcal{T}}_{\Delta}|^{-1} \cdot M. \end{displaymath}
Finally, using the $(\delta,s,C_{1})$-set property of $\mathcal{T}(p_{0})$, and recalling $|\mathcal{T}(p_{0})| \sim M$, we deduce that
\begin{displaymath} |\overline{\mathcal{T}}_{\Delta}(\mathbf{T}_{r})| \lessapprox_{\Delta} |\overline{\mathcal{T}}_{\Delta}| \cdot M^{-1} \cdot |\{T \in \mathcal{T}(p_0) : T \subset \mathbf{T}_{r}\}| \lesssim C_{1} \cdot |\overline{\mathcal{T}}_{\Delta}| \cdot r^{s}. \end{displaymath}
This completes the proof of the proposition.  \end{proof}

\section{An induction on scales scheme for incidence counting}\label{s:induction-on-scales}

The main result of this section is Proposition \ref{p:induction-on-scales}. Roughly speaking, this proposition will allow us to prove incidence estimates at scale $\delta$ in terms of incidence estimates at coarser scales, so it can be seen as an induction on scales mechanism for incidence counting. It will be a crucial step in the proof of Theorem \ref{t:mainTechnical}. The proof relies on Proposition \ref{prop2} and careful pigeonholing.

Fix two dyadic scales $0<\delta<\Delta\le 1$ and families $\mathcal{P}_0\subset\mathcal{D}_{\delta}$ and $\mathcal{T}_0 \subset \mathcal{T}^{\delta}$. For $Q \in \mathcal{D}_{\Delta}$ and $\mathbf{T} \in \mathcal{T}^{\Delta}$, we denote
\[
\mathcal{P}_0\cap Q = \{ p\in\mathcal{P}_0: p\subset Q\} \quad \text{and} \quad \mathcal{T}_0 \cap \mathbf{T} := \{T \in \mathcal{T}_0 : T \subset \mathbf{T}\}.
\]
We also write
\[
\mathcal{D}_{\Delta}(\mathcal{P}_0) = \{ Q\in\mathcal{D}_{\Delta}: \mathcal{P}_0 \cap Q \neq\emptyset\} \quad \text{and} \quad \mathcal{T}^{\Delta}(\mathcal{T}_0) := \{\mathbf{T} \in \mathcal{T}^{\Delta} : \mathcal{T}_0 \cap \mathbf{T} \neq \emptyset\}.
\]
Finally, if $S:\R^2\to\R^2$ is a map, we let $S(\mathcal{P}_0)=\{ S(p):p\in\mathcal{P}_0\}$ and $S(\mathcal{T}_0) = \{S(T) : T \in \mathcal{T}_0\}$.

\begin{definition}
Fix $\delta\in 2^{-\mathbb{N}}$, $s\in [0,1]$, $C>0$,  $M\in\mathbb{N}$. We say that a pair $(\mathcal{P}_0,\mathcal{T}_0) \subset \mathcal{D}_{\delta} \times \mathcal{T}^{\delta}$ s a \emph{$(\delta,s,C,M)$-nice configuration} if for every $p\in\mathcal{P}_0$ there exists a $(\delta,s,C)$-set $\mathcal{T}(p) \subset\mathcal{T}_0$ with $ |\mathcal{T}(p)| = M$ and such that $T \cap p \neq\emptyset$ for all $T\in\mathcal{T}(p)$.
\end{definition}

We make some remarks on this definition:
\begin{enumerate}[(\rm a)]
\item No non-concentration assumptions are made on $\mathcal{P}$ - only on the families $\mathcal{T}(p)$.
\item In practice, we often have $|\mathcal{T}(p)|\sim M$ rather than $|\mathcal{T}(p)|=M$. However, we can easily get a nice configuration by trimming each family $\mathcal{T}(p)$ to $M'\sim M$ elements - this will only incur an innocuous  constant loss in the parameter "$C$". So we will not distinguish between $|\mathcal{T}(p)|\sim M$ and $|\mathcal{T}(p)|=M$ in the sequel.
\item Finally, we point out that $|\mathcal{T}(p)| \geq \delta^{-s}/C$ by virtue of being a $(\delta,s,C)$-set, but $M$ is allowed to be much larger than $\delta^{-s}$.
\end{enumerate}

In the next proposition, for $\Delta \in 2^{-\N}$ and $Q \in \mathcal{D}_{\Delta}$, the map $S_{Q} \colon \R^{2} \to \R^{2}$ is the homothety that maps $Q$ to the square $[0,1)^{2}$.  Also, the notation $A \lessapprox_{\delta} B$ means that there exists an absolute constant $C \geq 1$ such that $A \leq C \cdot [\log (1/\delta)]^{C} B$.

\begin{proposition}\label{p:induction-on-scales} Let $\delta,\Delta \in 2^{-\N}$ with $\delta \leq \Delta$. Let $(\mathcal{P}_0,\mathcal{T}_0)$ be a $(\delta,s,C_1,M)$-nice configuration. Then there exist sets $\mathcal{P}\subset \mathcal{P}_0$ and $\mathcal{T}(p)\subset\mathcal{T}_0(p)$, $p\in \mathcal{P}$, such that denoting $\mathcal{T}=\bigcup_{p\in \mathcal{P}} \mathcal{T}(p)$ the following hold:

\begin{enumerate}[(\rm i)]
\item \label{it-induction-i} $|\mathcal{D}_{\Delta}(\mathcal{P})| \approx_{\delta} |\mathcal{D}_{\Delta}(\mathcal{P}_{0})|$ and $|\mathcal{P}\cap Q| \approx_{\delta}|\mathcal{P}_0\cap Q|$ for all $Q\in\mathcal{D}_{\Delta}(\mathcal{P})$.
\item \label{it-induction-ii} $|\mathcal{T}(p)|\gtrapprox_{\delta} |\mathcal{T}_0(p)|=M$ for $p\in\mathcal{P}$.
\item \label{it-induction-iii} There are $\mathcal{T}_{\Delta}\subset\mathcal{T}^{\Delta}$, $C_{\Delta} \approx_{\delta} C_{1}$ and $M_{\Delta} \geq 1$ such that $(\mathcal{D}_{\Delta}(\mathcal{P}),\mathcal{T}_{\Delta})$ is a $(\Delta,s,C_{\Delta},M_{\Delta})$-nice configuration. Moreover, the associated families $\mathcal{T}_{\Delta}(Q)$ satisfy
\begin{equation} \label{app21} 
\cup \mathcal{T}(p) \subset \cup \mathcal{T}_{\Delta}(Q), \qquad p \in \mathcal{P} \cap Q, \,\ Q\in\mathcal{D}_{\Delta}(\mathcal{P}). 
\end{equation}
\item \label{it-induction-iv} For each $Q\in\mathcal{D}_{\Delta}(\mathcal{P})$ there exist $C_Q \approx_{\delta} C_1$, $M_Q\ge 1$,  and a family of tubes $\mathcal{T}_{Q}\subset\mathcal{T}^{\delta/\Delta}$ such that $(S_{Q}(\mathcal{P}\cap Q),\mathcal{T}_Q)$ is $(\delta/\Delta,s,C_Q,M_Q)$-nice. Moreover,
\begin{equation}\label{app1} \mathcal{D}_{\delta/\Delta} \left[ \sigma\left(\mathcal{T}_{Q}(S_{Q}(p))\right) \right] = \mathcal{D}_{\delta/\Delta}[\sigma(\mathcal{T}(p))], \qquad p \in \mathcal{P} \cap Q. \end{equation}
\end{enumerate}
Furthermore, the families $\mathcal{T}_{\Delta},\mathcal{T}_{Q}$ can be chosen so that
\begin{equation}\label{lower-bound-T}
\frac{|\mathcal{T}_0|}{M} \gtrapprox_\delta \frac{|\mathcal{T}_{\Delta}|}{M_\Delta}\cdot \left( \max_{Q \in \mathcal{D}_{\Delta}(\mathcal{P})}\frac{|\mathcal{T}_{Q}|}{M_Q} \right).
\end{equation}
\end{proposition}

Thanks to Proposition \ref{p:induction-on-scales}, the problem of finding  lower bounds for $|\mathcal{T}_0|$ is reduced to finding lower bounds for the cardinalities of the families of $\Delta$-tubes $\mathcal{T}_{\Delta}$ and $(\delta/\Delta)$-tubes $\mathcal{T}_{Q}$. This is the induction on scales mechanism described at the beginning of this section.

\begin{remark} Note that $C_{\Delta},C_{Q} \approx_{\delta} C_{1}$ instead of $C_{\Delta} \approx_{\Delta} C_{1}$ and $C_{Q} \approx_{\delta/\Delta} \approx C_{1}$. So, for the proposition to be useful in practice one needs
\[
\log(1/\Delta),\log(\Delta/\delta)\sim \log(1/\delta).
\]
\end{remark}

\begin{remark} \label{r:P-P_0}
We also point out that no guarantees are made that $|\mathcal{P}| \approx_\delta |\mathcal{P}_0|$. However, it follows from Claim \eqref{it-induction-i} that if the sets $\mathcal{P}_0\cap Q$ , $Q\in\mathcal{D}_{\Delta}(\mathcal{P}_0)$, have comparable cardinalities to begin with, then the cardinalities of $\mathcal{P}$ and $\mathcal{P}_0$ are also roughly comparable - this will be the case in our applications.
\end{remark}

\begin{remark}\label{rem6} The following extra property can be added to the requirements of the families $\mathcal{T}_{\Delta}(Q)$, $Q \in \mathcal{D}_{\Delta}(\mathcal{P})$ (we will point out in Remark \ref{rem7} the small extra step which needs to be taken). Fix $\bar{\Delta} \in [\Delta,1] \cap 2^{-\N}$. Then, the map
\begin{displaymath} \overline{\mathbf{T}} \mapsto |\overline{\mathbf{T}} \cap \mathcal{T}_{\Delta}(Q)|, \qquad \overline{\mathbf{T}} \in  \mathcal{T}^{\bar{\Delta}}(\mathcal{T}_{\Delta}(Q)),\, Q\in\mathcal{D}_{\Delta}(\mathcal{P}), \end{displaymath}
is constant (independent of $Q$). 
\end{remark}


\begin{proof}[Proof of Proposition \ref{p:induction-on-scales}]

 Fix $Q \in \mathcal{D}_{\Delta}(\mathcal{P}_0)$. By applying Proposition \ref{prop2} to the set $\mathcal{P}_{0} \cap Q$, we may find a subset $\overline{\mathcal{P}}_{Q} \subset \mathcal{P}_0 \cap Q$ of cardinality $|\overline{\mathcal{P}}_{Q}| \approx_{\Delta} |\mathcal{P}_0 \cap Q|$, and a family of dyadic $\Delta$-tubes $\overline{\mathcal{T}}_{\Delta}(Q)$ intersecting $Q$ such that the following properties hold:
\begin{itemize}
\item[(T1) \phantomsection \label{T1}] $\overline{\mathcal{T}}_{\Delta}(Q)$ is a $(\Delta,s,C_{\Delta})$-set for some $C_{\Delta} \approx_{\Delta} C_1$.
\item[(T2) \phantomsection \label{T2}] There exists a constant $H_{Q} \approx_{\Delta} M \cdot |\overline{\mathcal{P}}_{Q}|/|\overline{\mathcal{T}}_{\Delta}(Q)|$ such that
\begin{displaymath} |\{(p,T) \in \overline{\mathcal{P}}_{Q} \times \mathcal{T}^\delta : T \in \mathcal{T}_0(p) \text{ and } T \subset \mathbf{T}\}| \gtrsim H_{Q}, \qquad \mathbf{T} \in \overline{\mathcal{T}}_{\Delta}(Q). 
\end{displaymath}
\end{itemize}
Note that, even though $C_{\Delta}$ a priori depends on $Q$, the implicit constant in $C_{\Delta} \approx_{\Delta} C_1$ is independent of $Q$ and so we can indeed take a uniform value over all $Q\in \mathcal{D}_{\Delta}(\mathcal{P}_0)$.

All the tubes in $\overline{\mathcal{T}}_{\Delta}(Q)$ intersect $[0,1)^{2}$, so $|\overline{\mathcal{T}}_{\Delta}(Q)| \le 100\Delta^{-2}$. By the pigeonhole principle, we may find $\overline{M}_{\Delta} \geq 1$, and a subset $\overline{\mathcal{Q}} \subset \mathcal{D}_{\Delta}(\mathcal{P}_{0})$ with cardinality $|\overline{\mathcal{Q}}| \approx_{\Delta} |\mathcal{D}_{\Delta}(\mathcal{P}_0)|$, such that $\overline{M}_{\Delta} \leq |\overline{\mathcal{T}}_{\Delta}(Q)| \leq 2\overline{M}_{\Delta}$ for all $Q \in \overline{\mathcal{Q}}$.

Write
\begin{displaymath}  \overline{\mathcal{T}}_{\Delta} := \bigcup_{Q \in \overline{\mathcal{Q}}} \overline{\mathcal{T}}_{\Delta}(Q).
\end{displaymath}
We will next perform another pigeonholing to ensure that $|\mathcal{T}_0 \cap \mathbf{T}|$ is roughly constant for all tubes $\mathbf{T}$ in a substantial subset of $\overline{\mathcal{T}}_{\Delta}$. To this end, we define
\begin{displaymath} \mathcal{I}(\overline{\mathcal{Q}},\overline{\mathcal{T}}_{\Delta}) := \{(Q,\mathbf{T}) \in \overline{\mathcal{Q}} \times \overline{\mathcal{T}}_{\Delta} : \mathbf{T} \in \overline{\mathcal{T}}_{\Delta}(Q)\}.\end{displaymath}
Note that $|\mathcal{I}(\overline{\mathcal{Q}},\overline{\mathcal{T}}_{\Delta})| \sim |\overline{\mathcal{Q}}| \cdot \overline{M}_{\Delta}$. For $j \geq 1$, let
\begin{displaymath} \overline{\mathcal{T}}_{\Delta,j} := \{\mathbf{T} \in \overline{\mathcal{T}}_{\Delta} : 2^{j - 1} < |\mathcal{T}_0 \cap \mathbf{T}| \leq 2^{j}\}. \end{displaymath}
Since $|\mathcal{T}^{\delta} \cap \mathbf{T}| \leq (\Delta/\delta)^{2} \leq \delta^{-2}$ for all $\mathbf{T} \in \overline{\mathcal{T}}_{\Delta}$, we have
\begin{displaymath}
|\overline{\mathcal{Q}}| \cdot \overline{M}_{\Delta} \sim |\mathcal{I}(\overline{\mathcal{Q}},\overline{\mathcal{T}}_{\Delta})| = \sum_{2^{j} \leq \delta^{-2}} |\{(Q,\mathbf{T}) \in \overline{\mathcal{Q}} \times \overline{\mathcal{T}}_{\Delta,j} : \mathbf{T} \in \overline{\mathcal{T}}_{\Delta}(Q)\}|.
\end{displaymath}
Therefore, we may pick $j \in \{1,\ldots,2\log (1/\delta)\}$ such that, writing
\[
\mathcal{T}_{\Delta} := \overline{\mathcal{T}}_{\Delta,j},\quad \mathcal{T}_{\Delta}(Q) := \mathcal{T}_{\Delta} \cap \overline{\mathcal{T}}_{\Delta}(Q), \qquad Q \in \overline{\mathcal{Q}},
\]
we have
\begin{equation}\label{form68} \sum_{Q \in \overline{\mathcal{Q}}} |\mathcal{T}_{\Delta}(Q)| = |\{(Q,\mathbf{T}) \in \overline{\mathcal{Q}} \times \mathcal{T}_{\Delta} : \mathbf{T} \in \mathcal{T}_{\Delta}(Q)\}| \approx_{\delta} |\overline{\mathcal{Q}}| \cdot \overline{M}_{\Delta}. \end{equation}
We write $N_{\Delta} := 2^{j}$ for this index "$j$", so
\begin{equation}\label{form70}
|\mathcal{T}_0 \cap \mathbf{T}| \sim N_{\Delta}, \qquad \mathbf{T} \in \mathcal{T}_{\Delta}(Q) \subset \mathcal{T}_{\Delta}.
\end{equation}
Since  $\overline{M}_{\Delta} \sim |\overline{\mathcal{T}}_{\Delta}(Q)|\ge |\mathcal{T}_{\Delta}(Q)|$ for $Q\in\overline{\mathcal{Q}}$,  we infer from \eqref{form68} that there exists a further subset of $\mathcal{Q} \subset \overline{\mathcal{Q}}$ of cardinality $|\mathcal{Q}|\approx_{\delta} |\overline{\mathcal{Q}}|$ such that
\begin{equation}\label{form86} |\mathcal{T}_{\Delta}(Q)| \approx_{\delta} \overline{M}_{\Delta} \sim |\overline{\mathcal{T}}_{\Delta}(Q)|, \qquad Q \in \mathcal{Q}. \end{equation}

\begin{remark}\label{rem7} At this point, if we desire the extra property in Remark \ref{rem6}, small additional refinements are needed. Recall that $\bar{\Delta} \in [\Delta,1] \cap 2^{-\N}$, and we desire that every $\overline{\mathbf{T}} \in \mathcal{T}^{\bar{\Delta}}(\mathcal{T}_{\Delta}(Q))$ contains a common number of elements from $\mathcal{T}_{\Delta}(Q)$. If $Q$-dependence is allowed, this is a matter of very straightforward pigeonholing: the only cost is that the cardinality of $\mathcal{T}_{\Delta}(Q)$ will decrease by a factor of $\approx_{\Delta} 1$, and in particular \eqref{form70}-\eqref{form86} are not affected.

After this step has been accomplished individually for every $Q \in \mathcal{Q}$, we first reduce $\mathcal{Q}$ to a further subset $\mathcal{Q}'$ with $|\mathcal{Q}'| \approx_{\Delta} |\mathcal{Q}|$ so that $\overline{\mathbf{T}} \mapsto |\overline{\mathbf{T}} \cap \mathcal{T}_{\Delta}(Q)|$ only varies within a factor of $2$ for $Q \in \mathcal{Q}$. After this, we finally discard a few tubes from each intersection $\overline{\mathbf{T}} \cap \mathcal{T}_{\Delta}(Q)$ to achieve the desired precise constancy. \end{remark}

To comply with the definition of "niceness" in Claim \eqref{it-induction-iii}, we reduce the families $\mathcal{T}_{\Delta}(Q)$ so that they have common cardinality
\begin{displaymath} M_{\Delta} := \min \{|\mathcal{T}_{\Delta}(Q)| : Q \in \mathcal{Q}\} \approx_{\delta} \overline{M}_{\Delta}. \end{displaymath}
(If we know and want that $\overline{\mathbf{T}} \mapsto |\overline{\mathbf{T}} \cap \mathcal{T}_{\Delta}(Q)|$ is constant in the sense of Remark \ref{rem6}, this reduction has to be performed by deleting "blocks" of the form $\overline{\mathbf{T}} \cap \mathcal{T}_{\Delta}(Q)$.) Since $\overline{\mathcal{T}}_{\Delta}(Q)$ was a $(\Delta,s,C_{\Delta})$-set by \nref{T1}, also $\mathcal{T}_{\Delta}(Q)$ remains a $(\Delta,s,C_{\Delta})$-set, with constant $C_{\Delta} \approx_{\delta} C_1$. We now define
\begin{displaymath} \mathcal{T}_{\Delta} := \bigcup_{Q \in \mathcal{Q}} \mathcal{T}_{\Delta}(Q). \end{displaymath}
This finalises the definition of the family $\mathcal{T}_{\Delta}$ appearing in Claim \eqref{it-induction-iii}.


We next begin processing the families $\mathcal{T}_{0}(p)$ towards the families $\mathcal{T}(p)$. For $Q \in \mathcal{Q}$ fixed, let
\begin{equation}\label{app22} \overline{\mathcal{T}}(p) := \bigcup_{\mathbf{T} \in \mathcal{T}_{\Delta}(Q)} (\mathcal{T}_{0}(p) \cap \mathbf{T}), \qquad p \in \overline{\mathcal{P}}_{Q}. \end{equation}
Now, it is clear that \eqref{app21} holds, so the proof of Claim \eqref{it-induction-iii} is complete -- at least when we declare that the final families $\mathcal{T}(p)$ (to be finalised in Section \ref{ss:contruction-T_Q}) will be subfamilies of $\overline{\mathcal{T}}(p)$. To be precise, we also need to know that $\mathcal{Q}=\mathcal{D}_{\Delta}(\mathcal{P})$ - this will indeed be the case, see \eqref{form5} below.

The first issue with definition \eqref{app22} is that it is not guaranteed that $|\overline{\mathcal{T}}(p)| \approx_{\delta} M$ for all $p \in \overline{\mathcal{P}}_{Q}$ (the second issue is that to attain \eqref{app1}, we will need to refine $\overline{\mathcal{T}}(p)$ further; we will return to this in Section \ref{ss:contruction-T_Q}). Recall from \eqref{form86} that $|\mathcal{T}_{\Delta}(Q)| \approx_{\delta} |\overline{\mathcal{T}}_{\Delta}(Q)|$ for $Q \in \mathcal{Q}$. Since also $H_{Q} \approx_{\Delta} M \cdot |\overline{\mathcal{P}}_{Q}|/|\overline{\mathcal{T}}_{\Delta}(Q)|$ by \nref{T2}, we record that
\begin{align*}
\sum_{p\in\overline{\mathcal{P}}_{Q}} |\overline{\mathcal{T}}(p)| & \overset{\eqref{app22}}{=}
\sum_{p\in\overline{\mathcal{P}}_{Q}} \sum_{\mathbf{T}\in\mathcal{T}_{\Delta}(Q)} |\mathcal{T}_0(p)\cap\mathbf{T}| \\
&=\sum_{\mathbf{T} \in \mathcal{T}_{\Delta}(Q)} |\{(p,T) \in \overline{\mathcal{P}}_{Q} \times \mathcal{T}^{\delta} : T \in \mathcal{T}_{0}(p) \text{ and } T \subset \mathbf{T}\}|  \\
&\gtrapprox_{\delta} |\mathcal{T}_{\Delta}(Q)| \cdot H_{Q} \approx_{\delta} M \cdot |\overline{\mathcal{P}}_{Q}|.
\end{align*}
Since $|\overline{\mathcal{T}}(p)|\le |\mathcal{T}_0(p)| = M$ for all $p \in \overline{\mathcal{P}}_Q \subset \mathcal{P}_0$, this implies the existence of a subset $\mathcal{P}_{Q} \subset \overline{\mathcal{P}}_{Q}$ of cardinality $|\mathcal{P}_{Q}| \approx_{\delta} |\overline{\mathcal{P}}_{Q}|$ such that
\begin{equation}\label{form72} |\overline{\mathcal{T}}(p)| \approx_{\delta} M, \qquad p \in \mathcal{P}_{Q}. \end{equation}
We now define
\begin{equation}\label{form5}
\mathcal{P} := \bigcup_{Q \in \mathcal{Q}} \mathcal{P}_{Q} \quad \text{and} \quad \overline{\mathcal{T}} := \bigcup_{p \in \mathcal{P}} \overline{\mathcal{T}}(p), \end{equation}
and we note, thanks to \eqref{form72}, that Claim \eqref{it-induction-ii} is satisfied for $p \in \mathcal{P}$, and for the families $\overline{\mathcal{T}}(p)$ defined in \eqref{app22}. Since $\overline{\mathcal{T}}(p)$ will be refined once more in the sequel into $\mathcal{T}(p)$, we remark that Claim \eqref{it-induction-ii} remains valid as long $|\mathcal{T}(p)| \approx_{\delta} |\overline{\mathcal{T}}(p)|$, and we will make sure that this is the case. Similarly, Claim \eqref{it-induction-i} is clearly satisfied by the set $\mathcal{P}$. In the sequel, the family $\mathcal{D}_{\Delta}(\mathcal{P}) = \mathcal{Q}$ will remain intact, but the sets $\mathcal{P}_{Q}$ will be refined once more while maintaining $|\mathcal{P}_{Q}|\approx_{\delta}|\mathcal{P}_0\cap Q|$. Clearly this will not influence the validity of Claim \eqref{it-induction-i}.

We start proving Claim \eqref{it-induction-iv} and, concurrently, the lower bound \eqref{lower-bound-T} on $|\mathcal{T}_0|$.  We first note that
\begin{equation}\label{form80}
|\mathcal{T}_0| \geq |\mathcal{T}_{\Delta}| \cdot \min_{\mathbf{T}\in\mathcal{T}_{\Delta}} |\mathcal{T}_0 \cap \mathbf{T}| \stackrel{\eqref{form70}}{\sim} |\mathcal{T}_{\Delta}| \cdot N_{\Delta}.
\end{equation}
We fix $Q \in \mathcal{Q}$ for the rest of the argument, and define the collection of $\delta$-tubes
\begin{displaymath} \mathcal{T}(Q) := \bigcup_{p \in \mathcal{P}_{Q}} \overline{\mathcal{T}}(p). \end{displaymath}

We next observe, using $|\mathcal{T}_{\Delta}(Q)| = M_{\Delta}$  and, recalling \eqref{app22}, that
\begin{equation}\label{form71}
|\mathcal{T}(Q)| \leq \sum_{\mathbf{T} \in \mathcal{T}_{\Delta}(Q)} |\mathcal{T}_{0} \cap \mathbf{T}| \stackrel{\eqref{form70}}{\lesssim} M_{\Delta} \cdot N_{\Delta}.
\end{equation}
We adopt the notation $\mathcal{P}^Q := S_Q(\mathcal{P}_Q) = S_Q(\mathcal{P}\cap Q) \subset \mathcal{D}_{\delta/\Delta}([0,1)^{2})$. During the remainder of the proof, we will construct a family $\mathcal{T}_Q$ of $(\delta/\Delta)$-tubes such that (after a final refinement of $\mathcal{P}_Q$) the pair $(\mathcal{P}^Q,\mathcal{T}_Q)$ is a $(\delta/\Delta,s,C_Q, M_Q)$-nice configuration, for some $C_{Q},M_{Q} \geq 1$, satisfying Claim \eqref{it-induction-iv}, in particular \eqref{app1}. We will also show that
\begin{equation}\label{form46} |\mathcal{T}(Q)| \gtrapprox_{\delta} \frac{|\mathcal{T}_{Q}|}{M_{Q}}\cdot M .\end{equation}
This will finish the proof, since combining \eqref{form80}, \eqref{form71}, and \eqref{form46} yields \eqref{lower-bound-T}.

\subsection{Construction of $\mathcal{T}_Q$ and proof of \eqref{form46}}\label{ss:contruction-T_Q} We abbreviate $\bar{\delta}= \delta/\Delta$. The main task remaining is to define the family $\mathcal{T}_Q \subset \mathcal{T}^{\bar{\delta}}$. We know that each square $p \in \mathcal{P}_{Q}$ intersects each $\delta$-tube in the $(\delta,s,C)$-set $\overline{\mathcal{T}}(p)$ defined in \eqref{app22}, with $C\approx_\delta C_1$. However, the separation of these tubes is of the order "$\delta$", not "$\delta/\Delta$", and this causes a need for further processing.

Fix $p \in \mathcal{P}_{Q}$. If $\mathbf{T}_{\bar{\delta}} \in \mathcal{T}^{\bar{\delta}}$ with $\mathbf{T}_{\bar{\delta}} \cap p \neq \emptyset$, the family $\overline{\mathcal{T}}(p) \cap \mathbf{T}_{\bar{\delta}}$ is called a \emph{tube packet}. A generic tube packet is denoted $\Xi(p)$, thus $\Xi(p) = \overline{\mathcal{T}}(p) \cap \mathbf{T}_{\bar{\delta}}$ for some $\mathbf{T}_{\bar{\delta}} \in \mathcal{T}^{\bar{\delta}}$. Since the tubes in a fixed tube packet $\Xi(p)$ have a common ancestor in $\mathcal{T}^{\bar{\delta}}$, the slope set $\sigma(\Xi(p))$ is contained in a dyadic interval of length $\bar{\delta}$, determined by $\mathbf{T}_{\bar{\delta}}$.

Every tube in $\overline{\mathcal{T}}(p)$ lies in precisely one tube packet, and every tube packet $\Xi(p)$ satisfies $|\Xi(p)| \in \{0,\ldots,\delta^{-2}\}$. Therefore, we may find dyadic numbers $m(p) \in \{1,\ldots,\delta^{-2}\}$ and $M(p) \in \{1,\ldots,|\overline{\mathcal{T}}(p)|\}$, and a collection of tube packets
\begin{displaymath} \Xi_{1}(p),\ldots,\Xi_{M(p)}(p) \subset \overline{\mathcal{T}}(p) \end{displaymath}
such that
\begin{displaymath} |\Xi_{j}(p)| \sim m(p), \, 1 \leq j \leq M(p), \quad \text{and} \quad m(p) \cdot M(p) \approx_{\delta} |\overline{\mathcal{T}}(p)| \stackrel{\eqref{form72}}{\approx_{\delta}} M. \end{displaymath}
By pigeonholing again, we may find a number $M_Q\in \N$, and a subset of $\mathcal{P}_{Q}$ of cardinality $\approx_\delta |\mathcal{P}_{Q}|$ (we keep the notation $\mathcal{P}_{Q}$) such that $M(p) \sim M_Q$ for all $p \in \mathcal{P}_{Q}$. For $p \in \mathcal{P}_{Q}$, we finally define 
\begin{equation}\label{app7} \mathcal{T}(p) := \bigcup_{j = 1}^{M(p)} \Xi_{j}(p). \end{equation}
Note that this (final) family $\mathcal{T}(p)$ still satisfies $|\mathcal{T}(p)| \approx_{\delta} M$, so Claim \eqref{it-induction-ii} was not violated. After this point, the family $\mathcal{P}_{Q}$ will no longer be refined, so the final definition of $\mathcal{P}$ is now given by \eqref{form5}, with the current definition of $\mathcal{P}_{Q}$.

With these conventions, $|\mathcal{P}_{Q}| \approx_{\delta} |\mathcal{P}_0 \cap Q|$, and
\begin{equation}\label{form47} |\Xi_{j}(p)| \approx_{\delta} |\mathcal{T}(p)|/M_Q\approx_{\delta} M/M_Q=: m_Q \approx_{\delta} m(p), \quad p \in \mathcal{P}_{Q}, \, 1 \leq j \leq M_Q. \end{equation}
We adopt the notational convention that squares in $\mathcal{P}^{Q} = S_{Q}(\mathcal{P}_{Q}) \subset \mathcal{D}_{\bar{\delta}}$ are denoted "$q$". We now intend to cook up $(\bar{\delta},s)$-sets of tubes $\mathcal{T}_{Q}(q)$ incident to the squares $q \in \mathcal{P}^{Q}$. For $q = S_{Q}(p)$, with $p \in \mathcal{P}_{Q}$, the most obvious attempt might be to define $\mathcal{T}_{Q}(q) := S_{Q}(\mathcal{T}(p))$, since all the sets $S_{Q}(T)$, $T \in \mathcal{T}(p)$, intersect $q = S_{Q}(p)$. The worst way in which this fails is that the slopes of the tubes in $\mathcal{T}(p)$ are $\delta$-separated, whereas we desire that the slopes of $\mathcal{T}_{Q}(q)$ are $\bar{\delta}$-separated; applying the homothety $S_{Q}$ does not rescue the situation, as homotheties preserve slopes.

To fix the issue, we start by performing a "thinning" procedure to the families $\mathcal{T}(p)$. Namely, for every $p \in \mathcal{P}_{Q}$, we define a new tube family $\mathcal{T}'(p)$ with the heuristic idea to select a single tube from each tube packet $\Xi_{1}(p),\ldots,\Xi_{M_Q}(p)$. A slightly different technical implementation is more tractable. Fix a tube packet $\Xi = \Xi_{j}(p)$, $1 \leq j \leq M_Q$, and let $\mathbf{T}_{j} \in \mathcal{T}^{\bar{\delta}} $ be the dyadic ancestor of the tubes in $\Xi$. Let $T_{j} \in \mathcal{T}^{\delta}$ be a tube with the properties
\begin{equation} \label{eq:properties-Tj}
T_{j} \cap p \neq \emptyset, \quad T_{j} \subset \mathbf{T}_{j}, \quad \text{and} \quad \sigma(T_{j}) = \sigma(\mathbf{T}_{j}) \in (\bar{\delta} \cdot \Z) \cap [-1,1). \end{equation}
Thus, $T_{j}$ may, or may not, be an element of $\Xi$; we will only need that $T_{j}$ has the same ancestor in $\mathcal{T}^{\bar{\delta}}$ as all the tubes in $\Xi_{j}(p)$. We say that the tube packet $\Xi_{j}(p)$ is \emph{represented by $T_{j}$.} Letting $1 \leq j \leq M_Q$ vary, we thus obtain a collection $\mathcal{T}'(p)$ of dyadic $\delta$-tubes, whose slopes are $\bar{\delta}$-separated, and in fact
\begin{equation} \label{form78} 
\sigma(p) := \sigma(\mathcal{T}'(p)) \subset (\bar{\delta} \cdot \Z) \cap [-1,1), \qquad p \in \mathcal{P}_{Q}. 
\end{equation}
Another useful property, needed to establish \eqref{app1}, is that all the slopes of the tubes in $\mathcal{T}(p)$ are contained in the $\bar{\delta}$-intervals around the slopes of $\mathcal{T}'(p)$. In fact, more precisely, the slope set $\sigma(p)$ consists exactly of the left endpoints of those dyadic $\bar{\delta}$-intervals which contain at least one slope from $\sigma(\mathcal{T}(p)) \subset \delta \cdot \Z$. In symbols,
\begin{equation}\label{app8} \mathcal{D}_{\bar{\delta}}\left[\sigma(p)\right] = \mathcal{D}_{\bar{\delta}}\left[\sigma(\mathcal{T}(p))\right], \qquad p \in \mathcal{P}_{Q}. \end{equation}
Note the resemblance with \eqref{app1}.

For $p \in \mathcal{P}_{Q}$ fixed, we claim that the slope set $\sigma(p)$ is a $(\bar{\delta},s,C_Q)$-set with $C_Q\approx_\delta C_1$. To see this, fix a dyadic interval $I \subset \R$ of length $|I| \geq \bar{\delta}$, and let $I_{1},\ldots,I_{k} \subset I$ be an enumeration of those dyadic $\bar{\delta}$-intervals which have non-empty intersection with $I \cap \sigma(p)$. Since each $I_{i}$ intersects $\sigma(p)$ at some point $\sigma(T_{j}) \in \bar{\delta} \cdot \Z$, with $T_{j} \in \mathcal{T}'(p)$, we see that $I_{i} = [\sigma(T_{j}),\sigma(T_{j}) + \bar{\delta})$, and $I_{i}$ contains all the slopes of the tube packet $\Xi_{j}(p)$ represented by $T_{j}$: in other words $\sigma(\Xi_{j}(p)) \subset I_{i}$. Since $|\sigma(\Xi_{j}(p))| \sim |\Xi_{j}(p)| \approx_{\delta} m_Q$ by \eqref{form47} (and since all the tubes in $\Xi_{j}(p)$ are incident to $p$, see Lemma \ref{tubesSlopes}), we may use the knowledge that $\sigma(\mathcal{T}(p))$ is a $(\delta,s,C)$-set of cardinality $\approx_{\delta} M$ for $C \approx_\delta C_1$ as follows:
\begin{equation}\label{form48} |I \cap \sigma(p)| \approx_{\delta} \tfrac{1}{m_Q} \cdot |I \cap \sigma(\mathcal{T}(p))| \lessapprox_{\delta} \tfrac{1}{m_Q} \cdot C \cdot M \cdot |I|^{s} \stackrel{\eqref{form47}}{=} C \cdot  M_Q\cdot |I|^{s}. \end{equation}
Since $|\sigma(p)| \sim M_Q$, we infer from \eqref{form48} that $\sigma(p)$ is a $(\bar{\delta},s,C_Q)$-set with $C_Q\approx_\delta C_1$.

Now that the slopes between the tubes in $\mathcal{T}'(p)$ are $\bar{\delta}$-separated, there is hope that the sets "$\mathcal{T}_{Q}(q) := S_{Q}(\mathcal{T}'(p))$" are, roughly speaking, $(\bar{\delta},s)$-sets of dyadic $\bar{\delta}$-tubes intersecting $q = S_{Q}(p)$. A minor issue is that the homothetic image of a dyadic tube is not exactly a dyadic tube. Instead, the following holds, as can be verified by straightforward computation: write $Q = [x_{0},x_{0} + \Delta) \times [y_{0},y_{0} + \Delta)$, and let $T = \mathbf{D}(I \times J) \in \mathcal{T}^{\delta}$ be arbitrary, with $|I|,|J| = \delta$, $I \times J \subset [-1,1) \times \R$. Then,
\begin{displaymath} S_{Q}(T) \subset \mathbf{D}\left(I \times \tfrac{x_{0}I + J - y_{0}}{\Delta} \right).  \end{displaymath}
Here $\mathbf{J} = (x_{0}I + J - y_{0})/\Delta$ is an interval of length $|\mathbf{J}| \sim \delta/\Delta = \bar{\delta}$. It follows that $S_{Q}(T)$ may be covered by a family $\overline{\mathcal{S}}_{Q}(T)$ consisting of $|\overline{\mathcal{S}}_{Q}(T)| \sim 1$ sets of the form $\mathbf{D}(I \times \mathbf{J}')$, where $\mathbf{J}' \in \mathcal{D}_{\bar{\delta}}(\R)$. The sets $\mathbf{D}(I \times \mathbf{J}')$ are not quite elements of $\mathcal{T}^{\bar{\delta}}$, because $I \notin \mathcal{D}_{\bar{\delta}}(\R)$.

We specialise the discussion to the situation where $T \in \mathcal{T}'(p)$ for some $p \in \mathcal{P}_{Q}$. Then $\sigma(T) \in \sigma(p) \subset (\bar{\delta} \cdot \Z) \cap [-1,1)$. Therefore, if $T = \mathbf{D}(I \times J)$, with $I = [\sigma,\sigma + \delta)$, we have $\mathbf{I} = [\sigma,\sigma + \bar{\delta}) \in \mathcal{D}_{\bar{\delta}}$. It follows that the sets $\mathbf{D}(I \times \mathbf{J}') \in \overline{\mathcal{S}}_{Q}(T)$, defined above, are individually covered by the dyadic tubes $\mathbf{D}(\mathbf{I} \times \mathbf{J}') \in \mathcal{T}^{\bar{\delta}}$. Now, the collection $\mathcal{S}_{Q}(T)$, formed by the sets $\mathbf{D}(\mathbf{I} \times \mathbf{J}')$, consists of $\sim 1$ elements of $\mathcal{T}^{\bar{\delta}}$, and its union covers $S_{Q}(T)$. Moreover,
\begin{equation}\label{form11} \sigma(\mathcal{S}_{Q}(T)) = \{\sigma\} = \{\sigma(T)\}, \qquad T \in \mathcal{T}'(p). \end{equation}

Now we are prepared to define the family $\mathcal{T}_{Q}$. We write
\begin{displaymath} \mathcal{T}' := \bigcup_{p \in \mathcal{P}_{Q}} \mathcal{T}'(p) \quad \text{and} \quad \mathcal{T}_{Q} := \cup \{\mathcal{S}_{Q}(T) : T \in \mathcal{T}'\}. \end{displaymath}
We also define the following sub-families $\mathcal{T}_{Q}(q) \subset \mathcal{T}_{Q}$, for $q=S_{Q}(p) \in \mathcal{P}^{Q}$. By \eqref{eq:properties-Tj}, the sets $S_{Q}(T)$, $T \in \mathcal{T}'(p)$, intersect $q$. Consequently, at least one of the dyadic $\bar{\delta}$-tubes in $\mathcal{S}_{Q}(T)$ also intersects $q$, and we include this tube in $\mathcal{T}_{Q}(q)$. The $(\bar{\delta},s,C)$-set property of $\mathcal{T}_{Q}(q)$ then follows from \eqref{form48}, \eqref{form11} and Corollary \ref{cor1}.

We now finish the proof of \eqref{app1}. Fix $p \in \mathcal{P}_{Q} = \mathcal{P} \cap Q$. By the definition of $\mathcal{T}_{Q}(S_{Q}(p))$ just above, the slope set $\sigma[\mathcal{T}_{Q}(S_{Q}(p))] \subset \bar{\delta} \cdot \Z$ coincides with $\sigma(\mathcal{T}'(p)) = \sigma(p)$. Therefore, \eqref{app1} is a consequence of \eqref{app8}.

It remains to show that \eqref{form46} holds.  Note that every tube $T' \in \mathcal{T}'$ represents at least one tube packet $\Xi(T') \subset \mathcal{T}(Q)$ with $|\Xi(T')| \approx_{\delta} m_Q$, recalling \eqref{form47}. If we knew that the tube packets represented by distinct tubes in $\mathcal{T}'$ are disjoint (this turns out to be "true enough"), we could easily complete our estimate as follows:
\begin{equation} \label{form9} |\mathcal{T}(Q)| \gtrsim \sum_{T' \in \mathcal{T}'} |\Xi(T')|  \stackrel{\eqref{form47}}{\gtrapprox_{\delta}} |\mathcal{T}'| \cdot \frac{M}{M_Q} \gtrsim |\mathcal{T}_Q| \cdot \frac{M}{M_Q},
\end{equation}
as desired in \eqref{form46}.

What is literally true, and still sufficient for the estimate above, is that there exists a subset $\mathcal{T}'' \subset \mathcal{T}'$ of cardinality $|\mathcal{T}''| \sim |\mathcal{T}'|$ such that the tube packets represented by the elements in $\mathcal{T}''$ are disjoint. The proof of the whole proposition is completed by justifying this statement.

Fix $p \in \mathcal{P}$ and $T' \in \mathcal{T}'(p)$, and write $\Xi(T') \in \{\Xi_{1}(p),\ldots,\Xi_{M_Q}(p)\}$ for the tube packet represented by $T'$. Then, note that all the tubes from $\Xi(T')$ have roughly the same intersection with $Q$, and this intersection roughly agrees with $T' \cap Q$. More precisely, there exists an absolute constant $C \geq 1$ such that, denoting the $\delta$-neighborhood of a set $A$ by $A_{\delta}$, the following holds:
\begin{equation}\label{form25} p \in \mathcal{P}_{Q} \text{ and } T' \in \mathcal{T}'(p) \quad \Longrightarrow \quad T_{2\delta} \cap Q \subset T_{C\delta}' \text{ for all } T \in \Xi(T'). \end{equation}
Indeed, this follows from the facts that $T,T'$ are $\delta$-tubes intersecting $p$ and their slopes are $(\delta/\Delta)$-close. This observation yields a criterion for checking that two tube packets are disjoint. With the constant "$C$" as in \eqref{form25}, we define that two tubes $T,T' \in \mathcal{T}^{\delta}$ are \emph{separated} if
\begin{displaymath} \sigma(T) \neq \sigma(T') \quad \text{or} \quad T_{C\delta} \cap T_{C\delta}' \cap [0,1)^{2} = \emptyset. \end{displaymath}

\begin{lemma}\label{lemma3} If $T_{1}',T_{2}' \in \mathcal{T}'$ are separated, then $\Xi(T_{1}') \cap \Xi(T_{2}') = \emptyset$. \end{lemma}

\begin{proof} If $T_{1}',T_{2}'$ are separated for the reason that $\sigma(T_{1}') \neq \sigma(T_{2}')$, then the slopes of all the tubes in $\Xi(T_{1}')$ are distinct from all the slopes in $\Xi(T_{2}')$: this follows from $\sigma(T_{1}'),\sigma(T_{2}') \in \bar{\delta} \cdot \Z$ (recall \eqref{form78}), and $\sigma(\Xi(T_{j}')) \subset [\sigma(T_{j}'),\sigma(T_{j}') + \bar{\delta})$ for $j \in \{1,2\}$.

Let us then assume that $\sigma(T_{1}') = \sigma(T_{2}')$, but the second separation condition holds. Since $T_{1}',T_{2}' \in \mathcal{T}'$, one may find squares $p_{1},p_{2} \in \mathcal{P}_{Q}$ such that $T_{j}' \in \mathcal{T}'(p_{j})$ for $j \in \{1,2\}$. If $T \in \Xi(T_{1}') \cap \Xi(T_{2}')$, then $T \cap p_{1} \neq \emptyset \neq T \cap p_{2}$, since $T \in \mathcal{T}(p_{1}) \cap \mathcal{T}(p_{2})$. Thus,
\begin{displaymath} p_{1},p_{2} \subset T_{2\delta} \cap Q \subset (T_{1}')_{C\delta} \cap (T_{2}')_{C\delta} \cap [0,1)^{2}. \end{displaymath}
by \eqref{form25}. This violates the assumption that $T_{1}',T_{2}'$ are separated. Hence  $\Xi(T_{1}') \cap \Xi(T_{2}') = \emptyset$ also in this case.  \end{proof}
It remains to find a separated collection $\mathcal{T}'' \subset \mathcal{T}'$ of cardinality $|\mathcal{T}''| \sim |\mathcal{T}'|$:

\begin{lemma}\label{lemma4} The family $\mathcal{T}'$ contains a sub-collection of comparable cardinality whose elements are separated. \end{lemma}
\begin{proof} Recall that $\sigma(\mathcal{T}') \subset \bar{\delta} \cdot \Z$ by \eqref{form78}. By the definition of "separation" it suffices to prove the following claim for any fixed $\sigma \in \bar{\delta} \cdot \Z$: the collection $\mathcal{T}'(\sigma) := \{T' \in \mathcal{T}' : \sigma(T') = \sigma\}$ contains a separated set of cardinality $\sim |\mathcal{T}'(\sigma)|$. We leave the algebra to the reader, but the idea is the following: one verifies that if $T_{j}' = \mathbf{D}([\sigma,\sigma + \delta) \times [b_{j},b_{j} + \delta)) \in \mathcal{T}'(\sigma)$, $j \in \{1,2\}$, are two tubes with $(T'_{1})_{C\delta} \cap (T'_{2})_{C\delta} \cap [0,1)^{2} \neq \emptyset$, then $|b_{1} - b_{2}| \leq C'\delta$ for some $C' \sim C$. Consequently, a separated subset of $\mathcal{T}'(\sigma)$ can be found by choosing a $(C'\delta)$-separated subset of $\{b \in \delta \cdot \Z : \mathbf{D}([\sigma,\sigma + \delta) \times [b,b + \delta)) \in \mathcal{T}'(\sigma)\}$. \end{proof}

As we discussed around \eqref{form9}, the proof of \eqref{form46}, and hence the proof of the proposition, is now complete. \end{proof}

\section{An improved incidence estimate for regular sets}\label{s:improvedIncidence}

\begin{definition}\label{def:regularity} Let $\delta \in 2^{-\N}$ be a dyadic number such that also $\delta^{1/2} \in 2^{-\N}$. Let $C,K > 0$, and let $0 \leq s \leq d$. A non-empty set $\mathcal{P} \subset \mathcal{D}_{\delta}$ is called $(\delta,s,C,K)$-regular if $\mathcal{P}$ is a $(\delta,s,C)$-set, and moreover
\begin{displaymath} |\mathcal{P}|_{\delta^{1/2}} \leq K \cdot \delta^{-s/2}. \end{displaymath}
\end{definition}

\begin{thm}\label{t:improvedIncidence} Given $s \in (0,1)$ and $t \in (s,2]$, there exists $\epsilon = \epsilon(s,t) > 0$ such that the following holds for small enough $\delta \in 2^{-\N}$, depending only on $s,t$. Let $\mathcal{P} \subset \mathcal{D}_{\delta}$ be a $(\delta,t,\delta^{-\epsilon},\delta^{-\epsilon})$-regular set. Assume that for every $p \in \mathcal{P}$, there exists a $(\delta,s,\delta^{-\epsilon})$-set of dyadic tubes $\mathcal{T}(p) \subset \mathcal{T}^{\delta}$ such that $T \cap p \neq \emptyset$ for all $T \in \mathcal{T}(p)$. Then,
\begin{equation}\label{form8} |\mathcal{T}| \geq \delta^{-2s - \epsilon}, \end{equation}
where $\mathcal{T} = \bigcup_{p \in P} \mathcal{T}(p)$.
\end{thm}

This theorem is a variant of  \cite[Theorem 3.12]{Orponen20}, with two essential changes. The first one is that in \cite[Theorem 3.12]{Orponen20}, the conclusion \eqref{form8} is replaced by the following alternative: either
\begin{equation}\label{form84} |\mathcal{T}| \geq \delta^{-2s - \epsilon} \quad \text{or} \quad |\mathcal{T}|_{\delta^{1/2}} \geq \delta^{-s - \epsilon}. \end{equation}
It turns out that this superficially weaker statement can be used to deduce the stronger one given in Theorem \ref{t:improvedIncidence}: this is the content of the present section. The second essential difference is that \cite[Theorem 3.12]{Orponen20} only considers the case $t = 1$. While the changes required in the proof are not difficult, they still affect the argument so substantially that we have decided to provide all the details in Appendix \ref{appA}, see Theorem \ref{mainAppendix}.

During the proof of Theorem \ref{t:improvedIncidence}, will write "$A \lessapprox B$" if $A \leq C\delta^{-C\epsilon}B$, where $C \geq 1$ is an absolute constant, and "$\epsilon$" is a sufficiently small constant to be determined (it will end up being the one from the statement of Theorem \ref{t:improvedIncidence}). The notation $A \approx B$ means that $A \lessapprox B \lessapprox A$. Factors which are logarithmic in $\delta$ will also be hidden by the "$\lessapprox$" notation: $A \leq B \cdot (\log \tfrac{1}{\delta})^{C}$ will be abbreviated to $A \lessapprox B$. We will also abbreviate the terms "$(\Delta,u,\delta^{-C\epsilon})$-set" and "$(\Delta,u,\delta^{-C\epsilon},\delta^{-C\epsilon})$-regular set" to "$(\Delta,u)$-set" and "$(\Delta,u)$-regular set", respectively, for $\Delta \in \{\delta,\delta^{1/2}\}$ and $u \in \{s,t\}$.

\begin{proof}[Proof of Theorem \ref{t:improvedIncidence}]
Since we aim to use Proposition \ref{p:induction-on-scales}, we will denote $\mathcal{P}_0 := \mathcal{P}$ and $\mathcal{T}_0 := \mathcal{T}$, reserving $\mathcal{P},\mathcal{T}$ for the objects provided by the proposition. By Lemma \ref{lem:thin-delta-subset}, we may assume that $|\mathcal{T}_0(p)| \approx \delta^{-s}$ for all $p \in \mathcal{P}$. Pigeonholing and passing to a subset $\mathcal{P}'_0\subset \mathcal{P}_0$ with $|\mathcal{P}'_0|\approx |\mathcal{P}_0|$ (which we continue to denote by $\mathcal{P}_0$), we may further assume that $|\mathcal{T}_0(p)|=M\approx \delta^{-s}$ for all $p\in \mathcal{P}_0$. By another application of Lemma \ref{lem:thin-delta-subset}, we may reduce to the case where $|\mathcal{P}_0| \approx \delta^{-t}$ (passing to subsets will not break any upper bound on $|\mathcal{P}_0|_{\delta^{1/2}}$, so we can also retain the $(\delta,t,\delta^{-\epsilon},\delta^{-\epsilon})$-regularity of $\mathcal{P}_0$).

Then the statement looks very much like Theorem \ref{mainAppendix}: the only difference is that, according to Theorem \ref{mainAppendix}, either \eqref{form8} holds, or then
\begin{displaymath} |\mathcal{T}_{0}|_{\delta^{1/2}} \geq \delta^{-s - \epsilon}. \end{displaymath}
The plan is to use Proposition \ref{p:induction-on-scales}, applied to a suitable refinement of $\mathcal{P}_0$, to show that either of the alternatives presented by Theorem \ref{mainAppendix} can be parlayed into the claim of Theorem \ref{t:improvedIncidence}.

Let $\mathcal{Q}_{0} := \mathcal{D}_{\delta^{1/2}}(\mathcal{P}_{0})$ (the minimal cover of $\mathcal{P}_{0}$ by dyadic $\delta^{1/2}$-squares). Then $|\mathcal{Q}_{0}| \leq \delta^{-t/2 - \epsilon}$ by the assumption that $\mathcal{P}_0$ is $(\delta,t,\delta^{-\epsilon},\delta^{-\epsilon})$-regular. We call a square $Q \in \mathcal{Q}_{0}$ \emph{light} if $|\{p \in \mathcal{P}_{0} : p \subset Q\}| \leq \delta^{-t/2 + 5\epsilon}$. From $|\mathcal{Q}_{0}| \leq \delta^{-t/2 - \epsilon}$ and $|\mathcal{P}_{0}| \geq \delta^{-t +\epsilon}$, it follows that $\leq \tfrac{1}{2} \cdot |\mathcal{P}_{0}|$ squares in $\mathcal{P}_{0}$ are covered by the union of the light squares in $\mathcal{Q}_{0}$. We now discard the light squares from $\mathcal{Q}_{0}$, and also their contents from $\mathcal{P}_{0}$. We keep the notation $\mathcal{P}_{0}$ and $\mathcal{Q}_{0}$ for the remaining squares. Then, by definition, the (remaining) squares in $\mathcal{Q}_{0}$ are \emph{heavy}: $|\{p \in \mathcal{P}_{0} : p \subset Q\}| \gtrapprox \delta^{-t/2}$ for all $Q \in \mathcal{Q}_{0}$.

After these initial refinements we apply Proposition \ref{p:induction-on-scales} to the $(\delta,s,\delta^{-\epsilon},M)$-nice configuration $(\mathcal{P}_0,\mathcal{T}_0)$, at the scale $\Delta=\delta^{1/2}$. Let
\[
\mathcal{P} \subset \mathcal{P}_{0},\quad\mathcal{T} \subset \mathcal{T}_{0}, \quad \mathcal{T}_{\delta^{1/2}}\subset\mathcal{T}^{\delta^{1/2}},\quad \mathbf{M}:= M_{\delta^{1/2}} \geq 1,\quad\mathbf{C} := C_{\delta^{1/2}} \approx 1
\]
be provided by the proposition, and write $\mathcal{Q} := \mathcal{D}_{\delta^{1/2}}(\mathcal{P})$. We remark that by Proposition \ref{p:induction-on-scales}\eqref{it-induction-i}, we have $|\mathcal{Q}| \approx |\mathcal{Q}_{0}|$ and $|\mathcal{P} \cap Q| \approx |\mathcal{P}_{0} \cap Q| \approx \delta^{-t/2}$. Thus, the cardinality of $\mathcal{Q}_{0}$ was not substantially reduced when passing to $\mathcal{Q}$, and the remaining squares in $\mathcal{Q}$ remain (essentially) as heavy relative to $\mathcal{P}$ as those in $\mathcal{Q}_{0}$ were relative to $\mathcal{P}_{0}$.

By Claim \eqref{it-induction-iii} in Proposition \ref{p:induction-on-scales}, $(\mathcal{Q},\mathcal{T}_{\delta^{1/2}})$ is a $(\delta^{1/2},s,\mathbf{C},\mathbf{M})$-nice configuration. This means, by definition, that for every $Q \in \mathcal{Q}$, there exists a non-empty $(\delta^{1/2},s,\mathbf{C})$-set $\mathcal{T}_{\delta^{1/2}}(Q) \subset \mathcal{T}_{\delta^{1/2}}$ of cardinality $\mathbf{M} = |\mathcal{T}_{\delta^{1/2}}(Q)| \gtrapprox \delta^{-s/2}$ such that $\mathbf{T} \cap Q \neq \emptyset$ for all $\mathbf{T} \in \mathcal{T}_{\delta^{1/2}}(Q)$. Thus, by Corollary \ref{prop5} applied at scale $\delta^{1/2}$, we obtain
\begin{equation}\label{form4} |\mathcal{T}_{\delta^{1/2}}| \gtrapprox \mathbf{M}\delta^{-s/2} \cdot (\mathbf{M}\delta^{s/2})^{\frac{t - s}{1 - s}}. \end{equation}
Next, for $Q\in \mathcal{Q}$, let $\mathcal{T}_Q, M_Q, C_Q$ be the objects in Proposition \ref{p:induction-on-scales}, Claim \eqref{it-induction-iv}.  In other words, for the homothety $S_{Q} \colon Q \to [0,1)^{2}$, the set $(S_{Q}(\mathcal{P}\cap Q),\mathcal{T}_{Q})$ is a $(\delta^{1/2},s,C_{Q},M_{Q})$-nice configuration, where $C_{Q} \approx_{\delta} 1$, and
\begin{displaymath}
 S_{Q}(\mathcal{P}\cap Q) = \{S_{Q}(p) : p \in \mathcal{P},p\subset  Q\} \subset \mathcal{D}_{\delta^{1/2}}.
\end{displaymath}
Since $|\mathcal{P}\cap Q| \approx |\mathcal{P}_0\cap Q| \approx \delta^{-t/2}$ for all $Q \in \mathcal{Q}$, as we already observed above, one can check that $S_{Q}(\mathcal{P}\cap Q)$ is a $(\delta^{1/2},t)$-set. Therefore, Corollary \ref{prop5} implies that
\begin{equation} \label{eq:improved-incidence-2}
|\mathcal{T}_Q| \gtrapprox M_Q\cdot \delta^{-s/2}.
\end{equation}

After these preliminaries, we apply Theorem \ref{mainAppendix}, or \eqref{form84} to the family $\mathcal{P}$, and the $(\delta,s)$-sets $\mathcal{T}(p)$, for $p \in \mathcal{P}$, produced by Claim \eqref{it-induction-ii} in Proposition \ref{p:induction-on-scales}. Note that since only heavy squares were retained in $\mathcal{Q}_{0}$, it follows from Claim \eqref{it-induction-i} in Proposition \ref{p:induction-on-scales} that $|\mathcal{P}|\gtrapprox_{\delta}|\mathcal{P}_0|$ and hence $\mathcal{P}$ is a $(\delta,t)$-regular set. Let $C$ be a large constant, depending on $s$ and $t$, to be determined in a moment. Provided $\epsilon$ is small enough in terms of $s,t$ and $C$, Theorem \ref{mainAppendix}, or \eqref{form84}, implies that either
\begin{displaymath} |\mathcal{T}_0| \geq |\mathcal{T}| \geq \delta^{-2s - C\epsilon} \quad \text{or} \quad |\mathcal{T}_{\delta^{1/2}}| = |\mathcal{T}|_{\delta^{1/2}}  \geq \delta^{-s - C\epsilon}. \end{displaymath}
The first alternative is what we seek to prove, so suppose that the second alternative holds. If $\mathbf{M} \ge \delta^{-s/2- C\epsilon/2}$, then we get from \eqref{form4} that there is $C_1\gtrsim_{s,t} C$ such that
\begin{equation} \label{eq:improved-incidence-1}
|\mathcal{T}_{\delta^{1/2}}| \ge \mathbf{M} \delta^{-s/2-C_1 \epsilon}.
\end{equation}
If, on the other hand, $\mathbf{M} <\delta^{-s/2- C\epsilon/2}$, then \eqref{eq:improved-incidence-1} also holds by the assumption $|\mathcal{T}_{\delta^{1/2}}| \geq \delta^{-s - C\epsilon}$.

We finally deduce from \eqref{lower-bound-T}, \eqref{eq:improved-incidence-2}, and \eqref{eq:improved-incidence-1} that, for any $Q \in \mathcal{Q}$,
\[
|\mathcal{T}_{0}|  \gtrapprox \frac{|\mathcal{T}_{\delta^{1/2}}|}{\mathbf{M}} \cdot \frac{|\mathcal{T}_{Q}|}{M_{Q}} \cdot M \gtrapprox M \cdot \delta^{-s- C_1 \epsilon} .
\]
Since $M\gtrapprox \delta^{-s}$, we conclude that if $C$ (and therefore $C_1$) is taken sufficiently large in terms of $s,t$ only, then $|\mathcal{T}_0| \ge \delta^{-2s-\epsilon}$. The proof of Theorem \ref{t:improvedIncidence} is complete. \end{proof}

\section{Combining incidence estimates from multiple scales}\label{s:combiningEstimates}

Recall the notation $\mathcal{D}_{\Delta}(P) = \{ Q\in\mathcal{D}_{\Delta}: P \cap Q \neq\emptyset\}$ for $P \subset [0,1)^{2}$ and $\Delta \in 2^{-\N}$, and $S_{Q}$ for the homothety mapping $Q$ to $[0,1)^2$.

\begin{definition} Let $0 < \delta < \Delta \leq 1$ be dyadic numbers, and $P \subset [0,1)^{2}$. Let also $0 \leq s \leq 2$ and $C > 0$.

\begin{enumerate}
\item We say that \emph{$P$ is an $(s,C)$-set between the scales $\delta$ and $\Delta$} if $S_{Q}(P\cap Q) \subset [0,1)^2$ is a $(\delta/\Delta,s,C)$-set for all $Q \in \mathcal{D}_{\Delta}(P)$.
\item We say that \emph{$P$ is $(s,C,K)$-regular between the scales $\delta$ and $\Delta$} if $S_{Q}(P\cap Q)$ is a $(\delta/\Delta,s,C)$-set such that
\[
|S_Q (P\cap Q)|_{(\delta/\Delta)^{1/2}} \le K (\delta/\Delta)^{-s/2}
\]
for all $Q \in \mathcal{D}_{\Delta}(P)$.
\end{enumerate}
\end{definition}

\begin{definition}\label{def:uniformity}
Let $n \geq 1$, and let
\begin{displaymath} \delta = \Delta_{n} < \Delta_{n - 1} < \ldots < \Delta_{1} \leq \Delta_{0} = 1 \end{displaymath}
be a sequence of dyadic scales.  We say that a set $P\subset [0,1)^2$ is \emph{$(\Delta_j)_{j=1}^n$-uniform} if there is a sequence $(N_j)_{j=1}^n$ such that $|P\cap Q|_{\Delta_{j}} = N_j$ for all $j\in \{1,\ldots,n\}$ and all $Q\in\mathcal{D}_{\Delta_{j - 1}}(P)$. As usual, we extend this definition to $\mathcal{P}\subset\mathcal{D}_{\delta}$ by applying it to $\cup\mathcal{P}$.
\end{definition}

\begin{lemma} \label{l:uniformization}
Given $P\subset [0,1)^{2}$ and a sequence $\delta = \Delta_{n} < \Delta_{n - 1} < \ldots < \Delta_{1} \leq \Delta_{0} = 1$ of dyadic numbers, $n \geq 1$, there is a $(\Delta_j)_{j=1}^n$-uniform set $P'\subset P$ such that
\begin{equation}\label{form12}
|P'|_\delta \ge  \left(4 n^{-1}\log(1/\delta)\right)^{-n} |P|_\delta. \end{equation}
\end{lemma}
\begin{proof}
This is a standard ``bottom-to-top'' pigeonholing argument, but since we have not seen this particular statement in the literature we provide the details. Pigeonholing, we can locate a dyadic number $N_n\in [1,(\Delta_{n-1}/\Delta_n)^2]$ such that
\[
\sum \big\{ |P\cap Q|_\delta: Q\in\calD_{\Delta_{n-1}}, |P\cap Q|_{\Delta_n}\in [N_n,2 N_n)  \big\} \ge  \left(2\log(\Delta_{n-1}/\Delta_n)\right)^{-1}|P|_\delta.
\]
We can then obtain a set $P_{n}\subset P$ such that $|P_{n} \cap Q|_{\Delta_n}=N_{n}$ whenever $|P \cap Q|_{\Delta_n}\in [N_{n},2 N_{n})$, $Q \in \mathcal{D}_{\Delta_{n - 1}}$, and $P_{n}\cap Q=\emptyset$ otherwise. Then
\[
|P_{n}|_\delta \ge \left(4\log(\Delta_{n-1}/\Delta_n)\right)^{-1}|P|_\delta.
\]
We continue inductively. Suppose $P_n \supset P_{n-1}\cdots \supset P_{i + 1}$ have been constructed for some $i \geq 0$ with the property that $|P_{i + 1} \cap Q|_{\Delta_{j}}\in\{ 0, N_{j}\}$ for all $j \in \{i + 1,\ldots,n\}$ and all $Q\in\calD_{\Delta_{j - 1}}$. If $i = 0$, the construction terminates, and we set $P' := P_{1}$. Otherwise, proceeding as above (replacing "$P$" by "$P_{i + 1}$"), there exist a dyadic number $N_{i}\in [1,(\Delta_{i - 1}/\Delta_{i})^2]$ and a subset $P_{i}\subset P_{i + 1}$ such that the following holds:
\begin{enumerate}
\item $P_{i} \cap Q \in \{\emptyset,P_{i + 1} \cap Q\}$ for all $Q\in\calD_{\Delta_{i}}$.
\item $|P_{i}\cap \overline{Q}|_{\Delta_i} \in \{ 0, N_{i}\}$ for all $\overline{Q} \in\calD_{\Delta_{i-1}}$.
\item $|P_{i}|_\delta \ge \left(4\log(\Delta_{i-1}/\Delta_i)\right)^{-1}|P_{i + 1}|_\delta$.
\end{enumerate}
It follows from  properties (1)--(2) and the inductive hypothesis that $|P_{i} \cap Q|_{\Delta_j}\in\{ 0, N_{j}\}$ for all $j \in \{i,\ldots,n\}$ and $Q \in \mathcal{D}_{\Delta_{j - 1}}$.

The set $P' := P_{1}$ is $(\Delta_j)_{j=1}^n$-uniform by construction, and the lower bound \eqref{form12} follows at once from property (3) and the geometric mean-arithmetic mean inequality.
\end{proof}

\begin{lemma}[Uniformisation lemma]\label{lemma5}  Let
\begin{displaymath} \delta = \Delta_{n} < \Delta_{n - 1} < \ldots < \Delta_{1} \leq \Delta_{0} = 1 \end{displaymath}
be a sequence of dyadic scales, and let $J_{1},J_{2} \subset \{1,\ldots,n\}$. Let $P\subset [0,1)^2$ be $(\Delta_j)_{j=1}^n$-uniform. Assume further that
\begin{itemize}
\item $P$ is a $(s_j,C)$-set between the scales $\Delta_{j}$ and $\Delta_{j - 1}$ for all $j\in J_1$, and
\item $P$ is $(t_j,C,K)$-regular between the scales $\Delta_j$ and $\Delta_{j-1}$ for all $j \in J_2$.
\end{itemize}
Let $P'\subset P$ satisfy $|P'|_\delta \ge L^{-1} |P|_\delta$, and write $M=M(n,L,\delta)=L \cdot [4\log(1/\delta)]^n$. Then there is a $(\Delta_j)_{j=1}^n$-uniform subset $P''\subset P'$ such that $|P''|\ge M^{-1} |P|$ and, furthermore, $P''$ is an $(s_j,M C)$-set between the scales $\Delta_{j}$ and $\Delta_{j - 1}$ for all $j\in J_1$, and  $P''$ is $(t_j,M C,K)$-regular between the scales $\Delta_j$ and $\Delta_{j-1}$ for all $j \in J_2$. Finally, if $P'$ is a union of $\delta$-squares to begin with, then the same property remains true for $P''$.
\end{lemma}
\begin{proof}
By assumption, there is a sequence $(N_j)_{j=1}^n$ such that  $|P\cap Q|_{\Delta_{j}} = N_j$ for all $j\in \{1,\ldots,n\}$ and all $Q\in\mathcal{D}_{\Delta_{j - 1}}(P)$. Let $P''$ be the set obtained by applying Lemma \ref{l:uniformization} to $P'$; let $N''_{j}$ be the corresponding cardinalities. Note that $N''_{j} \le N_{j}$ for all $j\in\{1,\ldots,n\}$ and that $|P''|_{\delta} \ge M^{-1} |P|_{\delta}$. Since
\[
|P''|_{\delta} = \prod_{j=1}^n N''_j ,\quad |P|_{\delta} =  \prod_{j=1}^n N_j,
\]
we must have $N''_j \ge M^{-1}N_j$ for all $j\in\{1,\ldots,n\}$. It follows that
\[
|P''\cap Q|_{\Delta_j} \ge M^{-1} |P\cap Q|_{\Delta_j} \quad\text{ for all }j \in \{1,\ldots,n\}, Q\in\mathcal{D}_{\Delta_{j - 1}}(P'').
\]
This is easily seen to imply the claim.
\end{proof}

\begin{proposition}\label{prop3} Given $s\in (0,1)$, $t>s$, $\tau\in (0,1)$, $n\ge 1$, if $\epsilon_{G},\eta>0$ are taken small enough in terms of $s$ and $t$ only, $\lambda$ is taken small enough in terms of $s,t, \tau$ and $n$, and $0< \epsilon_{N}\le \epsilon_{G}$, then the following holds for all $C_P\ge 1$, and $0 < \delta \leq \delta_{0} = \delta_{0}(s,t,\epsilon_{N},\tau,n,C_P,\lambda) \leq 1$.

Let $(\mathcal{P},\mathcal{T}) \subset \mathcal{D}_{\delta} \times \mathcal{T}^{\delta}$ be a $(\delta,s,\delta^{-\lambda},M)$-nice configuration for some $M \geq 1$. Let
\begin{displaymath} \delta = \Delta_{n} < \Delta_{n - 1} < \ldots < \Delta_{1} \leq \Delta_{0} = 1 \end{displaymath}
be a sequence of dyadic scales, and assume that $\mathcal{P}$ is $(\Delta_{j})_{j = 1}^{n}$-uniform. We assume that the scale indices $\{1,\ldots,n\}$ are partitioned into \emph{normal scales}, \emph{good scales}, and \emph{bad scales}, denoted $\mathcal{N},\mathcal{G}$, and $\mathcal{B}$, respectively. We assume that
\begin{equation}\label{form87} \Delta_{j}/\Delta_{j - 1} \leq \delta^{\tau}, \qquad j \in \mathcal{N} \cup \mathcal{G}. \end{equation}
Moreover, the family $\mathcal{P}$ has the following structure at the normal and good scales:
\begin{itemize}
\item If $j \in \mathcal{N}$, then $\mathcal{P}$ is an $(s,[\log(1/\delta)]^{C_{P}} \cdot (\Delta_{j - 1}/\Delta_{j})^{\epsilon_{N}})$-set between the scales $\Delta_{j}$ and $\Delta_{j - 1}$.
\item If $j \in \mathcal{G}$, there exists a number $t_{j} \geq t$ such that $\mathcal{P}$ is
\begin{displaymath} (t_{j},[\log(1/\delta)]^{C_{P}} \cdot (\Delta_{j - 1}/\Delta_{j})^{\epsilon_G},[\log(1/\delta)]^{C_{P}} \cdot (\Delta_{j - 1}/\Delta_{j})^{\epsilon_{G}})\text{-regular} \end{displaymath}
between the scales $\Delta_{j}$ and $\Delta_{j - 1}$.
\end{itemize}
If $j \in \mathcal{B}$, there is no information about the distribution of $\mathcal{P}$ between the scales $\Delta_{j}$ and $\Delta_{j - 1}$.

Then, there exist constants $C = C(n,s,\tau,C_{P}) > 0$ and $C'=C'(n,\tau)>0$ such that
\begin{equation}\label{form44} |\mathcal{T}| \geq \left[ \log \left(\tfrac{1}{\delta} \right) \right]^{-C} \cdot M \cdot \delta^{C' \lambda} \cdot  \delta^{-s + \epsilon_{N}} \cdot \prod_{j \in \mathcal{G}} \left(\tfrac{\Delta_{j - 1}}{\Delta_{j}} \right)^{\eta} \cdot \prod_{j \in \mathcal{B}} \tfrac{\Delta_{j}}{\Delta_{j - 1}}. \end{equation}
\end{proposition}

\begin{remark} Earlier in the paper we defined the notation $A \lessapprox_{\delta} B$, which meant that there exists an absolute constant $C \geq 1$ such that
\begin{equation}\label{form85} A \leq C \cdot \left[\log \left(\tfrac{1}{\delta}\right) \right]^{C} \cdot B. \end{equation}
Now, the constants "$C_{P},n,s,\tau$" appearing in the statement of Proposition \ref{prop3} will be regarded as "absolute" in the sense that the "$\lessapprox_{\delta}$" notation may also hide constants depending on $C_{P},n,s,\tau$. Therefore, the inequality \eqref{form44} could be abbreviated to
\begin{displaymath} |\mathcal{T}| \gtrapprox_\delta M \cdot \delta^{C'\lambda} \cdot \delta^{-s + \epsilon_{N}} \cdot \prod_{j \in \mathcal{G}} \left(\tfrac{\Delta_{j - 1}}{\Delta_{j}} \right)^{\eta} \cdot \prod_{j \in \mathcal{B}} \tfrac{\Delta_{j}}{\Delta_{j - 1}}. \end{displaymath}
In most occurrences of "$\lessapprox_{\delta}$" below, the constant "$C$" in \eqref{form85} will in fact be absolute, but we do not differentiate this in the notation. Even more wastefully, the notation "$\lessapprox_{\delta}$" will be used in situations where "$\log(1/\delta)$" in \eqref{form85} could be improved to "$\log(\Delta_{j - 1}/\Delta_{j})$". \end{remark}

\begin{proof}[Proof of Proposition \ref{prop3}] The proof will proceed by induction on the number "$n$" in $\{\Delta_{j}\}_{j = 1}^{n}$. More precisely, we will prove the case $n = 1$ separately below in a moment, whereas for $n > 1$, we use Proposition \ref{p:induction-on-scales} to reduce the proof of the main estimate \eqref{form44} to the case $n - 1$, applied with new constants $C_{P}',\lambda'$ of the form  $C_{P}' = C_{P} + C_{n}$ and $\lambda' = O(\tau^{-1})\lambda$.

While proving \eqref{form44}, we may assume that various powers of $\log(1/\delta)$, whose magnitude depends only on $n,s,\tau,C_{P}$, are tiny compared to $\delta^{-\epsilon_{N} \tau}$, thus
\begin{equation}\label{rev2} (\log(1/\delta))^{O(n,s,\tau,C_{P})} \leq \delta^{-\epsilon_{N} \tau} = \delta^{-\min\{\epsilon_{G},\epsilon_{N}\}\tau}. \end{equation}
This is because the upper bound $\delta_{0}$ for $\delta$ is allowed to depend on $s,t,\epsilon_{N},\tau,C_{P}$, and $\epsilon_{N}\le\epsilon_{G}$. This smallness of "$\delta$" will be assumed without further remark.

Recall that $(\mathcal{P},\mathcal{T})$ is assumed to be a $(\delta,s,\delta^{-\lambda},M)$-nice configuration. By definition, this means that for every $p \in \mathcal{P}$, there exists a $(\delta,s,\delta^{-\lambda})$-set $\mathcal{T}(p) \subset \mathcal{T}$ of cardinality $|\mathcal{T}(p)| = M$ such that $T \cap p \neq \emptyset$ for all $p \in \mathcal{T}(p)$.

\subsection{Case $n = 1$} If $1 \in \mathcal{B}$, then $|\mathcal{T}| \geq |\mathcal{T}(p)| = M$, where $p \in \mathcal{P}$ is arbitrary; this yields \eqref{form44} in the case $1 \in \mathcal{B}$, since $s<1$. If $1 \in \mathcal{N}$, Corollary \ref{prop5} applied with $t=s$, $C_{T}=\delta^{-\lambda}$ and $C_{P}\lessapprox_{\delta} \delta^{-\epsilon_{N}}$ shows that
\begin{equation}\label{form50} |\mathcal{T}| \gtrapprox_{\delta} M \cdot \delta^{\lambda} \cdot \delta^{-s + \epsilon_{N}}, \end{equation}
which yields \eqref{form44}. Finally, if $1 \in \mathcal{G}$, then $\mathcal{P}$ is
\[
(\delta,\bar{t},[\log(1/\delta)]^{C_{P}} \cdot \delta^{-\epsilon_{G}},[\log(1/\delta)]^{C_{P}} \cdot \delta^{-\epsilon_{G}})\text{-regular}
\]
between the scales $\delta$ and $1$ for some $\bar{t} \geq t$. We claim that if $\epsilon_{G},\eta > 0$ are sufficiently small, depending on $s$ and $t$ alone, then \eqref{form50} can be upgraded to
\begin{equation}\label{form51} |\mathcal{T}| \geq M \cdot \delta^{-s - \eta}, \end{equation}
which is better than needed.

Since $\bar{t} \geq t$, in particular $\mathcal{P}$ is a $(\delta,t,[\log(1/\delta)]^{C_{P}} \cdot \delta^{-\epsilon_{G}})$-set. By Corollary \ref{prop5}, this implies
\begin{equation}\label{rev1} |\mathcal{T}| \gtrapprox_{\delta} \delta^{\epsilon_{G}} \cdot M \delta^{-s} \cdot (M\delta^s)^{\tfrac{t-s}{1-s}}. \end{equation}
If it happens that $M\geq\delta^{-s-\gamma}$, where
\begin{equation}\label{choiceOfGamma}
\gamma = \frac{(\eta+2\epsilon_{G})(1-s)}{(t-s)},
\end{equation}
then \eqref{form51} follows immediately from \eqref{rev1}, since $\epsilon_{N} \leq \epsilon_{G}$. On the other hand, if $M \leq \delta^{-s - \gamma}$, we use Theorem \ref{t:improvedIncidence} instead. Let $\epsilon_0=\epsilon_0(s,t)>0$ be the constant given by the theorem. We assume that $\epsilon_G$, $\eta$ and $\lambda$ in the current proposition are taken sufficiently small in terms of $s, t,\epsilon_{0}$: then $\mathcal{P}$ is a $(\delta,\bar{t},\delta^{-\epsilon_{0}},\delta^{-\epsilon_{0}})$-regular set and the families $\mathcal{T}(p)$, $p \in \mathcal{P}$, are $(\delta,s,\delta^{-\epsilon_0})$-sets for $\delta > 0$ small enough. This implies \eqref{form51}, since $|\mathcal{T}| \geq \delta^{-2s - \epsilon_{0}} \geq M \cdot \delta^{-s - \epsilon_{0}+\gamma} \geq M \cdot \delta^{-s - \eta}$, again taking $\eta$ and $\epsilon_G$ small enough in terms of $\epsilon_0$ only. Since $\epsilon_0$ depends on $s,t$, we conclude that ultimately $\eta, \epsilon_G$ and $\lambda$ need to be taken small in terms of $s,t$ only.

\subsection{Case $n \geq 2$} We assume that the claim has already been established for indices strictly smaller than some given $n \geq 2$. We abbreviate
\[
\Delta := \Delta_{1},\quad \bar{\delta}:=\frac{\delta}{\Delta}=\frac{\delta}{\Delta_1}.
\]
We are going to apply Proposition \ref{p:induction-on-scales}; this is permissible since $(\mathcal{P},\mathcal{T})$ is a $(\delta,s,C_{1},M)$-nice configuration with $C_{1} := \delta^{-\lambda}$. Let $\mathcal{P}', \mathcal{T}'$ be the sets provided by the proposition, and let
\begin{equation}\label{rev3} \mathcal{P}_{\Delta} = \mathcal{D}_{\Delta}(\mathcal{P}') \quad \text{and} \quad \mathcal{T}_{\Delta} \subset \mathcal{T}^{\Delta}, \end{equation}
be the objects appearing in Claim \eqref{it-induction-iii} of Proposition \ref{p:induction-on-scales}. Thus $(\mathcal{P}_{\Delta},\mathcal{T}_{\Delta})$ is a $(\Delta,s,C_{\Delta},M_{\Delta})$-nice configuration for some $C_{\Delta} \approx_{\delta} \delta^{-\lambda}$ and $M_{\Delta} \geq 1$.

Fix any $Q\in\mathcal{P}_{\Delta}$ for the rest of the proof. Let us write $\mathcal{P}^Q = S_{Q}(\mathcal{P}'\cap Q)$ and $\mathcal{T}_Q$ for the families appearing in Claim \eqref{it-induction-iv} of Proposition \ref{p:induction-on-scales}. Thus $(\mathcal{P}^{Q},\mathcal{T}_Q)$ is a $(\bar{\delta},s,C_{\bar{\delta}},M_{\bar{\delta}})$-nice configuration for some $C_{\bar{\delta}} := C_{Q} \approx_{\delta} \delta^{-\lambda}$, and $M_{\bar{\delta}} := M_{Q} \geq 1$. Our goal is to bound $|\mathcal{T}_\Delta|$ and $|\mathcal{T}_Q|$ from below in such a way that the desired conclusion will follow from \eqref{lower-bound-T}. The lower bound for $|\mathcal{T}_Q|$ will rely on the induction hypothesis.

We start by bounding $|\mathcal{T}_\Delta|$ from below. If $1 \in \mathcal{B}$, then we use the trivial bound
\begin{equation}\label{form79} |\mathcal{T}_{\Delta}| \geq  M_{\Delta}. \end{equation}
Assume then that $1 \in \mathcal{N}$, so $\mathcal{P}$ is an $(s,[\log(1/\delta)]^{C_{P}} \cdot \Delta^{-\epsilon_{N}})$-set between the scales $\Delta = \Delta_{1}$ and $1 = \Delta_{0}$. It follows from \eqref{it-induction-i} in Proposition \ref{p:induction-on-scales} that $\mathcal{P}_\Delta$ is a $(\Delta,s,\log\left(1/\delta\right)^{C_{P} + C} \cdot \Delta^{-\epsilon_{N}})$-set for some $C \sim 1$. Since $1 \in \mathcal{N}$, the factor $[\log(1/\delta)]^{C_{P} + C}$ is small compared to $\Delta^{-\epsilon_{N}}$ by \eqref{form87} and \eqref{rev2}. Since $(\mathcal{P}_{\Delta},\mathcal{T}_{\Delta})$ is a $(\Delta,s,C_{\Delta},M_{\Delta})$-nice configuration with $C_{\Delta} \approx \delta^{-\lambda}$, it follows from Corollary \ref{prop5} applied with $t=s$, $C_T\approx_{\delta} \delta^{-\lambda}$ and $C_P\lessapprox_{\delta} \Delta^{-\epsilon_{N}}$ that
\begin{equation}\label{form45} |\mathcal{T}_{\Delta}| \gtrapprox_{\delta} M_{\Delta} \cdot \delta^{\lambda}\cdot \Delta^{-s + \epsilon_{N}}. \end{equation}
Finally, if $1 \in \mathcal{G}$, hence $P$ is $(t_{1},[\log(1/\delta)]^{C_{P}} \cdot \Delta^{-\epsilon_{G}},[\log(1/\delta)]^{C_{P}} \cdot \Delta^{-\epsilon_{G}})$-regular with $t_{1} \geq 1$, then \eqref{form45} can be upgraded as follows:
\begin{equation}\label{form45a} |\mathcal{T}_{\Delta}| \geq M_{\Delta} \cdot  \Delta^{-s - \eta}. \end{equation}
This requires an application of Corollary \ref{prop5} and Theorem \ref{t:improvedIncidence}, and a similar case chase between the options "$M_{\Delta} \geq \Delta^{-s -\gamma}$" and "$M_{\Delta} < \Delta^{-s - \gamma}$" for a suitable $\gamma$ as we recorded below \eqref{form51}; the details are so similar that we omit them here. In particular, at this point we need to know that $[\log(1/\delta)]^{C_{P}}$ is small compared to $\Delta^{-\epsilon_{G}}$. This is true by \eqref{form87}, \eqref{rev2}.

We move to estimating $|\mathcal{T}_Q|$ from below.  This will rely on the induction hypothesis, except in one special case, which we treat immediately. This is the case where $\{2,\ldots,n\} \subset \mathcal{B}$. Then we use the trivial bound $|\mathcal{T}_Q| \ge M_Q$ to deduce from \eqref{lower-bound-T} that
\[
|\mathcal{T}| \gtrapprox_\delta  M \cdot M_\Delta^{-1}\cdot |\mathcal{T}_\Delta| .
\]
Together with our earlier estimates for $|\mathcal{T}_\Delta|$, \eqref{form79}--\eqref{form45a}, this gives the desired claim \eqref{form44} since, telescoping,
\[
\prod_{j = 2}^{n} \frac{\Delta_{j}}{\Delta_{j - 1}}= \frac{\Delta_n}{\Delta_1}=\frac{\delta}{\Delta} \le \left(\frac{\delta}{\Delta}\right)^{s-\epsilon_N}.
\]

In the sequel, we may therefore assume that there exists $j \in \{2,\ldots,n\} \cap [\mathcal{N} \cup \mathcal{G}] \neq \emptyset$, and in particular
\begin{equation}\label{form88} \bar{\delta} = \frac{\delta}{\Delta} \leq \frac{\Delta_{j}}{\Delta_{j - 1}} \stackrel{\eqref{form87}}{\leq} \delta^{\tau}. \end{equation}

The short story is that, thanks to \eqref{it-induction-i} in Proposition \ref{p:induction-on-scales}  the set $\mathcal{P}^Q= S_{Q}(\mathcal{P}'\cap Q)$ satisfies roughly the same hypotheses as $\mathcal{P}$ for the scale sequence $\{\bar{\Delta}_{j}\}_{j = 0}^{n - 1} := \{\Delta_{j}/\Delta\}_{j = 1}^{n}$. So, the induction hypothesis may (roughly speaking) be applied to $\mathcal{P}^{Q}$. The main technical hurdle is that the set $\mathcal{P}^{Q}$ should be uniform, and a priori it does not need to be: all we know about $\mathcal{P}' \cap Q \subset \mathcal{P} \cap Q$ is that $|\mathcal{P}' \cap Q| \approx_{\delta} |\mathcal{P} \cap Q|$. We resolve this by appealing to the Uniformisation lemma, Lemma \ref{lemma5}. We proceed to the details. As indicated, we write
\begin{displaymath} \bar{\Delta}_{j - 1} := \Delta_{j}/\Delta, \qquad j \in \{1,\ldots,n\}. \end{displaymath}
By the hypothesis of the proposition, $S_{Q}(\mathcal{P} \cap Q)$ is $(\bar{\Delta}_{j})_{j = 1}^{n - 1}$-uniform, and additionally an $(s,[\log(1/\delta)]^{C_{P}} \cdot (\bar{\Delta}_{j - 1}/\bar{\Delta}_{j})^{\epsilon_{N}})$-set between the scales $\bar{\Delta}_{j}$ and $\bar{\Delta}_{j - 1}$ for all $j \in \bar{\mathcal{N}} := (\mathcal{N} - 1) \, \setminus \, \{0\}$, and finally a $(t_{j},[\log(1/\delta)]^{C_{P}} \cdot (\bar{\Delta}_{j - 1}/\bar{\Delta}_{j})^{\epsilon_{G}},[\log(1/\delta)]^{C_{P}} \cdot (\bar{\Delta}_{j - 1}/\bar{\Delta}_{j})^{\epsilon_{G}})$-regular set between the scales $\bar{\Delta}_{j}$ and $\bar{\Delta}_{j - 1}$ for all $j \in \bar{\mathcal{G}} := (\mathcal{G} - 1) \, \setminus \, \{0\}$. Further, we know from Proposition \ref{p:induction-on-scales}\eqref{it-induction-i} that $|\mathcal{P}^{Q}| \approx_{\delta} |\mathcal{P} \cap Q|$.

We then apply the Uniformisation lemma, Lemma \ref{lemma5}: there exists a further subset of $\mathcal{P}^{Q}$, which we keep denoting with the same letter, with the following properties:
\begin{itemize}
\item[(P1) \phantomsection \label{P1}] $\mathcal{P}^{Q}$ is $(\bar{\Delta}_{j})_{j = 1}^{n - 1}$-uniform, and $|\mathcal{P}^{Q}| \approx_{\delta} |\mathcal{P} \cap Q|$ (recall that the "$\approx_{\delta}$" notation also tolerates the constants depending on "$n$", which arise from Lemma \ref{lemma5}),
\item[(P2) \phantomsection \label{P2}] $\mathcal{P}^{Q}$ is an $(s,[\log(1/\delta)]^{C_{P} + C_{n}} \cdot (\bar{\Delta}_{j - 1}/\bar{\Delta}_{j})^{\epsilon_{N}})$-set between the scales $\bar{\Delta}_{j}$ and $\bar{\Delta}_{j - 1}$ for all $j \in \bar{\mathcal{N}}$, where $C_{n} \geq 1$ only depends on $n$,
\item[(P3) \phantomsection \label{P3}] $\mathcal{P}^{Q}$ is $(t_{j},[\log(1/\delta)]^{C_{P} + C_{n}} \cdot (\bar{\Delta}_{j - 1}/\bar{\Delta}_{j})^{\epsilon_{G}},[\log(1/\delta)]^{C_{P} + C_{n}} \cdot (\bar{\Delta}_{j - 1}/\bar{\Delta}_{j})^{\epsilon_{G}})$-regular between the scales $\bar{\Delta}_{j}$ and $\bar{\Delta}_{j - 1}$ for all $j \in \bar{\mathcal{G}}$.
\end{itemize}
In order to apply the inductive hypothesis to the set $\mathcal{P}^{Q}$, it will be important to note that the multiplicative factors $[\log(1/\delta)]^{C_{P} + C_{n}}$ in \nref{P2}-\nref{P3} can be further bounded from above by $[\log(1/\bar{\delta})]^{C_{P} + C_{n}}$, for a -- say -- twice larger constant $C_{n}$. This follows from \eqref{form88}:
\begin{equation}\label{form89} [\log(1/\delta)]^{C_{P} + C_{n}} \leq \left(\tfrac{1}{\tau} \right)^{C_{P} + C_{n}} \cdot [\log(1/\bar{\delta})]^{C_{P} + C_{n}} \leq [\log(1/\bar{\delta})]^{C_{P} + 2C_{n}},\end{equation}
assuming that $\delta$, and hence $\bar{\delta}$, is chosen small enough in terms of $\tau, C_{P}$. The same argument, together with the bound $C_{\bar{\delta}} = C_Q\lessapprox_\delta \delta^{-\lambda}$ given by Proposition \ref{p:induction-on-scales}\eqref{it-induction-iv}, shows that $\mathcal{T}_Q(q)$ can be assumed to be a $(\bar{\delta},s,(\bar{\delta})^{-2\tau^{-1}\lambda})$-set for all $q \in \mathcal{P}^{Q}$. Here $\mathcal{T}_{Q}(q) \subset \mathcal{T}_{Q}$ is the $(\bar{\delta},s,C_{\bar{\delta}})$-set of tubes in $\mathcal{T}_{Q}$, all elements of which intersect $q$, whose existence is guaranteed by the $(\bar{\delta},s,C_{\bar{\delta}},M_{\bar{\delta}})$-niceness of $(\mathcal{P}^{Q},\mathcal{T}_{Q})$.

We have now verified that, provided $\delta$ is taken small enough, $\mathcal{P}^Q$, $\mathcal{T}_Q$, $\bar{\delta}$ and $(\bar{\Delta}_j)_{j=0}^{n-1}$ satisfy all the assumptions of the current proposition, with (new) constants $C'_P, \lambda'$ that depend only on the original ones, and $n$; in particular, $\lambda'= 2\tau^{-1}\lambda$, so that $\lambda'$ can be taken sufficiently small in terms of $s,t,\tau,n - 1$. Writing $\bar{\mathcal{B}} := (\mathcal{B} - 1) \, \setminus \, \{0\}$, the inductive hypothesis yields
\begin{equation}\label{form49}
\begin{split}
|\mathcal{T}_{Q}| & \gtrapprox_{\bar{\delta}} M_{\bar{\delta}} \cdot \bar{\delta}^{C' \lambda}\cdot \bar{\delta}^{-s + \epsilon_{N}} \cdot \prod_{j \in \bar{\mathcal{G}}} \left(\tfrac{\bar{\Delta}_{j - 1}}{\bar{\Delta}_{j}} \right)^{\eta} \cdot \prod_{j \in \bar{\mathcal{B}}} \tfrac{\bar{\Delta}_{j}}{\bar{\Delta}_{j - 1}} \\
& \ge  M_{\bar{\delta}} \cdot \delta^{C'\lambda} \cdot \left(\tfrac{\Delta}{\delta}\right)^{s  - \epsilon_{N}} \cdot \prod_{j \in \mathcal{G} \, \setminus \, \{1\}} \left(\tfrac{\Delta_{j - 1}}{\Delta_{j}} \right)^{\eta} \cdot \prod_{j \in \mathcal{B} \, \setminus \, \{1\}} \tfrac{\Delta_{j}}{\Delta_{j - 1}}.
\end{split}
\end{equation}
To clarify, the implicit constant in \eqref{form49} is of the form $[\log(1/\bar{\delta})]^{-O(1)}$, where "$O(1)$" depends on the constants of the sets $\mathcal{P}^{Q}$. As we have observed, these constants ultimately depend only on the original constant $C_P$, and on $n$. Also, $C'$ depends only on $\tau$ and $n-1$. Furthermore, \eqref{form49} holds if $\delta$ is sufficiently small in terms of $C_P, \lambda,t,\epsilon_{N},n,s,\tau$ so that $\bar{\delta}$ is also sufficiently small, recall \eqref{form88}, to apply the inductive hypothesis.

Combining \eqref{form49} with \eqref{form79} and the trivial estimate $\Delta^{s - \epsilon_{N}} \geq \Delta$ in the case $1\in\mathcal{B}$, \eqref{form45} in the case $1 \in \mathcal{N}$ and, finally, \eqref{form45a} in the case $1 \in \mathcal{G}$, we deduce that, in any case,
\begin{equation}\label{rev4}
|\mathcal{T}_\Delta|\cdot |\mathcal{T}_Q| \gtrapprox_\delta M_\Delta \cdot M_{\bar{\delta}}  \cdot \delta^{(C'+2)\lambda}\cdot \delta^{-s + \epsilon_{N}} \cdot \prod_{j \in \mathcal{G}} \left(\tfrac{\Delta_{j - 1}}{\Delta_{j}} \right)^{\eta} \cdot \prod_{j \in \mathcal{B}} \tfrac{\Delta_{j}}{\Delta_{j - 1}}.
\end{equation}
To conclude the proof, recall from \eqref{lower-bound-T} (with current notation, see below \eqref{rev3}) that
\begin{displaymath} |\mathcal{T}| \gtrapprox_{\delta} M \cdot \frac{|\mathcal{T}_{\Delta}| \cdot |\mathcal{T}_{Q}|}{M_{\Delta} \cdot M_{\bar{\delta}}}. \end{displaymath}
Along with \eqref{rev4}, this shows that \eqref{form44} holds, finishing the proof of Proposition \ref{prop3}. \end{proof}

\section{Choosing good multiscale decompositions of uniform sets}\label{s:multiScaleDecomposition}

The goal of this section is to obtain a multiscale decomposition $(\Delta_i)_{i=1}^n$ of a uniform $(\delta,t,\delta^{-\e})$-set $P$ such the hypotheses of Proposition \ref{prop3} hold. More precisely, we will prove the following result:

\begin{proposition}\label{p:multiscaledecomp-kaufman}
Given $s\in (0,1)$, $t\in (s,2]$, $\Delta \in 2^{-\N}$ and $\e>0$ there is $\tau=\tau(\e,s,t) \in (0,\epsilon]$ such that the following holds for $m$ sufficiently large in terms of $s,t,\Delta,\e$.

Let $\delta=\Delta^m$. Let $P\subset [0,1)^{2}$ be a $(\Delta^i)_{i=1}^m$-uniform $(\delta,t,\delta^{-\e})$-set. Then there are numbers $t_{j} \in [s,2]$, $1 \leq j \leq n$, and scales
\[
\delta = \Delta_{n} < \Delta_{n - 1} < \ldots < \Delta_{1} < \Delta_{0} = 1,
\]
with each $\Delta_j$ an integer power of $\Delta$, and a partition $\{1,\ldots,n\} = \mathcal{S} \cup \mathcal{B}$ ("structured" and "bad" indices) such that the following properties hold:
\begin{enumerate}[(\rm i)]
\item  \label{cor:multiscaledecomp-kaufman:i} $\Delta_{j - 1}/\Delta_{j} \ge \delta^{-\tau}$ for all $j \in \mathcal{S}$, and $\prod_{j \in \mathcal{B}} (\Delta_{j - 1}/\Delta_{j}) \leq \delta^{-\epsilon}$.
\item \label{cor:multiscaledecomp-kaufman:ii} For each $j \in \mathcal{S}$, the set $P$ is a $(t_{j},(\Delta_{j - 1}/\Delta_{j})^{\e})$-set between the scales $\Delta_{j}$ and $\Delta_{j - 1}$. Moreover, if $t_{j} >s$, then $P$ is $(t_{j},(\Delta_{j - 1}/\Delta_{j})^{\e},(\Delta_{j - 1}/\Delta_{j})^{\e})$-regular between the scales $\Delta_{j}$ and $\Delta_{j - 1}$.
\item \label{cor:multiscaledecomp-kaufman:iii} $\prod_{j \in \mathcal{S}} (\Delta_{j - 1}/\Delta_{j})^{t_{j}} \ge \delta^{\epsilon -t}$.
\item \label{cor:multiscaledecomp-kaufman:iv} If $j\in\mathcal{B}$, then $j+1\notin\mathcal{B}$ for all $j \in \{1,\ldots,n - 1\}$.
\end{enumerate}
\end{proposition}

Let $P$ be as in the statement of this proposition. By definition of $(\Delta^i)_{i=1}^m$-uniformity, there is a sequence $(N_i)_{i=1}^{m}$ of numbers in $[1,\Delta^{-2}]$ such that
\[
|P\cap Q|_{\Delta^{i}} = N_i \quad\text{for all }Q\in\calD_{\Delta^{i - 1}}, P\cap Q\neq \emptyset,
\]
for all $i \in \{1,\ldots,m\}$. We define a function $f=f_P:[0,m]\to [0,2m]$, depending on the sequence $(N_i)_{i=1}^{m}$, by setting $f(0)=0$,
\begin{equation}\label{codeFunction}
f(j) = \sum_{i=1}^{j} \frac{\log(N_i)}{\log(1/\Delta)}, \qquad 1\le j\le m, \end{equation}
and interpolating linearly. This function encodes the ``branching structure'' of the set $P$, and it is convenient to study the multiscale geometry of $P$ via the function $f$.

\begin{definition} \label{def:eps-linear}
Given a function $f:[a,b]\to\R$, we let
\[
s_f(a,b) = \frac{f(b)-f(a)}{b-a}
\]
be the slope of the affine function $L_{f,a,b}$ that agrees with $f$ on $a$ and $b$, namely
\[
L_{f,a,b}(x) := f(a) +s_f(a,b)(x-a), \qquad x \in \R.
\]
We say that \emph{$(f,a,b)$ is $\e$-linear} or \emph{$f$ is $\e$-linear on $[a,b]$} if
\[
\big|f(x)-  L_{f,a,b}(x) \big|\le \e |b-a|, \qquad x\in [a,b].
\]
Likewise, we say that \emph{$(f,a,b)$ is $\e$-superlinear} or \emph{$f$ is $\e$-superlinear on $[a,b]$} if
\[
f(x) \ge L_{f,a,b}(x) - \e|b-a|, \qquad x\in [a,b].
\]
\end{definition}

The following lemma provides a dictionary between properties of the Lipschitz function $f_P$ and properties of the set $P$.
\begin{lemma} \label{lem:regular-is-frostman-ahlfors}
Let $P$ be a $(\Delta^i)_{i=1}^m$-uniform set with associated function $f=f_P$ and let $\delta=\Delta^m$.
\begin{enumerate}
  \item If $f$ is $\e$-superlinear on $[0,m]$, then $P$ is a $(\delta,s_f(0,m),O_\Delta(1)\delta^{-\e})$- set.
  \item If $f$ is $\e$-linear on $[0,m]$ then $P$ is $(s_f(0,m),O_{\Delta}(1)\delta^{-\e},O_\Delta(1)\delta^{-\e})$-regular between scales $\delta$ and $1$.
\end{enumerate}
\end{lemma}

\begin{proof}

Let $Q\in\mathcal{D}_{\Delta^j}(P)$, with $0\le j< m$. By the uniformity of $P$, the definition of $f$, and the assumption that $f$ is $\e$-superlinear,
\begin{align*}
|P\cap Q|_\delta &= N_{j+1}\cdots N_m = \frac{|P|}{N_1\cdots N_j}\\
&= |P|\cdot \Delta^{f(j)} \le \delta^{-\e m}\cdot |P|\cdot (\Delta^j)^{s_f(0,m)}.
\end{align*}
If $Q\in\mathcal{D}_r(P)$ for a dyadic $r\ge \delta$, then $Q\subset\widehat{Q}$ for some $\widehat{Q}\in\mathcal{D}_{\Delta^j}(P)$ with $\Delta^j \le \Delta^{-1} r$, so we get the same bound up to a $O_{\Delta}(1)$ factor.

Now suppose $f$ is $\e$-linear. By the $(\Delta^i)_{i=1}^m$uniformity of $P$,
\begin{align*}
|P|_{\delta^{1/2}} &\le N_1\cdots N_{\lceil m/2\rceil} = \Delta^{f(\lceil m/2\rceil)}\\
&\le \delta^{-\e}\cdot \Delta^{\lceil m/2\rceil \cdot s_f(0,m)} \le \Delta^{-1/2} \delta^{-\e} \cdot \delta^{-s_f(0,m)/2}.
\end{align*}
This completes the proof. \end{proof}

The following lemma is \cite[Lemma 4.4]{Shmerkin20}:
\begin{lemma} \label{lem:tube-null}
For every $\e>0$ there is $\tau>0$ such that the following holds: for any $1$-Lipschitz function $f:[a,b]\to\R$ there exists a family of non-overlapping intervals $\{ [c_j,d_j]\}_{j=1}^M$ such that:
\begin{enumerate}[(\rm i)]
  \item $(f,c_j,d_j)$ is $\e$-linear for all $j$.
  \item $d_j-c_j\ge \tau(b-a)$ for all $j$.
  \item $\big|[a,b] \, \setminus \, \bigcup_j [c_j,d_j]\big| \le \e(b-a)$.
\end{enumerate}
\end{lemma}
In fact, by considering $f/2$ instead of $f$, Lemma \ref{lem:tube-null} also holds for $2$-Lipschitz functions.

\begin{lemma} \label{lem:combinatorial-kaufman}
Fix $s\in(0,1)$ and $t\in (s,2]$. For every $\epsilon>0$ there is $\tau>0$ such that the following holds: for any $2$-Lipschitz function $f:[0,m]\to\R$ with $f(0)=0$ such that
\[
f(x)\ge t x-\e m \quad\text{for all }x\in [0,m],
\]
there exists a family of non-overlapping intervals $\{ [c_j,d_j]\}_{j=1}^n$ contained in $[0,m]$ such that:
\begin{enumerate}[(\rm i)]
  \item \label{it:lem:combinatorial-kaufman-i} For each $j$, at least one of the following alternatives holds:
  \begin{enumerate}[(\rm a)]
  \item \label{it:lem:combinatorial-kaufman-ia}  $(f,c_j,d_j)$ is $\e$-linear with $s_f(c_j,d_j)\ge s$.
  \item  \label{it:lem:combinatorial-kaufman-ib} $(f,c_j,d_j)$ is $\e$-superlinear with $s_f(c_j,d_j) = s$.  \end{enumerate}
  \item \label{it:lem:combinatorial-kaufman-ii} $d_j-c_j\ge \tau m$ for all $j$.
  \item \label{it:lem:combinatorial-kaufman-iii} $\big|[0,m] \, \setminus \, \bigcup_j [c_j,d_j]\big| \lesssim_{s,t} \e m$.
\end{enumerate}
\end{lemma}

\begin{remark} We briefly discuss the difference between Lemmas \ref{lem:tube-null} and \ref{lem:combinatorial-kaufman}. Lemma \ref{lem:tube-null} is simply a quantified version of Rademacher's theorem: a Lipschitz function is approximated by affine functions almost everywhere. \emph{A priori}, we cannot expect any lower bounds on the slopes of these affine functions. Lemma \ref{lem:combinatorial-kaufman} yields such lower bounds under the extra assumption $f(x) \geq tx - \epsilon m$. This hypothesis is far too weak to guarantee that the slopes of each affine function would exceed $t$ -- or even some fixed $s < t$. However, it guarantees that we can combine the "intervals of affinity" in a useful way. Roughly speaking, we leave those intervals untouched where the slope of the affine function exceeds "$s$" to begin with. Then, we combine those intervals together where the slope of the affine function is initially smaller than "$s$". As a consequence, we lose the affine approximation on the compound intervals $[a,b]$, but -- and this is the key point -- we are able to retain the property that $f$ is $\epsilon$-superlinear on $[a,b]$ with slope $s_{f}(a,b) = s$.

Why could we simply not combine \textbf{all} the intervals together? After all, the initial hypothesis $f(x) \geq tx - \epsilon m$ (almost) says that $f$ is $\epsilon$-superlinear on $[0,t]$ with slope $\geq t$. This kind of information would not be useful: in fact, the upper bound $s_{f}(c_{j},d_{j}) \leq s$ in part (b) is just as crucial as the (matching) lower bound. Since $t > s$, and $f(x) \geq tx - \epsilon m$ for all $x \in [0,m]$, it will imply that alternative (b) cannot occur "all the time". In other words, the total length of intervals satisfying alternative (a) is not negligible, quantitatively. This is, at the end of the day, where all our "gains" stem from, see \eqref{form90}. \end{remark}

\begin{proof}[Proof of Lemma \ref{lem:combinatorial-kaufman}]
Let $\{ [c_j,d_j]\}_{j=1}^M$ be the decomposition provided by Lemma \ref{lem:tube-null} applied with $\e^2$ in place of $\e$, with the $c_j$ in increasing order. Let $\tau=\tau(\e^2)$ be the number provided by  Lemma \ref{lem:tube-null}. We will denote the intervals in the decomposition we are seeking by $\{ [\wt{c}_i,\wt{d}_i]\}_i$.

If $s_f(c_j,d_j)\ge s$ for all $j$ there is nothing to do. Otherwise, let $k$ satisfy $s_f(c_k,d_k) < s$ and be largest with this property. If $d_k\le \e m/(t-s)$, then we remove all intervals $[c_i,d_i]$ with $i\le k$ and we are done. Otherwise, note that
\[
s_f(0,d_k)= \frac{f(d_k)}{d_k} \ge  \frac{t d_k-\e m}{d_k} >s,
\]
see Figure \ref{fig2}. Since $s_f(c_k,d_k)< s$ and $f$ is piecewise affine, there is a largest $c'\in (0,c_k)$ such that $s_f(c',d_k)=s$.

\begin{figure}[h!]
\begin{center}
\begin{overpic}[scale = 1]{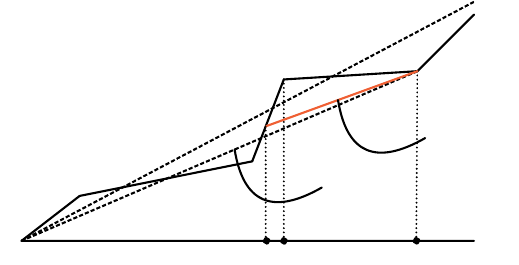}
\put(83,24){$s_{f}(c',d_{k}) = s$}
\put(49.5,-0.5){$c'$}
\put(53.5,-0.5){$c_{k}$}
\put(78,-0.5){$d_{k}$}
\put(3.5,-0.5){$0$}
\put(89,-0.5){$m$}
\put(89,40){$f$}
\put(93,50){$x \mapsto tx$}
\put(63.5,14.5){$s_{f}(0,d_{k}) > s$}
\end{overpic}
\caption{Selecting the interval $(c',d_{k})$.}\label{fig2}
\end{center}
\end{figure}
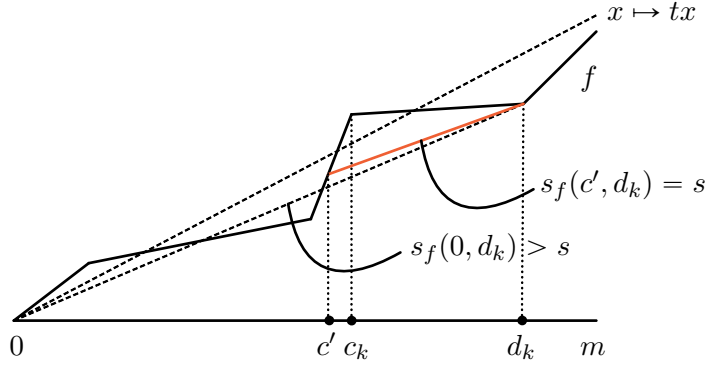

Fix $x\in (c',c_k)$. Since $s_f(c',d_k)=s$ is a convex combination of $s_f(c',x)$ and $s_f(x,d_k)<s$ (by the maximality of $c'$), we have $f(x)> f(c')+s(x-c')$. On the other hand, if $x\in [c_k,d_k]$ then, since $(f,c_k,d_k)$ is $\epsilon^2$-superlinear and $s_f(c',d_k)=s$, $s_f(c_k,d_k)<s$,
\begin{align*}
f(x) &\ge f(d_k)-s_f(c_k,d_k)(d_k-x) - \e^2(d_k-c_k) \\
&> f(c')+(d_k-c')s - s(d_k-x)-\e^2(d_k-c') \\
&= f(c')+ s (x-c')-\e^2(d_k-c').
\end{align*}
We conclude that $(f,c',d_k)$ is $\e^2$-superlinear. We define $[\wt{c}_i,\wt{d}_i]=[c_i,d_i]$ for $i>k$,  and $[\wt{c}_k,\wt{d}_k]=[c',d_k]$. To continue the construction, consider several cases:
\begin{itemize}
  \item If $c'\notin \cup_j [c_j,d_j]$, then let $\ell$ be the largest index with $d_\ell<c'$ and repeat the procedure by selecting a largest $k'\le\ell$ such that $s_f(c_{k'},d_{k'})<s$ (if $c'<c_1$ or no such $k'$ exists, we stop).
  \item If $c'\in [c_{\ell},d_{\ell}]$ and $(c' - c_{\ell}) \leq \e (d_{\ell} - c_{\ell})$, we discard the piece $[c_\ell,c']$ and repeat the process for the intervals up to $[c_{\ell-1},d_{\ell-1}]$ (if $\ell=1$, we stop).
  \item If $c'\in [c_{\ell},d_{\ell}]$ and $(c' - c_{\ell}) > \e (d_{\ell} - c_{\ell})$, we replace the interval by $[c_{\ell},d_{\ell}]$ by $[c_{\ell},c']$ and repeat the procedure of the first point after this substitution. Note that $(f,c_\ell,c')$ is $O(\e)$-linear (this is the reason why we applied Lemma \ref{lem:tube-null} with $\e^2$ in place of $\e$).
\end{itemize}
By construction, the process must finish in $\le \tau^{-1}$ steps and all resulting intervals satisfy one of the alternatives in \eqref{it:lem:combinatorial-kaufman-i}. All the resulting intervals have length $\ge \e\tau(b-a)$, since they contain at least a proportion $\e$ of some $[c_j,d_j]$. Finally, the sum of their lengths is at least $(1-O_{s,t}(\e))m$ since a proportion at least $(1-\e)$ of each interval $[c_j,d_j]$ with $d_j \gtrsim_{s,t} \e m$ is contained in some $[\wt{c}_i,\wt{d}_i]$.  This yields the claim, with $\e\tau$ in place of $\tau$.
\end{proof}

We can now conclude the proof of Proposition \ref{p:multiscaledecomp-kaufman}.

\begin{proof}[Proof of Proposition \ref{p:multiscaledecomp-kaufman}]
Let $(N_i)_{i=1}^{m}$ be the sequence associated to the $(\Delta^i)_{i=1}^m$-uniformity of $P$. Note that, since $P$ is a $(\delta,t,\delta^{-\e})$-set, if $Q\in\calD_{\Delta^{i - 1}}(P)$, then
\[
N_i \cdots N_{m} = |P\cap Q|_{\delta} \le \delta^{-\e} \Delta^{(i - 1)t} |P|_\delta = \delta^{-\e} \Delta^{(i - 1)t} N_{1}\cdots N_{m}.
\]
Hence $\prod_{j=1}^{i - 1} N_{j} \ge \Delta^{(1 - i)t}\delta^\e = \Delta^{(1 - i)t + \epsilon m}$ for all $i \in \{1,\ldots,m\}$, and evidently also $\prod_{j = 1}^{m} N_{j} = |P|_{\delta} \geq \delta^{-t + \epsilon} = \Delta^{-m t  + \epsilon m}$. Translated into the function $f=f_P$ from \eqref{codeFunction}, this implies that $f(i) \ge  ti -\e m$ for all $i \in \{0,1,\ldots,m\}$ and therefore, by piecewise linearity, $f(x) \ge t x-\e m$ for all $x\in [0,m]$. Then we can apply Lemma \ref{lem:combinatorial-kaufman} to $f$; let  $([a_i,b_i])_{i=1}^{\bar{n}}$ be the resulting intervals. By perturbing the endpoints a little, we may assume that $a_i$ and $b_i$ are integers (these perturbations can be absorbed into any $\e$-terms or $\tau$-terms by taking $m$ sufficiently large in terms of $\e$ and adjusting the values of $\e,\tau$ slightly). Then, we define the scale sequence
\begin{displaymath} \delta = \Delta_{n} < \Delta_{n - 1} < \ldots < \Delta_{1} < \Delta_{0} = 1 \end{displaymath}
as the unique sequence "spanned" by the intervals $[\Delta^{b_{i}},\Delta^{a_{i}}]$, $1 \leq i \leq \bar{n}$. This means that every interval $[\Delta_{j},\Delta_{j - 1}]$, $1 \leq j \leq n$, has one of the following types (a)-(c):
\begin{itemize}
\item[(a)] $[\Delta_{j},\Delta_{j - 1}] \in \{[\Delta^{b_{i}},\Delta^{a_{i}}] : 1 \leq i \leq \bar{n}\}$, or
\item[(b)] $[\Delta_{j},\Delta_{j - 1}] = [\Delta^{a_{i + 1}},\Delta^{b_{i}}]$ for some $1 \leq i \leq \bar{n} - 1$, or
\item[(c)] $[\Delta_{j},\Delta_{j - 1}] \in \{[\Delta^{a_{1}},1],[\delta,\Delta^{b_{\bar{n}}}]\}$.
\end{itemize}

We let
\begin{displaymath} \mathcal{S} := \{1 \leq j \le n : [\Delta_{j},\Delta_{j - 1}] \text{ has type (a)}\}, \end{displaymath}
and $\mathcal{B} := \{1,\ldots,n\} \, \setminus \, \mathcal{S}$. Claim \eqref{cor:multiscaledecomp-kaufman:i} is now immediate from Lemma \ref{lem:combinatorial-kaufman}. To prove Claim \eqref{cor:multiscaledecomp-kaufman:ii}, fix $j \in \mathcal{S}$, so $[\Delta_{j},\Delta_{j - 1}] = [\Delta^{b_{i}},\Delta^{a_{i}}]$ for some $1 \leq i \leq \bar{n}$. Let $t_j =s_f(a_i,b_i)$. It follows at once from the assumed $(\Delta^i)_{i=1}^m$-uniformity of $P$ (see Definition \ref{def:uniformity}) that if $Q\in\calD_{\Delta^{a_i}}$, then $S_Q(P\cap Q)$ is $(\Delta^j)_{j=1}^{b_i-a_i}$-uniform. In this case Claim \eqref{cor:multiscaledecomp-kaufman:ii} is a consequence of Lemma \ref{lem:regular-is-frostman-ahlfors} and Lemma \ref{lem:combinatorial-kaufman}\eqref{it:lem:combinatorial-kaufman-i} (with $O(\epsilon)$ in place of $\epsilon$ - to get $\epsilon$ one simply applies the argument to $\epsilon/C$ for a suitable $C$).

To prove Claim \eqref{cor:multiscaledecomp-kaufman:iii}, recall that $(b_i-a_i)t_i = f(b_i)-f(a_i)$, and the union of the intervals $[a_i,b_i]$  covers $[0,m]$ up to measure $\lesssim_{s,t}\epsilon m$. Since $f$ is $2$-Lipschitz, $f(0) = 0$, and $f(m)\ge m(t-\epsilon)$ as we saw earlier, we get
\begin{align*}
\sum_{i=1}^{\bar{n}} (b_i-a_i)t_i & = \sum_{i = 1}^{\bar{n}} f(b_{i}) - f(a_{i})\\
& \geq (f(m) - f(0)) - 2\left| [0,m] \, \setminus \bigcup_{i = 1}^{\bar{n}} [a_{i},b_{i}] \right|\\
& \ge m(t - O_{s,t}(\epsilon)).
\end{align*}
It follows that
\begin{displaymath} \prod_{j \in \mathcal{S}} (\Delta_{j - 1}/\Delta_{j})^{t_{j}} = \prod_{i = 1}^{\bar{n}} \Delta^{-(b_{i} - a_{i})t_{i}}  \geq \Delta^{-m(t - O_{s,t}(\epsilon))} = \delta^{-t + O_{s,t}(\epsilon)}. \end{displaymath}
This implies Claim \eqref{cor:multiscaledecomp-kaufman:iii} with constant "$O_{s,t}(\epsilon)$" in place of "$\epsilon$". To remove the dependence on $s,t$, simply run the previous argument with $\epsilon' := \epsilon/C_{s,t}$ in place of $\epsilon$. This makes the choice of "$\tau$" also dependent on $s,t$, as stated in Proposition \ref{p:multiscaledecomp-kaufman}. Finally, Claim \eqref{cor:multiscaledecomp-kaufman:iv} is clear from construction. The proof of Proposition \ref{p:multiscaledecomp-kaufman} is complete. \end{proof}

\section{Proof of Theorem \ref{t:mainTechnical}}\label{s:9}

\label{s:proof-main-thm}

We have now all the ingredients needed to prove Theorem \ref{t:mainTechnical} whose statement we repeat here:
\begin{thm} For $s \in (0,1)$ and $t\in (s,2]$, there exists $\epsilon(s,t) > 0$ such that the following holds for all small enough $\delta \in 2^{-\N}$, depending only on $s$ and $t$. Let $\mathcal{P} \subset \mathcal{D}_{\delta}$ be a $(\delta,t,\delta^{-\epsilon})$-set with $\cup \mathcal{P} \subset [0,1)^{2}$, and let $\mathcal{T} \subset \mathcal{T}^{\delta}$ be a family of dyadic $\delta$-tubes. Assume that for every $p \in \mathcal{P}$, there exists a $(\delta,s,\delta^{-\epsilon})$-set $\mathcal{T}(p) \subset \mathcal{T}$ such that $T \cap p \neq \emptyset$ for all $T \in \mathcal{T}(p)$. Then $|\mathcal{T}| \geq \delta^{-2s - \epsilon}$.  \end{thm}

\begin{proof} We may assume without loss of generality that the given dyadic number $\delta > 0$ has the form $\delta = \Delta^{m}$, where $\Delta \in 2^{-\N}$ is another dyadic number, small in a manner depending only on $\epsilon$. Indeed, we can apply this case to a scale $\delta$ satisfying $\Delta^{m+1}<\delta\le \Delta^m$; this reduction will generate several constants depending on $\Delta$ that can be absorbed into a $\delta^{-\epsilon/2}$ term if $\delta$ is small enough. As another reduction, we may assume that the families $\mathcal{T}(p)$ have constant cardinality $|\mathcal{T}(p)| =: M \geq 1$, and that they are $(\delta,s,\delta^{-\lambda})$-sets, where $\lambda>0$ depends on $\epsilon$ in addition to $s$ and $t$. Since the statement of the theorem is preserved by making each of the "$\epsilon$" smaller, this is a formally equivalent statement.

The proof will involve a large number of positive quantities, some of which have already been mentioned, that ultimately depend on $s$ and $t$. To convince the reader (and the authors!) that there  is no circular reasoning, these are their dependencies: $\epsilon_{G}=\epsilon_{G}(s,t)$, $\eta=\eta(s,t)$, $\epsilon=\epsilon(\epsilon_{G},\eta)$, $\Delta=\Delta(\epsilon)$, $\tau=\tau(\epsilon,s,t)$, $n\le n_0(\tau)$, $\epsilon_{N}=\epsilon_{N}(\epsilon)$, $C=C(s,n,\tau)$, $C'=C'(n,\tau)$, $\lambda=\lambda(s,t,\tau,n,\epsilon_{G},\eta,C')$. Since there are only $n_0(\tau)$ possible values for $n$, any quantity depending on it effectively depends on $\tau$. Thus ultimately all parameters depend on $s,t$ only.

We start by applying Lemma \ref{l:uniformization} to $\mathcal{P}$ (which as usual we identify with $\cup\mathcal{P}$) and the sequence of scales $\Delta^{0} > \Delta^{1} > \ldots > \Delta^{m} = \delta$. The product is a $(\Delta^{i})_{i = 1}^{m}$-uniform subset $\mathcal{P}' \subset \mathcal{P} \subset \mathcal{D}_{\delta}$ with
\begin{displaymath} |\mathcal{P}'| \geq (4m^{-1} \log(1/\delta))^{-m} \cdot |\mathcal{P}| = \delta^{\frac{\log 4+\log\log(1/\Delta)}{\log(1/\Delta)}}\cdot|\mathcal{P}|.\end{displaymath}
In particular, if $\Delta \in 2^{-\N}$ is chosen small enough, depending on $\epsilon > 0$ only, then  $\mathcal{P}'$ is a $(\delta,t,\delta^{-2\epsilon})$-set. Performing such a reduction, we simplify notation by assuming, to begin with, that $\mathcal{P}$ is a $(\Delta^{i})_{i = 1}^{m}$-uniform $(\delta,t,\delta^{-\epsilon})$-set.

Next, we apply Proposition \ref{p:multiscaledecomp-kaufman} to the set $\mathcal{P}$, with the parameters $\Delta \in 2^{-\N}$ and $s \in (0,t)$ fixed above, and of course the parameter $\epsilon > 0$ in the statement of the theorem; we are free to choose $\epsilon$ small in a manner depending on $s$ and $t$, and we will exercise this freedom a little later. At this point, the application of Proposition \ref{p:multiscaledecomp-kaufman} produces a number $\tau = \tau(\epsilon,s,t) > 0$, a sequence of scales
\begin{displaymath} \delta = \Delta_{n} < \Delta_{n - 1} < \ldots < \Delta_{1} < \Delta_{0} = 1, \end{displaymath}
with $n \leq m$, each of the form $\Delta_{j} = \Delta^{i_{j}}$ for some $i_{j} \in \{0,\ldots,m\}$, and a partition of $\{1,\ldots,n\}$ into the \emph{structured} and \emph{bad} indices $\mathcal{S}$ and $\mathcal{B}$. Since $\mathcal{P}$ is $(\Delta^i)_{i=0}^m$-uniform and the $\Delta_{j}$ are integer powers of $\Delta$, we deduce that $\mathcal{P}$ is $(\Delta_i)_{i=0}^{n}$-uniform. It follows from Claims \eqref{cor:multiscaledecomp-kaufman:i},\eqref{cor:multiscaledecomp-kaufman:iv} in Proposition \ref{p:multiscaledecomp-kaufman} that $n\le n_0(\tau):= 2 \floor{\tau^{-1}}+1$.

The next step will be to apply Proposition \ref{prop3} to the set $\mathcal{P}$, and the sequence of scales $\{\Delta_{j}\}_{j = 1}^{n}$ located above. Before doing this, we further split the structured indices $\mathcal{S}$ into \emph{normal} and \emph{good} indices $\mathcal{N}$ and $\mathcal{G}$. In doing this, one has to pay great attention to the various small parameters involved. Namely, Proposition \ref{prop3} promises certain conclusions if $\epsilon_{G}, \eta$ are chosen small enough in terms of $s$ and $t$ only, and $\lambda$ is chosen small enough in terms of $s,t,\tau$ and $n$. While $\epsilon_{G}$ and $\eta$ will remain henceforth fixed (with many other parameters, notably $\epsilon$, depending on them), we will later make $\lambda$ even smaller in terms of other parameters.

With this notation, a scale index $j \in \mathcal{S}$ is declared \emph{good}, hence placed in $\mathcal{G}$, in case $\mathcal{P}$ is $(t_{j},(\Delta_{j - 1}/\Delta_{j})^{\epsilon_{G}},(\Delta_{j - 1}/\Delta_{j})^{\epsilon_{G}})$-regular between the scales $\Delta_{j}$ and $\Delta_{j - 1}$ for some $t_{j} \geq t$. Otherwise, $j \in \mathcal{S}$ is declared \emph{normal}, and placed in $\mathcal{N}$.

Are there any good indices with this definition? Yes, there are, if $\epsilon > 0$ is now chosen small enough in terms of $\epsilon_{G}$. Indeed, recall from Proposition \ref{p:multiscaledecomp-kaufman}\eqref{cor:multiscaledecomp-kaufman:ii} that $\mathcal{P}$ is a $(t_{j},(\Delta_{j - 1}/\Delta_{j})^{\epsilon})$-set for a certain index $t_{j} \in [s,2]$, and if $t_{j} > s$, then moreover $\mathcal{P}$ is $(t_{j},(\Delta_{j - 1}/\Delta_{j})^{\epsilon},(\Delta_{j - 1}/\Delta_{j})^{\epsilon})$-regular between the scales $\Delta_{j}$ and $\Delta_{j - 1}$.

Now, if
\begin{displaymath} t_{j} \geq t - \epsilon_{G}/2 > s \quad \text{and} \quad \epsilon \leq \epsilon_{G}/2, \end{displaymath}
then it is a formal consequence of $(t_{j},(\Delta_{j - 1}/\Delta_{j})^{\epsilon},(\Delta_{j - 1}/\Delta_{j})^{\epsilon})$-regularity that the set $\mathcal{P}$ is also $(t,(\Delta_{j - 1}/\Delta_{j})^{\epsilon_{G}},(\Delta_{j - 1}/\Delta_{j})^{\epsilon_{G}})$-regular between the scales $\Delta_{j}$ and $\Delta_{j - 1}$. In summary,
\begin{displaymath} t_{j} \geq t - \epsilon_{G}/2 \quad \Longrightarrow \quad j \in \mathcal{G}. \end{displaymath}
Equivalently, if $j \in \mathcal{N} = \mathcal{S} \, \setminus \, \mathcal{G}$, then $t_{j} < t - \epsilon_{G}/2$. Next, one will wonder if it ever happens that $t_{j} \geq t - \epsilon_{G}/2$, and again the answer is affirmative. This follows from Proposition \ref{p:multiscaledecomp-kaufman}\eqref{cor:multiscaledecomp-kaufman:iii}, which stated that
\begin{displaymath} \prod_{j \in \mathcal{S}} \left(\tfrac{\Delta_{j - 1}}{\Delta_{j}} \right)^{t_{j}} \geq \delta^{-t + \epsilon}. \end{displaymath}
Since $t_{j} \leq 2$ uniformly, and $t_{j} \leq t - \epsilon_{G}/2$ for $j \in \mathcal{N}$, it follows that
\begin{displaymath} \delta^{-t + \epsilon} \leq \prod_{j \in \mathcal{G}} \left(\tfrac{\Delta_{j - 1}}{\Delta_{j}} \right)^{2} \cdot \prod_{j \in \mathcal{N}} \left(\tfrac{\Delta_{j - 1}}{\Delta_{j}} \right)^{t - \epsilon_{G}/2} \leq \left[ \prod_{j \in \mathcal{G}} \left(\tfrac{\Delta_{j - 1}}{\Delta_{j}} \right) \right]^{2} \cdot \delta^{-t + \epsilon_{G}/2}, \end{displaymath}
so indeed
\begin{equation}\label{form90} \prod_{j \in \mathcal{G}} \left(\tfrac{\Delta_{j - 1}}{\Delta_{j}} \right) \geq \delta^{\epsilon/2 - \epsilon_{G}/4} \geq \delta^{-\epsilon_{G}/8}, \end{equation}
assuming here that $\epsilon \le \epsilon_{G}/4$.

Now we are finally in a position to apply Proposition \ref{prop3}, or more precisely the conclusion \eqref{form44}. We remark that since $t_{j} \geq s$ for all $j \in \mathcal{S}$, in particular for $j \in \mathcal{N}$, we know from Proposition \ref{p:multiscaledecomp-kaufman}\eqref{cor:multiscaledecomp-kaufman:ii} that $\mathcal{P}$ is an $(s,(\Delta_{j - 1}/\Delta_{j})^{\epsilon_{N}})$-set between the scales $\Delta_{j}$ and $\Delta_{j - 1}$ for all $j \in \mathcal{N}$, with $\epsilon_{N} = \epsilon\le \epsilon_{G}$. Also, from Proposition \ref{p:multiscaledecomp-kaufman}\eqref{cor:multiscaledecomp-kaufman:i} we know that $\prod_{j \in \mathcal{B}} (\Delta_{j}/\Delta_{j - 1}) \geq \delta^{\epsilon}$. Inserting all this information into the estimate \eqref{form44} yields
\begin{align*} |\mathcal{T}| & \geq \left[ \log \left(\tfrac{1}{\delta} \right) \right]^{-C} \cdot M \cdot \delta^{C' \lambda}\cdot \delta^{-s + \epsilon_{N}} \cdot \prod_{j \in \mathcal{G}} \left(\tfrac{\Delta_{j - 1}}{\Delta_{j}} \right)^{\eta} \cdot \prod_{j \in \mathcal{B}} \tfrac{\Delta_{j}}{\Delta_{j - 1}}\\
& \stackrel{\eqref{form90}}{\geq} \left[ \log \left(\tfrac{1}{\delta} \right) \right]^{-C} \cdot M \cdot \delta^{C' \lambda} \cdot \delta^{-s + \epsilon_{N}} \cdot \delta^{-\epsilon_{G}\eta/8} \cdot \delta^{\epsilon}. \end{align*}
Here $C=C(n,s,\tau)$ and $C'=C'(n,\tau)$. We also know that $M = |\mathcal{T}(p)| \geq \delta^{-s + \lambda}$ for all $p \in \mathcal{P}$. Therefore, choosing $\epsilon_{N} =\epsilon \le \epsilon_{G}\eta/100$, then choosing $\lambda\le \epsilon_{G}\eta/(100 C')$, and then taking $\delta > 0$ sufficiently small in terms of all previous parameters, we find that $|\mathcal{T}| \geq \delta^{-2s - \epsilon_{G}\eta/16}$. This completes the proof of Theorem \ref{t:mainTechnical}. \end{proof}

\appendix

\section{A generalised incidence estimate}\label{appA}

In this section, we record the details of the incidence theorem needed as a black box to prove Theorem \ref{t:improvedIncidence}. Here is the precise statement:
\begin{thm}\label{mainAppendix} Given $s \in (0,1)$ and $t \in (s,2]$, there exists an $\epsilon = \epsilon(s,t) > 0$ such that the following holds for small enough $\delta \in 2^{-\N}$, depending only on $s$ and $t$. Assume that $\mathcal{P} \subset \mathcal{D}_{\delta}$ is a $(\delta,t,\delta^{-\epsilon})$-set, and that
\begin{equation}\label{deltaHalfAssumptionIntro} |\mathcal{P} \,|_{\delta^{1/2}} \leq \delta^{-t/2 - \epsilon}. \end{equation}
Assume that $\calT \subset \mathcal{T}^{\delta}$ is a collection of dyadic $\delta$-tubes such that for every $p \in \mathcal{P}$, there exists a $(\delta,s,\delta^{-\epsilon})$-subset $\calT(p) \subset \mathcal{T}$ with the property that $T \cap p \neq \emptyset$ for all $T \in \mathcal{T}(p)$. Then either
\begin{equation}\label{alternativeIntro} |\calT| \geq \delta^{-2s - \epsilon} \quad \text{or} \quad |\calT|_{\delta^{1/2}} \geq \delta^{-s - \epsilon}. \end{equation}
\end{thm}

\subsection{Preliminaries}\label{s:prelimA} In the appendix, it will be slightly more convenient to use the following variant of point-line duality:

\begin{definition}[Point-line duality revisited]\label{pointLineDuality} In the appendix, let
\begin{equation} \label{duality-appendix} \mathbf{D} : (a,b) \mapsto \{x = ay + b : y \in \R\}. \end{equation}
\end{definition}

Earlier in the paper, $\mathbf{D}$ mapped $(a,b) \mapsto \{y = ax + b : x \in \R\}$. Dyadic tubes are defined, formally as before, $\mathcal{T}^{\delta} = \{\mathbf{D}(p) : p \in \mathcal{D}_{\delta}([-1,1) \times \R)\}$, and the slope of a dyadic $\delta$-tube $T = \mathbf{D}([a,a + \delta) \times [b,b + \delta))$ remains defined as $\sigma(T) := a$. Earlier in the paper $\sigma(T) = 0$ meant that $T$ is a roughly horizontal tube, whereas under the new notation $T$ is roughly vertical. The choice of duality plays no role in the validity of Theorem \ref{mainAppendix}: the map $\mathbf{J}(x,y) := (y,x)$ sends "horizontal" dyadic tubes to "vertical" ones, and vice versa, and also preserves the properties of $\mathcal{P}$.

We repeat here the contents of Corollary \ref{cor1}:

\begin{lemma}\label{tubesAndSlopes} Let $p \in \mathcal{D}_{\delta}$, and let $\mathcal{T}$ be a collection of dyadic $\delta$-tubes, all of which intersect $p$. If the slope set $\sigma(\mathcal{T})$ is a $(\delta,s,C)$-set for some $s \geq 0$ and $C > 0$, then also $\mathcal{T}$ is a $(\delta,s,10C)$-set. Conversely, if $\mathcal{T}$ is a $(\delta,s,C)$-set, then $\sigma(\mathcal{T})$ is a $(\delta,s,C')$-set for some $C' \sim C$. Moreover, the map $T \mapsto \sigma(T)$ is at most $10$-to-$1$, and $|\sigma(\mathcal{T})| \leq |\mathcal{T}| \leq 10|\sigma(\mathcal{T})|$.
\end{lemma}

Lemma \ref{tubesAndSlopes} has the following corollary:
\begin{cor}\label{auxLemma} Assume that $0 < \delta \leq \Delta \leq 1$ are dyadic numbers, $p \in \mathcal{D}_{\delta}$, and $\mathbf{T} = \mathbf{D}(I \times J) \in \mathcal{T}^{\Delta}$. Further, assume that $\calT(p)$ is a $(\delta,s,C)$-set of dyadic $\delta$-tubes $T$, all of which intersect $p$. Then $|\calT(p) \cap \mathbf{T}| \lesssim C \cdot |\mathcal{T}(p)| \cdot \Delta^{s}$.
\end{cor}

\begin{proof} All the tubes $T \in \mathcal{T}(p) \cap \mathbf{T}$ satisfy $\sigma(T) \in I \in \mathcal{D}_{\Delta}([-1,1))$. By Lemma \ref{tubesAndSlopes},
\begin{displaymath} |\mathcal{T}(p) \cap \mathbf{T}| \leq 10|\sigma(\mathcal{T}(p)) \cap I| \lesssim C \cdot |\sigma(\mathcal{T}(p))| \cdot \Delta^{s} \leq C \cdot |\mathcal{T}(p)| \cdot \Delta^{s}, \end{displaymath}
as desired. \end{proof}

Other results from the main text, which will also be used below, are Proposition \ref{prop2}, which is too long to restate here, and Corollary \ref{prop5}. Proposition \ref{prop2} was not needed in \cite{Orponen20} to prove the case $t = 1$ of Theorem \ref{mainAppendix}, and incorporating this additional information to the argument from \cite{Orponen20} is one of the technical novelties in the appendix.

With respect to Corollary \ref{prop5}, we will only need the following slightly weaker version. Below,  $A\lessapprox_\delta B$ stands for $A \leq C \cdot \log \left(\tfrac{1}{\delta} \right)^{C} B$, where $C\ge 1$ is universal, and likewise for $A\gtrapprox_\delta B$, $A\approx B$.
 \begin{lemma}\label{2sBound} Let $0 < s \leq t \leq 1$, and $M,C_{P},C_{T} \geq 1$. Assume that $\mathcal{P} \subset \mathcal{D}_{\delta}$ is a $(\delta,t,C_{P})$-set, and $\calT \subset \mathcal{T}^{\delta}$. Assume that for all $p \in \mathcal{P}$, there exists a $(\delta,s,C_{T})$-set $\calT(p) \subset \mathcal{T}$ such that $T \cap p \neq \emptyset$ for all $T \in \mathcal{T}$, and such that $\tfrac{M}{2} < |\mathcal{T}(p)| \leq M$. Then
 \begin{displaymath} |\calT| \gtrapprox_{\delta} (C_{P}C_{T})^{-1} \cdot M\delta^{-s}. \end{displaymath}
 \end{lemma}

\subsection{Proof of Theorem \ref{mainAppendix}}\label{s:mainProof}

Recall the objects $\mathcal{P}$, $(\mathcal{T}(p))_{p \in \mathcal{P}}$, and $\epsilon > 0$, from the statement of Theorem \ref{mainAppendix}. It will be useful to make the cardinalities of the families $\mathcal{T}(p)$ almost constant, say $|\mathcal{T}(p)| \sim M$ for all $p \in \mathcal{P}$. This can be achieved by pigeonholing, replacing $\mathcal{P}$ by a subset $\overline{\mathcal{P}}$ of cardinality $|\overline{\mathcal{P}}| \approx_{\delta} |\mathcal{P}|$. Several "refinements" of similar nature will be seen below. Note that $M \sim |\mathcal{T}(p)| \geq \delta^{-s + \epsilon}$, since $\mathcal{T}(p)$ was assumed to be a $(\delta,s,\delta^{-\epsilon})$-set of dyadic $\delta$-tubes.

Another harmless assumption is that
\begin{equation}\label{cardP} |\mathcal{P}| \leq \delta^{-t}. \end{equation}
Indeed, by Lemma \ref{lem:thin-delta-subset}, every $(\delta,t,C)$-set contained in $[-2,2]^2$ contains a $(\delta,t,O(C))$-set of cardinality $\leq \delta^{-t}$. After these initial reductions, it is time to make a counter assumption:
\begin{equation}\label{counterAss} \delta^{-2s + 3\epsilon} \lesssim \delta^{-s + 2\epsilon} \cdot M \lessapprox_{\delta} |\calT| \leq \delta^{-2s - \epsilon} \quad \text{and} \quad |\calT|_{\delta^{1/2}} \leq \delta^{-s - \epsilon}. \end{equation}
The lower bound for $|\calT|$ is not part of the counter assumption, but simply a consequence of Lemma \ref{2sBound}. Together with the (counter) assumed upper bound, it implies that
\begin{equation}\label{MBound} M \lessapprox_{\delta} \delta^{-s - 3\epsilon}. \end{equation}
In the sequel, we will use the notation $A \lessapprox B$ to signify that $A \leq C_{\epsilon}\delta^{-C\epsilon}B$, where $\epsilon > 0$ is the counter assumption parameter from \eqref{counterAss}, and $C \geq 1$ is an absolute constant. The value of the constant $C_{\epsilon}$ may depend on $\epsilon$. The notation $A \lessapprox_{\delta} B$ will no longer be used. It is also convenient to define that a set $\mathcal{P} \subset \mathcal{D}_{\delta}$, is a $(\delta,u)$-set, if it is a $(\delta,u,C_{\epsilon}\delta^{-C\epsilon})$-set for constants $C,C_{\epsilon} \geq 1$ as above, and likewise with $\delta^{1/2}$ in place of $\delta$.

A sketch of the upcoming proof is the following: if the counter assumption \eqref{counterAss} is true, then it is possible to construct a set $\mathbf{Z}\subset\R^2$ with ``product structure'' in the following sense: $\mathbf{Z}= \bigcup_{y\in\mathbf{Y}} \mathbf{X}_{y} \times \{y\}$, and an associated family $\mathcal{T}(\mathbf{Z}) \subset \mathcal{T}^{\delta}$, with the following properties:
\begin{enumerate}
\item $\mathbf{Y}$ is a $(\delta,t - s)$-set, and $\mathbf{X}_y$ is a $(\delta,s)$-set for each $y\in\mathbf{Y}$.
\item $\mathcal{T}(\mathbf{Z})$ contains a $(\delta,s)$-set $\mathcal{T}_{\mathbf{z}}$ for all $\mathbf{z} \in \mathbf{Z}$.
\item $|\mathcal{T}(\mathbf{Z})| \lessapprox \delta^{-2s}$.
\end{enumerate}
Thus, $\mathbf{Z},\mathcal{T}(\mathbf{Z})$ satisfy hypotheses similar to those of $\mathcal{P},\mathcal{T}$, but $\mathbf{Z}$ has the additional ``product structure''. Under this extra information, Proposition \ref{productProp} below (quoted from \cite{Orponen20}) implies that $|\mathcal{T}(\mathbf{Z})| \geq \delta^{-2s - \eta(s,t)}$ for some $\eta(s,t) > 0$ depending only on $s,t$. Consequently (3), hence \eqref{counterAss}, cannot hold for arbitrarily small $\epsilon > 0$, the threshold depending only on $s,t$. Theorem \ref{mainAppendix} follows.

\subsection{Considerations at scale $\delta^{1/2}$}\label{scaleDeltaHalf} To make notation prettier, write $\Delta := \delta^{1/2}$. Let $\calQ := \mathcal{D}_{\Delta}(\mathcal{P})$ be the collection of dyadic $\Delta$-squares containing at least one element of $\mathcal{P}$. Then $|\calQ| \lessapprox \Delta^{-t}$ by the assumption \eqref{deltaHalfAssumptionIntro}, but also $|\calQ| \gtrapprox \Delta^{-t}$, since $\mathcal{P}$ is a $(\delta,t)$-set. Hence
\begin{equation}\label{form39Appendix} |\mathcal{Q}| \approx \Delta^{-t}. \end{equation}
Write $\mathcal{P} \cap Q := \{p \in \mathcal{P} : p \subset Q\}$, for $Q \in \mathcal{Q}$. Since $|\mathcal{P} \cap Q| \lessapprox |\mathcal{P}| \cdot \Delta^{t} \leq \Delta^{-t}$ for every $Q \in \mathcal{Q}$ by the $(\delta,t)$-set assumption, combined with \eqref{cardP}, one sees that a subset $\mathcal{P}' \subset \mathcal{P}$ of the form $\mathcal{P}' = \bigcup_{Q \in \mathcal{Q'}} \mathcal{P} \cap Q$ with $|\mathcal{P}'| \approx \delta^{-t}$ is covered by "heavy" squares $Q \in \calQ'$ with
\begin{equation}\label{form2Appendix} |\mathcal{P} \cap Q| \approx \Delta^{-t}. \end{equation}
Since $\mathcal{P}'$ satisfies all the same assumptions as $\mathcal{P}$, with slightly worse constants, we may assume that $\mathcal{P}' = \mathcal{P}$; thus, we assume that \eqref{form2Appendix} holds for all squares $Q \in \calQ$. Now \eqref{form2Appendix} implies that $\calQ$ is a $(\Delta,t)$-set: if $Q_{r} \in \mathcal{D}_{r}$ with $\Delta \leq r \leq 1$, then
\begin{displaymath} |\mathcal{Q} \cap Q_{r}|_{\Delta} \stackrel{\eqref{form2Appendix}}{\approx} \Delta^{t} \cdot |\mathcal{P} \cap Q_{r}| \lessapprox \Delta^{t} \cdot |\mathcal{P}| \cdot r^{t} \approx |\mathcal{Q}| \cdot r^{t}. \end{displaymath}
Now, let $\calT_{\Delta} := \mathcal{T}^{\Delta}(\mathcal{T})$ be a minimal cover of $\calT$ by dyadic $\Delta$-tubes. By \eqref{counterAss},
\begin{equation}\label{form3} |\calT_{\Delta}| = |\mathcal{T}|_{\Delta} \lessapprox \delta^{-s}. \end{equation}
We next claim that for each $Q\in\mathcal{Q}$, the family $\mathcal{T}_{\Delta}$ contains a $(\Delta,s)$-set all of whose tubes intersect $Q$. This is virtually the same argument we saw during the proof of Proposition \ref{p:induction-on-scales}, but we repeat the details. Fix $Q \in \mathcal{Q}$.  By applying Proposition \ref{prop2} at scale $\Delta = \delta^{1/2}$, and to the family $\mathcal{P} \cap Q$ in place of $\mathcal{P}$, we find a subset $\mathcal{P}_{Q} \subset \mathcal{P} \cap Q$ of cardinality $|\mathcal{P}_{Q}| \approx |\mathcal{P} \cap Q| \approx \Delta^{-t}$, and a family of dyadic $\Delta$-tubes $\mathcal{T}_{\Delta}(Q) \subset \mathcal{T}_{\Delta}$ intersecting $Q$ such that the following properties hold:
\begin{itemize}
\item[(H1) \phantomsection \label{H1}] $\mathcal{T}_{\Delta}(Q)$ is a $(\Delta,s)$-set.
\item[(H2) \phantomsection \label{H2}] There exists a constant $H_{Q} \approx M \cdot |\mathcal{P}_{Q}|/|\mathcal{T}_{\Delta}(Q)| \approx \delta^{-s - t/2}/|\mathcal{T}_{\Delta}(Q)|$, such that
\begin{displaymath} |\{(p,T) \in \mathcal{P}_{Q} \times \mathcal{T} : T \in \mathcal{T}(p) \cap \mathbf{T}\}| \gtrsim H_{Q}, \qquad \mathbf{T} \in \mathcal{T}_{\Delta}(Q). \end{displaymath}
\end{itemize}
All the tubes in $\mathcal{T}_{\Delta}(Q)$ intersect $Q$, so $|\mathcal{T}_{\Delta}(Q)| \in \{1,\ldots,100\Delta^{-1}\}$. By the pigeonhole principle, we may find $\mathbf{M} \geq 1$, and a subset $\overline{\mathcal{Q}} \subset \mathcal{Q}$ with the properties
\begin{equation}\label{formA38} |\overline{\mathcal{Q}}| \approx |\mathcal{Q}| \quad \text{and} \quad |\mathcal{T}_{\Delta}(Q)| \sim \mathbf{M} \text{ for } Q \in \overline{\mathcal{Q}}. \end{equation}
In analogy with the bound  $M \lessapprox \delta^{-s}$ for $M \sim |\mathcal{T}(p)|$ (recall \eqref{MBound}), we now record that
\begin{equation}\label{secondRed} |\mathcal{T}_{\Delta}(Q)| \sim \mathbf{M} \lessapprox \Delta^{-s}, \qquad Q \in \overline{\mathcal{Q}}. \end{equation}
Indeed, since $\mathcal{T}_{\Delta}(Q)$ is a $(\Delta,s)$-set for all $Q \in \overline{\mathcal{Q}}$, with cardinality $\sim \mathbf{M}$, and $\overline{\mathcal{Q}}$ is a $(\Delta,t)$-set, a combination of Lemma \ref{2sBound} (applied to $\overline{\mathcal{Q}}$ and the families $\mathcal{T}_{\Delta}(Q)$) and the counter assumption \eqref{counterAss} imply that
\begin{displaymath} \mathbf{M}\cdot \Delta^{-s} \lessapprox |\mathcal{T}_{\Delta}| = |\mathcal{T}|_{\Delta} \lessapprox \delta^{-s}. \end{displaymath}
This implies \eqref{secondRed}. Of course the fact that $\mathcal{T}_{\Delta}(Q)$ is a $(\Delta,s)$-set, by \nref{H1}, alone implies that $\mathbf{M} \gtrapprox \Delta^{-s}$, and hence $|\mathcal{T}_{\Delta}| \gtrapprox \delta^{-s}$ by the previous inequality. We record these observations for future use:
\begin{equation}\label{form3+} |\mathcal{T}_{\Delta}(Q)| \sim \mathbf{M} \approx \Delta^{-s} \text{ for } Q \in \overline{\mathcal{Q}} \quad \text{and} \quad |\calT_{\Delta}| \approx \delta^{-s}. \end{equation}
In the sequel, there will be no difference between the collections $\overline{\mathcal{Q}}$ and $\mathcal{Q}$ (since already the squares in $\overline{\mathcal{Q}}$ cover $\approx |\mathcal{P}|$ squares in $\mathcal{P}$), so we rename $\overline{\mathcal{Q}}$ to $\mathcal{Q}$  for now; the notation will be recycled soon enough. Further, in order to avoid having to remember the difference of $\mathcal{P}_{Q}$ and $\mathcal{P} \cap Q$, we redefine $\mathcal{P} := \bigcup_{Q \in \overline{\mathcal{Q}}} \mathcal{P}_{Q}$. Thus
\begin{equation}\label{form1} \mathcal{P}_{Q} = \mathcal{P} \cap Q = \{p \in \mathcal{P} : p \subset Q\}, \qquad Q \in \mathcal{Q}. \end{equation}
Then $\mathcal{P}$ remains a $(\delta,t)$-set, the squares $Q \in \mathcal{Q}$ remain "heavy" in the sense \eqref{form2Appendix}, and the property \nref{H2} above holds with $\mathcal{P} \cap Q$ in place of $\mathcal{P}_{Q}$.

Different tubes $\mathbf{T} \in \mathcal{T}_{\Delta}$ may \emph{a priori} contain different numbers of tubes from $\mathcal{T}$. It will be advantageous to "freeze" the cardinality of $\mathcal{T} \cap \mathbf{T} = \{T \in \mathcal{T} : T \subset \mathbf{T}\}$, $\mathbf{T} \in \mathcal{T}_{\Delta}$, by further piegonholing. This can be carried out by refining the families $\mathcal{Q}$ and $\mathcal{T}_{\Delta}(Q)$ slightly -- in such a way that \nref{H1}-\nref{H2} and \eqref{form3+} persist. First define
\begin{displaymath} \mathcal{I}(\mathcal{Q},\mathcal{T}_{\Delta}) := \{(Q,\mathbf{T}) \in \mathcal{Q} \times \mathcal{T}_{\Delta} : \mathbf{T} \in \mathcal{T}_{\Delta}(Q)\}.\end{displaymath}
Then (recall \eqref{formA38}), one has
\begin{displaymath} |\mathcal{I}(\mathcal{Q},\mathcal{T}_{\Delta})| \approx |\mathcal{Q}| \cdot \mathbf{M} \stackrel{\eqref{form39Appendix}+\eqref{form3+}}{\approx} \Delta^{-s - t}. \end{displaymath}
Let
\begin{displaymath} \mathcal{T}_{\Delta,j} := \{\mathbf{T} \in \mathcal{T}_{\Delta} : 2^{j - 1} < |\mathcal{T} \cap \mathbf{T}| \leq 2^{j}\}. \end{displaymath}
Since $|\mathcal{T} \cap \mathbf{T}| \leq \delta^{-1}$ for all $\mathbf{T} \in \mathcal{T}_{\Delta}$, one has
\begin{displaymath} \Delta^{-s - t} \approx |\mathcal{I}(\mathcal{Q},\mathcal{T}_{\Delta})| = \sum_{2^{j} \leq \delta^{-1}} |\{(Q,\mathbf{T}) \in \mathcal{Q} \times \mathcal{T}_{\Delta,j} : \mathbf{T} \in \mathcal{T}_{\Delta}(Q)\}|. \end{displaymath}
Therefore, one may pick $j \leq \log \delta^{-1}$ such that, writing
\[
\overline{\mathcal{T}}_{\Delta} := \mathcal{T}_{\Delta,j},\quad \overline{\mathcal{T}}_{\Delta}(Q) := \overline{\mathcal{T}}_{\Delta} \cap \mathcal{T}_{\Delta}(Q), \qquad Q \in \mathcal{Q},
\]
one has
\begin{displaymath} \sum_{Q \in \mathcal{Q}} |\overline{\mathcal{T}}_{\Delta}(Q)| = |\{(Q,\mathbf{T}) \in \mathcal{Q} \times \overline{\mathcal{T}}_{\Delta} : \mathbf{T} \in \overline{\mathcal{T}}_{\Delta}(Q)\}| \approx \Delta^{-s - t}. \end{displaymath}
Write $\mathbf{N} := 2^{j}$ for this index "$j$", so $|\mathcal{T} \cap \mathbf{T}| \sim \mathbf{N}$ for all $\mathbf{T} \in \overline{\mathcal{T}}_{\Delta}$. Since one has $|\overline{\mathcal{T}}_{\Delta}(Q)| \leq |\mathcal{T}_{\Delta}(Q)| \approx \Delta^{-s}$ uniformly in $Q \in \mathcal{Q}$ by \eqref{form3+}, and $|\mathcal{Q}| \approx \Delta^{-t}$, one may pick a subset $\overline{\mathcal{Q}} \subset \mathcal{Q}$ of cardinality $|\overline{\mathcal{Q}}| \approx |\mathcal{Q}|$ such that $|\overline{\mathcal{T}}_{\Delta}(Q)| \approx \Delta^{-s}$ for all $Q \in \overline{\mathcal{Q}}$. Since $\mathcal{T}_{\Delta}(Q)$ was a $(\Delta,s)$-set of dyadic $\Delta$-tubes intersecting $Q$, the same remains true for $\overline{\mathcal{T}}_{\Delta}(Q)$. Thus, the family $\overline{\mathcal{T}}_{\Delta}$ contains a $(\Delta,s)$-set of $\Delta$-tubes incident to every square in the $(\Delta,t)$-set $\overline{\mathcal{Q}}$. Lemma \ref{2sBound} therefore implies
\begin{equation}\label{formA47} \Delta^{-2s} \lessapprox |\overline{\mathcal{T}}_{\Delta}| \leq |\mathcal{T}_{\Delta}| \stackrel{\eqref{counterAss}}{\lessapprox} \Delta^{-2s}. \end{equation}
Based on this, we claim that
\begin{equation}\label{formA37} |\mathcal{T} \cap \mathbf{T}| \sim \mathbf{N} \lessapprox \delta^{-s}, \qquad \mathbf{T} \in \overline{\mathcal{T}}_{\Delta}. \end{equation}
Indeed, by the disjointness of the families $\mathcal{T} \cap \mathbf{T} \subset \mathcal{T}$, for $\mathbf{T} \in \overline{\mathcal{T}}_{\Delta}$, one has
\begin{displaymath} \delta^{-2s} \stackrel{\eqref{counterAss}}{\gtrapprox} |\mathcal{T}| \geq \sum_{\mathbf{T} \in \overline{\mathcal{T}}_{\Delta}} |\mathcal{T} \cap \mathbf{T}| \sim |\overline{\mathcal{T}}_{\Delta}| \cdot \mathbf{N} \stackrel{\eqref{formA47}}{\approx} \Delta^{-2s} \cdot \mathbf{N}, \end{displaymath}
and \eqref{formA37} follows. We pause for a moment to record the achievements so far:
\begin{itemize}
\item[(G1) \phantomsection \label{G1}] $\overline{\mathcal{Q}} \subset \mathcal{Q}$ satisfies $|\overline{\mathcal{Q}}| \approx \Delta^{-t}$, and $|\mathcal{P} \cap Q| \approx \Delta^{-t}$ for all $Q \in \overline{\mathcal{Q}}$.
\item[(G2) \phantomsection \label{G2}] Every tube $\mathbf{T} \in \overline{\mathcal{T}}_{\Delta}$ satisfies $|\mathcal{T} \cap \mathbf{T}| \sim \mathbf{N} \lessapprox \delta^{-s}$.
\item[(G3) \phantomsection \label{G3}] For every square $Q \in \overline{\mathcal{Q}}$, there corresponds a $(\Delta,s)$-set $\overline{\mathcal{T}}_{\Delta}(Q) \subset \overline{\mathcal{T}}_{\Delta}$ of cardinality $\approx \mathbf{M} \approx \Delta^{-s}$ such that $\mathbf{T} \cap Q \neq \emptyset$ for all $\mathbf{T} \in \overline{\mathcal{T}}_{\Delta}(Q)$.
\item[(G4) \phantomsection \label{G4}] $|\mathcal{T}| \approx \delta^{-2s}$ and $|\overline{\mathcal{T}}_{\Delta}| \approx \Delta^{-2s}$.
\end{itemize}
At this point, we simplify writing by dropping all the "overlines" from the notation for the families $\overline{\mathcal{Q}},\overline{\mathcal{T}}_{\Delta}(Q)$, and $\overline{\mathcal{T}}_{\Delta}$. These families satisfy the same properties as $\mathcal{Q},\mathcal{T}_{\Delta}(Q)$, and $\mathcal{T}_{\Delta}$, up to worse constants, and additionally \nref{G2} holds for all $\mathbf{T} \in \overline{\mathcal{T}}_{\Delta}$. (One thing which ceases to hold after these notational changes is that every tube in $\mathcal{T}$, or $\mathcal{T}(p)$, is contained in one of the tubes from $\mathcal{T}_{\Delta} = \overline{\mathcal{T}}_{\Delta}$. This information will not be used. It will follow from \nref{H2} that the tubes in $\overline{\mathcal{T}}_{\Delta}$ contain "sufficiently" many tubes from $\mathcal{T}$.)

\subsection{Refining the families $\mathcal{T}_{\Delta}(Q)$ further} Fix $Q \in \mathcal{Q}$, and write
\begin{displaymath} \sigma(Q) := \{\sigma(\mathbf{T}) : \mathbf{T} \in \mathcal{T}_{\Delta}(Q)\} \subset (\Delta \cdot \Z) \cap [-1,1) \end{displaymath}
for the slope set of the family $\mathcal{T}_{\Delta}(Q)$. Then $\sigma(\mathcal{T})$ is a $(\Delta,s)$-set by Lemma \ref{tubesAndSlopes}. For $\sigma = \sigma(\mathbf{T}) \in \sigma(Q)$, write $\pi_{\sigma}(x,y) := x - \sigma y$ for the orthogonal projection (up to rescaling) to the direction "perpendicular" to $\mathbf{T}$. Informally, the next lemma says that "in at least half of the directions $\sigma \in \sigma(Q)$, the family $\pi_{\sigma}(\cup(\mathcal{P} \cap Q))$, and all its reasonably large subsets, contain a $(\delta,s)$-set". This is not literally true, since $\diam(\pi_{\sigma}(\cup (\mathcal{P} \cap Q))) \lesssim \Delta$, and with our definition of $(\delta,s)$-sets, no short interval can contain a $(\delta,s)$-set. The more precise conclusion, therefore, is that a $\Delta^{-1}$-rescaled version $\pi_{\sigma}(\cup (\mathcal{P} \cap Q))$ contains a $(\Delta,s)$-set.

To make the statement precise, it will be useful to write out that $\sigma(Q)$ is a $(\Delta,s,\Delta^{-\mathbf{A}\epsilon})$-set, and the $\Delta^{-1}$-rescaled copy of $\mathcal{P} \cap Q$ is a $(\Delta,t,\delta^{-\mathbf{A}\epsilon})$-set, where $\mathbf{A} \geq 1$ is the absolute constant lurking behind the "$\approx$"-notation. As before, we write $S_{Q}$ for the homothety $Q\to[0,1)^{2}$.

\begin{lemma}\label{projections} There exists a subset $\Sigma=\Sigma(Q) \subset \sigma(Q)$ with $|\Sigma| \geq \tfrac{1}{2}|\sigma(Q)|$ such that the following holds for all $\sigma \in \Sigma$: if $\mathcal{P}_Q\subset \mathcal{P} \cap Q$ is an arbitrary subset of cardinality $|\mathcal{P}_Q| \geq \delta^{\mathbf{B}\epsilon}|\mathcal{P} \cap Q|$, for some $\mathbf{B} \geq 1$, then $\pi_{\sigma}(\cup S_Q(\mathcal{P}_Q))$ contains a $(\Delta,s,\delta^{-\mathbf{C}(\mathbf{A} + \mathbf{B})\epsilon})$-set, where $\mathbf{C} \geq 1$ is absolute. \end{lemma}

\begin{proof} The proof is a variation of the standard "potential theoretic" argument, invented by Kaufman \cite{Ka}. Let $\overline{\mathcal{P}}^{Q} := \{S_{Q}(p) : p \in \mathcal{P}\cap Q\} \subset \mathcal{D}_{\Delta}$. Since $\mathcal{P}$ is a $(\delta,t)$-set and all the squares $Q$ are "heavy" in the sense of \eqref{form2Appendix} (which continues to hold after all refinements), we see that $\overline{\mathcal{P}}^{Q}$ is a $(\Delta,t)$-set with $|\overline{\mathcal{P}}^{Q}| \approx \Delta^{-t}$. We denote elements of $\overline{\mathcal{P}}^{Q}$ by "$q$". Let $\mathcal{L}^d$ denote Lebesgue measure on $\R^d$, and consider the probability measures
\begin{displaymath} \mu := \frac{1}{|\overline{\mathcal{P}}^{Q}|} \sum_{q \in \overline{\mathcal{P}}^{Q}} \frac{\mathcal{L}^{2}|_{q}}{\Delta^{2}} \quad \text{and} \quad \nu := \frac{1}{|\sigma(Q)|} \sum_{\sigma \in \sigma(Q)} \frac{\mathcal{L}^{1}|_{[\sigma,\sigma + \Delta)}}{\Delta}. \end{displaymath}
It is not hard to check (cf. Remark \ref{r:delta-set-via-measure}) that $\mu(B(x,r)) \lessapprox r^{t}$ and $\nu(B(x,r)) \lessapprox r^{s}$ for all $r \in (0,1]$. Let $I_{s}$ denote the $s$-dimensional Riesz energy, that is, for a finite Borel measure $\rho$ on $\R^d$,
\[
I_{s}(\rho) = \int |x-y|^{-s}\,d\rho(x)\,d\rho(y).
\]
By a standard argument going back to \cite{Ka},
\begin{displaymath} \int_{-1}^{1} I_{s}(\pi_{\sigma}\mu) \, d\nu(\sigma) = \iint \left[ \int_{-1}^{1} \frac{d\nu(\sigma)}{|\pi_{\sigma}(x) - \pi_{\sigma}(y)|^{s}} \right] \, d\mu(x) \, d\mu(y) \lessapprox 1. \end{displaymath}
Indeed, the inner integral can be estimated by $\lessapprox 1/|x - y|^{s}$ by the $s$-dimensional Frostman property of $\nu$. Then
\begin{displaymath} \int_{-1}^{1} I_{s}(\pi_{\sigma}\mu) \, d\nu(\sigma) \lessapprox \int \left[\int \frac{d\mu(x)}{|x - y|^{s}} \right] \, d\mu(y) \lessapprox 1, \end{displaymath}
since the inner integral is again bounded by $\lessapprox 1$ for every $y \in \R^{2}$, recalling that $\mu$ satisfies a $t$-dimensional Frostman condition, and $t > s$. The implicit constants here are of the form $\delta^{-C_{1}\mathbf{A}\epsilon}$ for some absolute $C_{1} \geq 1$. Consequently, by Chebyshev's inequality,
\begin{displaymath} \nu(\{\sigma \in [-1,1] : I_{s}(\pi_{\sigma}\mu) \geq \delta^{-2C_{1}\mathbf{A}\epsilon}\}) \leq \tfrac{1}{2}, \end{displaymath}
provided $\delta > 0$ is small enough. Let $\Sigma' := \{\sigma \in [-1,1] : I_{s}(\pi_{\sigma}\mu) \leq \delta^{-2C_{1}\mathbf{A}\epsilon}\}$. Then one needs $\geq \tfrac{1}{2}|\sigma(Q)|$ intervals of the form $[\sigma,\sigma + \Delta)$, $\sigma \in \sigma(Q)$, to cover $\Sigma'$. The left end-points of these intervals form a set $\Sigma \subset \sigma(Q)$ with $|\Sigma| \geq \tfrac{1}{2}|\sigma(Q)|$. We claim $\Sigma$ satisfies the statement of Lemma \ref{projections}.

For every $\sigma \in \Sigma$, by definition, there exists a point $\sigma' \in [\sigma,\sigma + \Delta)$ with
\begin{equation}\label{formA20} \iint \frac{d\mu(x) \, d\mu(y)}{|\pi_{\sigma'}(x) - \pi_{\sigma'}(y)|^{s}} = I_{s}(\pi_{\sigma'}\mu) \leq \delta^{-2C_{1}\mathbf{A}\epsilon}. \end{equation}
Now, if $\mathcal{P}_Q\subset \mathcal{P} \cap Q$ is a subset of cardinality $|\mathcal{P}_Q| \geq \delta^{\mathbf{B}\epsilon}|\mathcal{P} \cap Q|$, as in the statement, then the probability measure $\bar{\mu}$, defined in the obvious way by restricting and re-normalising $\mu$ to the subset $S_Q(\mathcal{P}_Q) \subset \overline{\mathcal{P}}^{Q}$, still satisfies \eqref{formA20}, up to multiplying the right hand side by $\delta^{-2\mathbf{B}\epsilon}$. It follows that $\calH^{s}_{\infty}(\pi_{\sigma'}(\spt \bar{\mu})) \gtrsim \delta^{2(C_{1}\mathbf{A} + \mathbf{B})\epsilon}$ (see e.g. \cite[pp. 109--112]{zbMATH01249699} for the standard argument), and hence $\pi_{\sigma'}(\spt \bar{\mu}) = \pi_{\sigma'}(\cup S_Q(\mathcal{P}_Q))$ contains a $(\Delta,s,\delta^{-2(C_{1}\mathbf{A} + \mathbf{B})\epsilon})$-set by Proposition \ref{deltasSet}. Finally, using $|\sigma - \sigma'| \leq \Delta$, the conclusion remains valid for $\pi_{\sigma}(\cup S_Q(\mathcal{P}_Q))$. This proves the lemma. \end{proof}

Define
\begin{equation}\label{formA44} \calT_{\Delta}^{\pi}(Q) := \{\mathbf{T} \in \mathcal{T}_{\Delta}(Q) : \sigma(\mathbf{T}) \in \Sigma(Q)\}, \qquad Q \in \mathcal{Q}, \end{equation}
where $\Sigma(Q) \subset \sigma(Q)$ is the set of good directions of cardinality $|\Sigma(Q)| \geq \tfrac{1}{2}|\sigma(Q)| \sim |\mathcal{T}_{\Delta}(Q)|$ indicated by Lemma \ref{projections} (the relation $|\sigma(Q)| \sim |\mathcal{T}_{\Delta}(Q)|$ follows from Lemma \ref{tubesAndSlopes}). Then the families $\mathcal{T}_{\Delta}^{\pi}(Q)$, $Q \in \mathcal{Q}$, remain $(\Delta,s)$-sets of cardinality $\approx \Delta^{-s}$, and the properties \nref{G1}-\nref{G4} remain valid if $\mathcal{T}_{\Delta}(Q)$ gets replaced by $\mathcal{T}_{\Delta}^{\pi}(Q)$ (only \nref{G3} is affected). This completes the first refinement of $\calT_{\Delta}(Q)$. The conclusion is that the sets $\mathcal{P} \cap Q$, $Q \in \calQ$, and their large subsets, have "$s$-dimensional" projections in every direction perpendicular to the tubes $\mathbf{T} \in \mathcal{T}_{\Delta}^{\pi}(Q)$. The symbol "$\pi$" will stay as a reminder of this fact.

Next: a further refinement of $\mathcal{T}^{\pi}_{\Delta}(Q)$. This refinement is concerned with the distribution of the squares in the family $\{Q \in \calQ : \mathbf{T} \in \mathcal{T}_{\Delta}^{\pi}(Q)\}$, for a fixed tube $\mathbf{T} \in \mathcal{T}_{\Delta}$. Roughly speaking, we claim that "without loss of generality", these sets are $(\Delta,t - s)$-sets.

To formalise such thoughts, write $d(Q,Q')$ for the distance between the midpoints of $Q,Q'$, and consider the following inequality (see explanations below):
\begin{equation}\label{formA6}
\begin{split}
\sum_{\mathbf{T} \in \calT_{\Delta}} \mathop{\sum_{Q,Q' \in \calQ}}_{Q \neq Q'}  \frac{\mathbf{1}_{\calT_{\Delta}^{\pi}(Q) \cap \calT_{\Delta}^{\pi}(Q')}(\mathbf{T})}{d(Q,Q')^{t - s}} &=\mathop{\sum_{Q,Q' \in \calQ}}_{Q \neq Q'} \frac{|\calT_{\Delta}^{\pi}(Q) \cap \calT_{\Delta}^{\pi}(Q')|}{d(Q,Q')^{t - s}}\\
& \lessapprox \mathop{\sum_{Q,Q' \in \calQ}}_{Q \neq Q'} \frac{1}{d(Q,Q')^{t}} \lessapprox \Delta^{-2t}.
\end{split}
\end{equation}
The first "$\lessapprox$"-inequality uses the fact that $\calT_{\Delta}^{\pi}(Q)$  is a $(\Delta,s)$-set of tubes with $|\calT_{\Delta}^{\pi}(Q)| \approx \Delta^{-s}$: since $\mathbf{T} \in \calT_{\Delta}^{\pi}(Q) \cap \calT_{\Delta}^{\pi}(Q')$ implies that $Q \cap \mathbf{T} \neq \emptyset \neq Q' \cap \mathbf{T}$, this can only hold for $\lessapprox 1/d(Q,Q')^{s}$ choices of $\mathbf{T} \in \calT_{\Delta}(Q)$. The second "$\lessapprox$"-inequality in \eqref{formA6} follows from the fact that $\mathcal{Q}$ is a $(\Delta,t)$-set with $|\mathcal{Q}| \approx \Delta^{-t}$, by considering for each $Q$ and dyadic number $\overline{\Delta}\in [\Delta,1]$ those $Q'$ with $d(Q,Q')\sim \overline{\Delta}$.

Fix a large but absolute constant $C \geq 1$. It follows from \eqref{formA6} and Chebyshev's inequality that
\begin{equation}\label{formA21} \mathop{\sum_{Q,Q' \in \calQ}}_{Q \neq Q'}  \frac{\mathbf{1}_{\calT_{\Delta}^{\pi}(Q) \cap \calT_{\Delta}^{\pi}(Q')}(\mathbf{T})}{d(Q,Q')^{t - s}} \geq \Delta^{-C\epsilon + 2(s - t)} \end{equation}
can only hold for $\lessapprox \Delta^{C\epsilon - 2s}$ tubes $\mathbf{T} \in \calT_{\Delta}$. Recalling that $|\calT_{\Delta}| \approx \Delta^{-2s}$ by \nref{G4}, this roughly says that the tubes satisfying \eqref{formA21} are exceptional. We need a more accurate statement: we claim that  there exists a subset $\overline{\mathcal{Q}} \subset \mathcal{Q}$ with $|\overline{\mathcal{Q}}| \geq \tfrac{1}{2}|\mathcal{Q}|$ such that for all $Q_{0} \in \overline{\mathcal{Q}}$, at most half of the tubes $\mathbf{T} \in \calT^{\pi}_{\Delta}(Q_{0})$ satisfy \eqref{formA21}, if $C \geq 1$ is chosen large enough.

Assume the converse: for at least half of the squares $Q_{0} \in \mathcal{Q}$, say those in $\mathcal{Q}_{\mathrm{bad}}$, at least $\tfrac{1}{2}|\mathcal{T}^{\pi}_{\Delta}(Q_{0})|$ tubes  $\mathbf{T} \in \mathcal{T}_{\Delta}^{\pi}(Q_{0})$ satisfy \eqref{formA21}. Then Lemma \ref{2sBound} can be applied at scale $\Delta$, with the $(\Delta,t)$-set $\mathcal{Q}_{\mathrm{bad}}$, and the bad parts of $\mathcal{T}^{\pi}_{\Delta}(Q_{0})$, which are evidently still $(\Delta,s)$-sets. The conclusion is that the total number of tubes in $\calT_{\Delta}$ satisfying \eqref{formA21} is $\gtrapprox \Delta^{-2s}$. On the other hand, it was noted above that the number of tubes in $\mathcal{T}_{\Delta}$ satisfying \eqref{formA21} is $\lessapprox \Delta^{C\epsilon - 2s}$. If $C \geq 1$ is sufficiently large, depending on the constants in the "$\approx$" notation, and $\delta > 0$ is small enough, a contradiction ensues. The size of $C \geq 1$ in this argument is an absolute constant, so we may abbreviate $A \leq \delta^{-C\epsilon}B$ to $A \lessapprox B$ in the sequel.

Summarizing, assuming that $C \geq 1$ is sufficiently large, and $\delta > 0$ is sufficiently small, the converse of \eqref{formA21} holds for all $Q_{0} \in \overline{\mathcal{Q}}$, where $|\overline{\mathcal{Q}}| \geq \tfrac{1}{2}|\mathcal{Q}|$, and for $\geq \tfrac{1}{2}|\mathcal{T}^{\pi}_{\Delta}(Q_{0})|$ tubes in $\calT_{\Delta}^{\pi}(Q_{0})$. We replace $\mathcal{Q}$ by $\overline{\mathcal{Q}}$, and the families $\mathcal{T}^{\pi}_{\Delta}(Q)$ by their good subsets, without altering notation. All the properties \nref{G1}-\nref{G4} remain valid. With the new notation,
\begin{equation}\label{formA31} \mathop{\sum_{Q,Q' \in \calQ}}_{Q \neq Q'}  \frac{\mathbf{1}_{\calT^{\pi}_{\Delta}(Q) \cap \calT^{\pi}_{\Delta}(Q')}(\mathbf{T})}{d(Q,Q')^{t - s}} \lessapprox \Delta^{2(s - t)}, \qquad \mathbf{T} \in \mathcal{T}^{\pi}_{\Delta}(Q_{0}), \, Q_{0} \in \mathcal{Q}. \end{equation}
(The careful reader will notice that the converse of \eqref{formA21} would be a stronger statement than \eqref{formA31}: the families $\mathcal{Q}$ and $\mathcal{T}^{\pi}_{\Delta}(Q)$ in \eqref{formA31} have already been reduced to subsets.)

The next aim is to find a fixed tube $\mathbf{T}_{0} \in \mathcal{T}_{\Delta}$ with the property
\begin{equation}\label{formA41} |\mathbf{T}_{0}(\mathcal{Q})| := |\{Q \in \mathcal{Q} : \mathbf{T}_{0} \in \mathcal{T}^{\pi}_{\Delta}(Q)\}| \gtrapprox \Delta^{s - t}. \end{equation}
This is easy, using $|\mathcal{T}_{\Delta}| \approx \Delta^{-2s}$, $|\mathcal{Q}| \approx \Delta^{-t}$, and $|\mathcal{T}^{\pi}_{\Delta}(Q)| \approx \Delta^{-s}$ for $Q \in \mathcal{Q}$:
\begin{displaymath} \frac{1}{|\mathcal{T}_{\Delta}|} \sum_{\mathbf{T} \in \mathcal{T}_{\Delta}} |\{Q \in \mathcal{Q} : \mathbf{T} \in \mathcal{T}^{\pi}_{\Delta}(Q)\}| = \frac{1}{|\mathcal{T}_{\Delta}|} \sum_{Q \in \mathcal{Q}} |\mathcal{T}_{\Delta}^{\pi}(Q)| \approx \frac{|\mathcal{Q}| \cdot \Delta^{-s}}{\Delta^{-2s}} \approx \Delta^{s - t}.  \end{displaymath}
Therefore, the average tube $\mathbf{T}_{0} \in \mathcal{T}_{\Delta}$ satisfies \eqref{formA41}.

We now claim that, as a corollary of \eqref{formA31} and \eqref{formA41}, the family $\mathbf{T}_{0}(\mathcal{Q}) \subset \{Q \in \mathcal{Q} : Q \cap \mathbf{T}_{0} \neq \emptyset\}$ contains a $(\Delta,t - s)$-set. Indeed, \eqref{formA31} yields
\begin{displaymath} \mathop{\sum_{Q,Q' \in \mathbf{T}_{0}(\mathcal{Q})}}_{Q \neq Q'} \frac{1}{d(Q,Q')^{t - s}} \lessapprox \Delta^{2(s - t)}. \end{displaymath}
Let
\[
\mathbf{T}'_{0}(\mathcal{Q}) = \Big\{ Q \in \mathbf{T}_{0}(\mathcal{Q}): \sum_{Q'\in\mathbf{T}_0(\mathcal{Q})\setminus\{ Q\}} d(Q,Q')^{s - t} \le \Delta^{s - t- C\epsilon}\Big\}.
\]
By Chebyshev's inequality and \eqref{formA41}, if the constant $C$ is chosen large enough, then $|\mathbf{T}'_{0}(\mathcal{Q})|\ge \tfrac{1}{2}|\mathbf{T}_{0}(\mathcal{Q})|\gtrapprox \Delta^{s-t}$. Considering, for each $Q\in\mathbf{T}'_0(\mathcal{Q})$ and each dyadic number $\overline{\Delta}\in [\Delta,1]$, those $Q'\in\mathbf{T}'_0(\mathcal{Q})$ with $d(Q,Q')\sim \overline{\Delta}$, we see that $\mathbf{T}'_{0}(\mathcal{Q})$ is a $(\Delta,t-s)$-set, as desired.
From now on, we simplify notation by denoting the $(\Delta,t - s)$-set $\mathbf{T}'_{0}(\mathcal{Q}) \subset \mathbf{T}_{0}(\mathcal{Q})$ again by $\mathbf{T}_{0}(\mathcal{Q})$.

Note that since $\mathbf{T}_{0} \in \mathcal{T}_{\Delta}$, which at the time of writing \nref{G2} was known as $\overline{\mathcal{T}}_{\Delta}$, one has
\begin{equation}\label{formA42} |\mathcal{T} \cap \mathbf{T}_{0}| \lessapprox \delta^{-s} = \Delta^{-2s}. \end{equation}
Second, every square $Q \in \mathbf{T}_{0}(\mathcal{Q})$ satisfies $\mathbf{T}_{0} \in \mathcal{T}^{\pi}_{\Delta}(Q)$ by definition. The family $\mathcal{T}^{\pi}_{\Delta}(Q)$ is a subset of the family $\mathcal{T}_{\Delta}(Q)$ originally defined in \nref{H2}, so the conclusion of \nref{H2} holds for $\mathbf{T}_{0}$:
\begin{equation}\label{formA43}  |\{(p,T) \in (\mathcal{P} \cap Q) \times \mathcal{T} : T \in \mathcal{T}(p) \cap \mathbf{T}_{0}\}| \gtrapprox \Delta^{-s - t}, \qquad Q \in \mathbf{T}_{0}(\mathcal{Q}). \end{equation}
We used \nref{G3} to estimate the constant $H_Q$ from \nref{H2}. Here one should recall that $\mathcal{P}$ was redefined above \eqref{form1} in such a way that $\mathcal{P} \cap Q = \mathcal{P}_{Q}$, where $\mathcal{P}_{Q}$ is the subset from \nref{H2}. It is also important to note that the families $\mathcal{P},\mathcal{T},\mathcal{T}(p)$ have not undergone further refinements which might influence the validity of \eqref{formA43}.

Fix $Q \in \mathbf{T}_{0}(\mathcal{Q})$. Since $|\mathcal{P} \cap Q| \approx \Delta^{-t}$ by \nref{G1}, and $|\mathcal{T}(p) \cap \mathbf{T}_{0}| \lessapprox \Delta^{-s}$ for every $p \in \mathcal{P}$ by Corollary \ref{auxLemma}, we may infer from \eqref{formA43} the existence of a subset $\mathcal{P}_{Q} \subset \mathcal{P}\cap Q$ with the properties
\begin{equation}\label{formA46} |\mathcal{P}_{Q}| \approx |\mathcal{P} \cap Q| \approx \Delta^{-t} \quad \text{and} \quad |\mathcal{T}(p) \cap \mathbf{T}_{0}| \approx \Delta^{-s} \text{ for all } p \in \mathcal{P}_{Q}. \end{equation}

\subsection{Reduction to $\mathbf{T}_{0} = \mathbf{D}([0,\Delta)^{2})$}\label{s:verticalTube} The purpose of this short section is to show that, without loss of generality, $\mathbf{T}_{0} = \mathbf{D}([0,\Delta)^{2})$. (Recall that in the appendix we are using the modified duality map given by \eqref{duality-appendix}.) This means that $\mathbf{T}_{0}$ is a "vertical" tube containing the origin. (If our results were formulated in terms of "ordinary" tubes in place of dyadic tubes, such reductions could be trivially performed by rotations and translations. The technicalities for dyadic tubes are more painful. A reader who is willing to believe that the difference between dyadic and "ordinary" tubes is not too critical is encouraged to skip this section.) Write $\sigma_{0} := \sigma(\mathbf{T}_{0}) \in [-1,1)$, so $\mathbf{T}_{0} = \mathbf{D}([\sigma_{0},\sigma_{0} + \Delta) \times [h_{0},h_{0} + \Delta))$ for some $h_{0} \in \R$. Then it is straightforward to compute that the affine map $F(x,y) := F_{\sigma_{0},h_{0}}(x,y) := (x - \sigma_{0}y - h_{0},y)$ satisfies $F(\mathbf{D}(a,b))=\mathbf{D}(a-\sigma_0,b-h_0)$, and hence transforms dyadic tubes into other dyadic tubes in the following manner:
\begin{displaymath} F\left( \mathbf{D}([\sigma_{1},\sigma_{1} + r) \times [h_{1},h_{1} + r)) \right) = \mathbf{D}([\sigma_{1} - \sigma_{0},\sigma_{1} - \sigma_{0} + r) \times [h_{1} - h_{0},h_{1} - h_{0} + r)), \end{displaymath}
for all $\sigma_{1},h_{1} \in \R$ and $r \in 2^{-\N}$. 
Applying the formula with $(\sigma_{1},h_{1}) = (\sigma_{0},h_{0})$ and $r = \Delta$ first shows that
\begin{displaymath} F (\mathbf{T}_{0}) = \mathbf{D}([0,\Delta)^{2}). \end{displaymath}
Second, the formula shows that $F$ maps dyadic $\delta$-tubes contained in $\mathbf{T}_{0}$ to dyadic $\delta$-tubes contained in $\mathbf{D}([0,\Delta)^{2})$. The map $F$ does not quite preserve dyadic squares: if $Q \in \mathbf{T}_{0}(\mathcal{Q})$, then $F(Q)$ is a quadrilateral of diameter $\sim \Delta$. The same is true for $p \in \mathcal{P}(Q)$: the images $F(p)$ are quadrilaterals of diameter $\sim \delta$, contained in $F(Q)$. It is also easy to check that the $(\Delta,t - s)$-set property of $\mathbf{T}_{0}(\mathcal{Q})$ is preserved under the map $F$. We also record that for $\pi(x,y) := \pi_{0}(x,y) = x$, we have
\begin{equation}\label{rev8} \pi \circ F = \pi_{\sigma_{0}} = \pi_{\sigma(\mathbf{T}_{0})}. \end{equation}

Thus, if $\mathbf{T}_{0} \neq \mathbf{D}([0,\Delta)^{2})$ to begin with, then $F(\mathbf{T}_{0}) = \mathbf{D}([0,\Delta)^{2})$, and we roughly speaking redefine the collections $\mathbf{T}_{0}(\mathcal{Q})$, $\mathcal{P}_Q$ for $Q \in \mathbf{T}_{0}(\mathcal{Q})$, $\mathcal{T}$, and $\mathcal{T}(p) \cap \mathbf{T}_{0}$ as images under $F$. We write "roughly speaking" here, because it will be convenient in the sequel to know that the collections $\mathbf{T}_{0}(\mathcal{Q})$ and $\mathcal{P}_{Q}$ still consist of dyadic squares, and this property is not quite preserved under $F$. Here is the more precise construction. Write $\overline{\mathbf{T}}_{0} := F(\mathbf{T}_{0}) = \mathbf{D}([0,\Delta)^{2})$ and $\overline{\mathcal{T}} := F(\mathcal{T})$. Then the property \eqref{formA42} remains valid in the form $|\overline{\mathcal{T}} \cap \overline{\mathbf{T}}_{0}| \lessapprox \delta^{-s}$. For $Q \in \mathbf{T}_{0}(\mathcal{Q})$ and $p \in \mathcal{P}_{Q}$ fixed, we use \eqref{formA46} to deduce that
\begin{displaymath}
|F(\mathcal{T}(p)) \cap \overline{\mathbf{T}}_{0}| = |\mathcal{T}(p) \cap \mathbf{T}_{0}| \approx \Delta^{-s}.
\end{displaymath}
Now, each dyadic tube $T \in F(\mathcal{T}(p)) \subset \overline{\mathcal{T}}$ intersects $F(p)$, a quadrilateral of diameter $\sim \delta$ which can be covered by $\sim 1$ dyadic $\delta$-squares. Since each tube $T \in F(\mathcal{T}(p)) \cap \overline{\mathbf{T}}_{0}$ meets at least one of these squares, we can fix one of them, say $\bar{p}$, such that the collection
\begin{displaymath} \overline{\mathcal{T}}(\bar{p}) := \{T \in F(\mathcal{T}(p)) : T \cap \bar{p} \neq \emptyset\} \text{ satisfies } |\overline{\mathcal{T}}(\bar{p}) \cap \overline{\mathbf{T}}_{0}| \approx \Delta^{-s}. \end{displaymath}
There is an additional technicality: while $p \in \mathcal{P}_{Q}$ implies $p \subset Q$, hence $F(p) \subset F(Q)$, it is not necessarily true that $\bar{p} \subset F(Q)$. Also, $F(Q)$ is not a dyadic $\Delta$-square. To fix these issues, we associate to $F(Q)$ a dyadic $\Delta$-square $\overline{Q}$ which contains the largest possible number of the $\delta$-squares $\bar{p}$, $p \in \mathcal{P}_{Q}$: in particular, $\overline{Q}$ contains $\bar{p}$ for all $p \in \mathcal{P}_{Q}' \subset \mathcal{P}_{Q}$, where $|\mathcal{P}_{Q}'| \sim |\mathcal{P}_{Q}|$. For this choice $\overline{Q}$, we set $\overline{\mathcal{P}}_{\overline{Q}} := \{\bar{p} : \bar{p} \subset \overline{Q}\}$. We record that
\begin{equation}\label{rev7} |\overline{\mathcal{P}}_{\overline{Q}}| \sim |\mathcal{P}_{Q}| \approx \Delta^{-t} \quad \text{and} \quad |\overline{\mathcal{T}}(\bar{p}) \cap \overline{\mathbf{T}}_{0}| \approx \Delta^{-s} \text{ for all } \bar{p} \in \overline{\mathcal{P}}_{\overline{Q}}, \end{equation}
and, for a suitable absolute constant $C > 0$,
\begin{equation}\label{rev9} \left(\cup \overline{\mathcal{P}}_{\overline{Q}} \right)_{C\delta} \supset F\left(\cup \mathcal{P}_{Q}' \right). \end{equation}
The properties \eqref{rev7} are the precise analogues of \eqref{formA46}, whereas \eqref{rev9} ensures that
\begin{equation}\label{rev10} \pi\left( \left(\cup \overline{\mathcal{P}}_{\bar{Q}} \right)_{C\delta} \right) \stackrel{\eqref{rev8}}{\supset} \pi_{\sigma_{0}(\mathbf{T}_{0})}(\cup \mathcal{P}'_{Q}). \end{equation}
This inclusion is useful, because when $Q \in \mathbf{T}_{0}(\calQ)$, as above, then $\mathbf{T}_{0} \in \mathcal{T}^{\pi}_{\Delta}(Q)$, which meant that $\sigma(\mathbf{T}_{0}) \in \Sigma(Q)$, and therefore $\pi_{\sigma_{0}(\mathbf{T}_{0})}(\cup S_{Q}(\mathcal{P}_{Q}'))$ contains a $(\Delta,s)$-set, according to Lemma \ref{projections}. From the inclusion \eqref{rev10}, we may infer that also
\begin{equation}\label{rev6} \pi(\cup S_{\overline{Q}}(\overline{\mathcal{P}}_{\overline{Q}})) \text{ contains a } (\Delta,s)\text{-set}. \end{equation}
The square "$Q$" in this discussion was chosen from $\mathbf{T}_{0}(\mathcal{Q})$, a collection known for its properties of being a $(\Delta,t - s)$-set of cardinality $|\mathbf{T}_{0}(\mathcal{Q})| \gtrapprox \Delta^{s - t}$, see \eqref{formA41} and below. We define $\overline{\mathbf{T}}_{0}(\mathcal{Q})$ as the collection of dyadic $\Delta$-squares $\overline{Q}$, derived from $Q \in \mathbf{T}_{0}(\mathcal{Q})$. Then $\overline{\mathbf{T}}_{0}(\mathcal{Q})$ retains the separation and cardinality properties stated above, and moreover $\overline{Q} \cap \overline{\mathbf{T}}_{0} \neq \emptyset$ for all $\overline{Q} \in \overline{\mathbf{T}}_{0}(\mathcal{Q})$: this is because there exist (many) squares $\bar{p} \subset \overline{Q}$ with $\overline{\mathcal{T}}(\bar{p}) \cap \overline{\mathbf{T}}_{0} \neq \emptyset$, so in particular $\overline{Q} \cap \overline{\mathbf{T}}_{0} \supset \bar{p} \cap \overline{\mathbf{T}}_{0} \neq \emptyset$.

From this point on, we "drop the bars", and rename the collections $\overline{\mathbf{T}}_{0}(\mathcal{Q})$, $\overline{\mathcal{P}}_{\bar{Q}}$ for $\overline{Q} \in \overline{\mathbf{T}}_{0}(\mathcal{Q})$, $\overline{\mathcal{T}}$, and $\overline{\mathcal{T}}(\bar{p})$ for $\bar{p} \in \overline{\mathcal{P}}_{\overline{Q}}$, as $\mathbf{T}_{0}(\mathcal{Q}),\mathcal{\mathcal{P}}_{Q},\mathcal{T},\mathcal{T}(p)$. For all practical purposes we are back to the notation above the present section, with the sole difference that $\mathbf{T}_{0}$ is now the vertical tube $\mathbf{T}_{0} = \mathbf{D}([0,\Delta)^{2})$. One thing that \emph{has} changed is that the new squares $Q$ and $p$ do not need to be entirely contained in $[0,1)^2$; this can be easily dealt with by slightly enlarging the reference square $[0,1)^2$, but to avoid overloading the notation we will continue to work inside $[0,1)^2$ in the sequel. Also, Lemma \ref{projections} is not literally true in the new setting, but all we will use from now on is that \eqref{rev6} holds.

\subsection{Finding a product-like structure}\label{quasiProduct} Assume then that $\mathbf{T}_{0} = \mathbf{D}([0,\Delta)^{2})$. Since $\sigma(\mathbf{T}_{0}) = 0$, the projection $\pi_{\sigma(\mathbf{T}_{0})}(x,y) = x =: \pi(x,y)$ is simply the projection to the first coordinate.

Recall, now, the sets $\mathcal{P}_{Q} \subset \mathcal{P} \cap Q$, defined above \eqref{formA46} for $Q \in \mathbf{T}_{0}(\mathcal{Q})$. They had cardinality $|\mathcal{P}_{Q}| \geq \Delta^{\mathbf{B}\epsilon}|\mathcal{P} \cap Q|$ for some constant $\mathbf{B} \geq 1$, which makes visible the constant behind the "$\approx$"-notation in \eqref{formA46}. The plan is to apply Lemma \ref{projections}, with this $\mathbf{B}$, to the sets $\mathcal{P}_{Q} \subset \mathcal{P} \cap Q$, for each $Q \in \mathbf{T}_{0}(\mathcal{Q})$.

By definition, $Q \in \mathbf{T}_{0}(\mathcal{Q})$ implies $\mathbf{T}_{0} \in \mathcal{T}^{\pi}_{\Delta}(Q)$, hence $0 = \sigma(\mathbf{T}_{0}) \in \Sigma(Q)$, recall \eqref{formA44}. Here $\Sigma(Q)$ was the family of good directions from Lemma \ref{projections}. For the set $\mathcal{P}_{Q}$ defined above, let $\mathcal{P}^{Q} = \{S_{Q}(p) : p \in \mathcal{P}_{Q}\}$. With this notation, it follows from Lemma \ref{projections} (or to be precise \eqref{rev6}, in case $\mathbf{T}_{0} \neq \mathbf{D}([0,\Delta)^{2})$ to begin with) that $\pi(\cup \mathcal{P}^{Q})$ contains a $(\Delta,s,\delta^{-C\epsilon})$-set $\Pi^{Q}$, where $C = \mathbf{C}(\mathbf{A} + \mathbf{B})$ is the constant from Lemma \ref{projections}. Thus, there exists a set $\overline{\mathcal{P}}^{Q}\subset \mathcal{P}^{Q}$ such that
\[
\Pi^{Q} = \pi(\cup \overline{\mathcal{P}}^{Q}).
\]
Since $\pi(\cup \overline{\mathcal{P}}^{Q}) \subset [0,1)$ is a union of dyadic $\Delta$-intervals, we may identify $\Pi^{Q}$ with the left end-points of these intervals. In particular,
\begin{displaymath} \Pi^{Q} \subset (\Delta \cdot \Z) \cap [0,1). \end{displaymath}
The set $\Pi^{Q}$ is not a very natural object, and we will be more interested in $\Pi_{Q}$, which is the "version of $\Pi^{Q}$ before the rescaling". Let us define this set more carefully. 
if
\begin{displaymath} Q = [\mathbf{x}_{Q},\mathbf{x}_{Q} + \Delta) \times [\mathbf{y}_{Q},\mathbf{y}_{Q} + \Delta), \end{displaymath}
then $S_{Q}^{-1}(x,y) = (\Delta x + \mathbf{x}_{Q},\Delta y + \mathbf{y}_{Q})$, and consequently
\begin{equation}\label{formA48} \Pi_{Q} = \Delta\cdot \Pi^{Q} + \mathbf{x}_{Q}:= \{ \Delta x + \mathbf{x}_{Q}: x\in \Pi^{Q}\}. \end{equation}
We also write $\overline{\mathcal{P}}_{Q}= \{ S_Q^{-1} q : q\in\overline{\mathcal{P}}^{Q}\}\subset\mathcal{P}_{Q}$ for the rescaled version of $\overline{\mathcal{P}}^{Q}$. With this notation, $\Pi_Q=\{ \mathbf{x}_p: p\in \overline{\mathcal{P}}_{Q} \}$, where $\pi p=[\mathbf{x}_p,\mathbf{x}_p+\delta)$.

\begin{figure}[h!]
\begin{center}
\begin{overpic}[scale = 1.4]{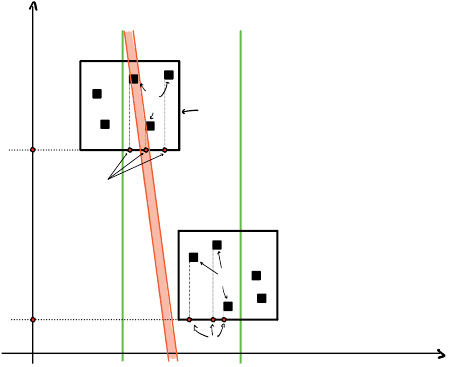}
\put(7.5,38){$(\mathbf{x}_{p},\mathbf{y}_{Q}) \in Z$}
\put(44,55){$Q$}
\put(32,56){$\overline{\mathcal{P}}_{Q}$}
\put(45,3.5){$Z$}
\put(8,49){$\mathbf{y}_{Q}$}
\put(55,25){$Q'$}
\put(35,35){$T$}
\put(8,12){$\mathbf{y}_{Q'}$}
\put(47,65){$\mathbf{T}_{0}$}
\put(47.5,18){$\overline{\mathcal{P}}_{Q'}$}
\end{overpic}
\caption{The tube $\mathbf{T}_{0}$ drawn with a green outline, and two squares $Q,Q' \in \mathbf{T}_{0}(\mathcal{Q})$. Part of the set $Z$ is also visible as the union of the small red circles on the lower boundaries of the squares $Q,Q' \in \mathbf{T}_{0}(\mathcal{Q})$.}\label{fig1}
\end{center}
\end{figure}
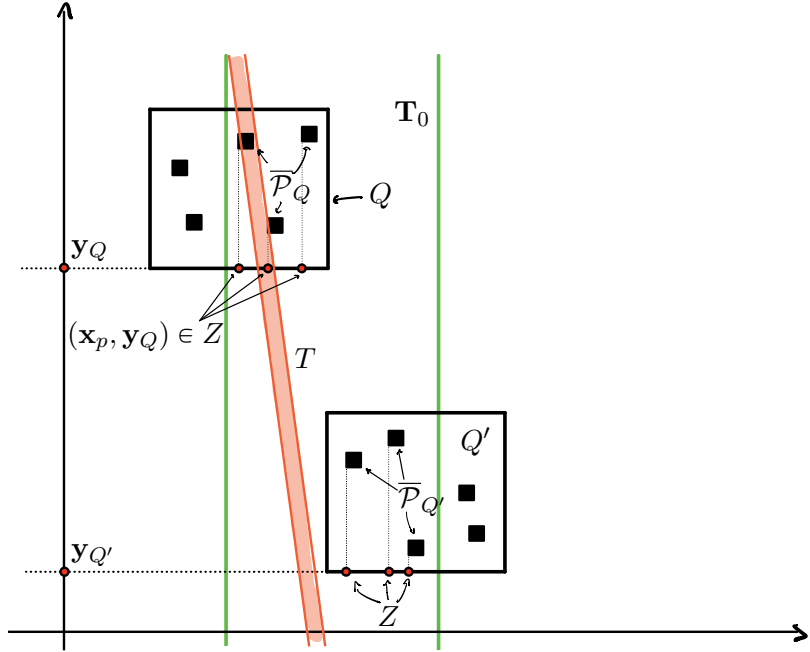

Note that since $Q \cap \mathbf{T}_{0} \neq \emptyset$, $Q \subset [0,1)^{2}$, and $\mathbf{T}_{0} = \mathbf{D}([0,\Delta)^{2})$, one has $\mathbf{x}_{Q} \in [0,2\Delta]$ for all $Q \in \mathbf{T}_{0}(\mathcal{Q})$. Moreover, the $(\Delta,t - s)$-set property of $\mathbf{T}_{0}(\mathcal{Q})$ implies that
\begin{displaymath} \mathbf{Y} := \{\mathbf{y}_{Q} : Q \in \mathbf{T}_{0}(\mathcal{Q})\} \subset (\Delta \cdot \Z) \cap [0,1) \text{ is a $(\Delta,t - s)$-set}. \end{displaymath}
Every point $\mathbf{y} \in \mathbf{Y}$ can correspond to only $\sim 1$ squares $Q \in \mathbf{T}_{0}(Q)$, and by removing a few excess squares from $\mathbf{T}_{0}(Q)$, we assume that the map $Q \mapsto \mathbf{y}_{Q}$ is injective. Then for $\mathbf{y}\in \mathbf{Y}$ we can unambiguously define $\Pi_{\mathbf{y}}= \Pi_Q$, $\Pi^{\mathbf{y}}= \Pi^Q$ if $\mathbf{y}=\mathbf{y}_Q$. With these definitions in hand, consider the sets
\begin{equation}\label{ZZ} Z := \bigcup_{\mathbf{y} \in \mathbf{Y}} \Pi_{\mathbf{y}} \times \{\mathbf{y}\} \quad \text{and} \quad \mathbf{Z} := \bigcup_{\mathbf{y} \in \mathbf{Y}} (\Delta^{-1}\Pi_{\mathbf{y}}) \times \{\mathbf{y}\}. \end{equation}
For a depiction of the set $Z$, and related concepts, see Figure \ref{fig1}. Thus $Z,\mathbf{Z}$ are finite, discrete, sets, whose projections to the $y$-axis are the $(\Delta,t - s)$-set $\mathbf{Y}$. Their projections to the $x$-axis are the unions of the sets $\Pi_{\mathbf{y}}$ or $\Delta^{-1}\Pi_{\mathbf{y}}$, respectively, for $\mathbf{y} \in \mathbf{Y}$. Let $\mathbf{x}_{\mathbf{y}}=\mathbf{x}_{Q}$ for $\mathbf{y} = \mathbf{y}_{Q}$. The sets
\begin{equation}\label{formA53} \Delta^{-1}\Pi_{\mathbf{y}} \stackrel{\eqref{formA48}}{=} \Pi^{\mathbf{y}} + \tfrac{\mathbf{x}_{\mathbf{y}}}{\Delta} \subset [0,3] \end{equation}
are $(\Delta,s)$-sets. The second inclusion follows from $\mathbf{x}_{\mathbf{y}} \in [0,2\Delta]$. Therefore, the set $\mathbf{Z}$ looks somewhat like the product of a $(\Delta,s)$-set and a $(\Delta,t - s)$-set; this would be precisely true if the various sets $\Delta^{-1}\Pi_{\mathbf{y}}$ were all the same.

For every $\mathbf{z} \in \mathbf{Z}$, we will next construct a $(\Delta,s)$-set of dyadic $\Delta$-tubes $\mathbf{T}(\mathbf{z}) \subset \mathcal{T}^{\Delta}$ with the following properties:
\begin{enumerate}
\item $\mathbf{z} \in \mathbf{T}$ for all $\mathbf{T} \in \mathcal{T}(\mathbf{z})$,
\item the union $\mathcal{T}(\mathbf{Z}) = \bigcup_{\mathbf{z} \in \mathbf{Z}} \mathcal{T}(\mathbf{z})$ satisfies $|\mathcal{T}(\mathbf{Z})| \lessapprox \Delta^{-2s}$.
\end{enumerate}
To accomplish this, fix
\begin{displaymath} \mathbf{z} = (\Delta^{-1}\mathbf{x}_{p},\mathbf{y}_{Q}) \in \mathbf{Z}, \end{displaymath}
where $Q \in \mathbf{T}_{0}(\mathcal{Q})$, and $\mathbf{x}_{p}$ is the left end-point of some interval $\pi(p)=[\mathbf{x}_p,\mathbf{x}_p+\delta) \in \mathcal{D}_{\delta}([0,1))$, with $p \in \overline{\mathcal{P}}_{Q} \subset \mathcal{P}_{Q}$. Thus $(\mathbf{x}_{p},\mathbf{y}_{Q}) \in Z$. 
Then, since $p \in \mathcal{P}_{Q}$, we recall from \eqref{formA46} that $|\mathcal{T}(p) \cap \mathbf{T}_{0}| \approx \Delta^{-s}$. The idea is to modify (essentially dilate horizontally) the tubes from $\mathcal{T}(p) \cap \mathbf{T}_{0}$ to generate the family $\mathcal{T}(\mathbf{z})$. To this end, fix $T = \mathbf{D}([\sigma_{T},\sigma_{T} + \delta) \times [h_{T},h_{T} + \delta)) \in \mathcal{T}(p) \cap \mathbf{T}_{0}$ (this would be the red tube in Figure \ref{fig1}). Then
\begin{displaymath} T \cap p \neq \emptyset \quad \text{and} \quad [\sigma_{T},\sigma_{T} + \delta) \times [h_{T},h_{T} + \delta) \subset [0,\Delta)^{2}. \end{displaymath}
Fix also an arbitrary point $(x,y) \in T \cap p\subset T\cap Q$. Then $x \in [\mathbf{x}_{p},\mathbf{x}_{p} + \delta)$ and $y \in [\mathbf{y}_{Q},\mathbf{y}_{Q} + \Delta)$, and since $(x,y) \in T$, the coordinates $x,y$ satisfy the relation
\begin{equation}\label{form2} x = \sigma'y + h' \quad \text{for some} \quad (\sigma',h') \in [\sigma_{T},\sigma_{T} + \delta) \times [h_{T},h_{T} + \delta) \subset [0,\Delta)^{2}. \end{equation}
By manipulating the equation \eqref{form2}, one finds that
\begin{equation}\label{formA50} \mathbf{x}_{p} = \sigma'\mathbf{y}_{Q} + h' + [(\mathbf{x}_{p} - x) + \sigma'(y - \mathbf{y}_{Q})]. \end{equation}
Here $0\le (\mathbf{x}_{p} - x) + \sigma'(y - \mathbf{y}_{Q}) \leq \delta + \Delta^{2} = 2\delta$, so \eqref{form2}-\eqref{formA50} imply that
\begin{equation}\label{formA51} (\mathbf{x}_{p},\mathbf{y}_{Q}) \in \mathbf{D}([\sigma_{T},\sigma_{T} + \delta) \times [h_{T},h_{T} + 3\delta)). \end{equation}
Before proceeding, we record that the affine map $A(x,y) := (\Delta^{-1}x,y)$ acts on dual sets of lines in the following way: $A\mathbf{D}(I \times J) = \mathbf{D}((\Delta^{-1}I) \times (\Delta^{-1}J))$.
Combining this straightforward fact with \eqref{formA51}, we find that
\begin{equation}\label{formA52} \mathbf{z} = (\Delta^{-1}\mathbf{x}_{p},\mathbf{y}_{Q}) \in \mathbf{D}([\Delta^{-1}\sigma_{T},\Delta^{-1}\sigma_{T} + \Delta) \times [\Delta^{-1}h_{T},\Delta^{-1}h_{T} + 3\Delta)). \end{equation}
The set on the right can be covered by $3$ dyadic $\Delta$-tubes with slope $\Delta^{-1}\sigma_{T} \in \Delta \cdot \Z \cap [0,1)$ (note also that $h_{T} \in [0,\Delta)$, so $\Delta^{-1}h_{T} \in [0,1)$). One of these tubes contains the point $\mathbf{z}$ by \eqref{formA52}, and one adds this tube to $\mathcal{T}_{\mathbf{z}}$.

The tube $T \in \mathcal{T}(p) \cap \mathbf{T}_{0}$ was a "free parameter" in the argument above. Therefore, we may add to $\mathcal{T}(\mathbf{z})$ one tube for each choice of $T \in \mathcal{T}(p) \cap \mathbf{T}_{0}$. Then, by \eqref{formA52}, the slope set $\sigma(\mathcal{T}(\mathbf{z}))$ satisfies
\begin{displaymath} \sigma(\mathcal{T}(\mathbf{z})) = \Delta^{-1} \cdot \sigma(\mathcal{T}(p) \cap \mathbf{T}_{0}). \end{displaymath}
It follows that $|\mathcal{T}(\mathbf{z})| \sim |\sigma(\mathcal{T}(\mathbf{z}))| \sim |\mathcal{T}(p) \cap \mathbf{T}_{0}| \approx \Delta^{-s}$ by \eqref{formA46} and Lemma \ref{tubesAndSlopes}. The latter lemma also implies that $\sigma(\mathcal{T}(p))$ is a $(\delta,s)$-set, and therefore if $I \in \mathcal{D}_{r}$, $r \geq \Delta$, then
\begin{displaymath} |\sigma(\mathcal{T}(\mathbf{z})) \cap I|_{\Delta} \leq |\sigma(\mathcal{T}(p)) \cap \Delta I|_{\delta} \lessapprox |\mathcal{T}(p)| \cdot (\Delta r)^{s} \stackrel{\eqref{MBound}}{\approx} \Delta^{-s} \cdot r^{s} \approx |\sigma(\mathcal{T}(\mathbf{z}))| \cdot r^{s}.  \end{displaymath}
Another application of Lemma \ref{tubesAndSlopes} concludes the proof that $\mathcal{T}(\mathbf{z})$ is a $(\Delta,s)$-set.

It remains to show that $|\mathcal{T}(\mathbf{Z})| \lessapprox \Delta^{-2s}$. This follows from \eqref{formA42}, which stated that $|\mathcal{T} \cap \mathbf{T}_{0}| \lessapprox \Delta^{-2s}$. In particular, every tube $T \in \mathcal{T}(p) \cap \mathbf{T}_{0} \subset \mathcal{T} \cap \mathbf{T}_{0}$ mentioned above lies in this universal collection $\mathcal{T} \cap \mathbf{T}_{0}$. More precisely, every tube $\mathbf{T} \in \mathcal{T}(\mathbf{Z})$ has the form
\begin{displaymath} \mathbf{D}([\Delta^{-1}\sigma_{T},\Delta^{-1}\sigma_{T} + \Delta) \times [\Delta^{-1}h_{T} + \rho,\Delta^{-1}h_{T} + \rho + \Delta)), \end{displaymath}
where $\mathbf{D}([\sigma_{T},\sigma_{T} + \delta) \times [h_{T},h_{T} + \delta)) \in \mathcal{T} \cap \mathbf{T}_{0}$, and $\rho \in \{0,\Delta,2\Delta\}$. Therefore $|\mathcal{T}(\mathbf{Z})| \leq 3|\mathcal{T} \cap \mathbf{T}_{0}| \approx \Delta^{-2s}$, as desired.

\subsection{Concluding the proof of Theorem \ref{mainAppendix}}

\begin{proposition}\label{productProp} Given $0 < s < 1$ and $\tau > 0$, there exists a number $\eta = \eta(s,\tau) > 0$ such that the following holds for all $\delta > 0$ small enough, depending on $s,\tau$.  Let $\mathbf{Y} \subset (\delta\cdot \mathbb{Z}) \cap [0,1)$ be a $(\delta,\tau,\delta^{-\eta})$-set, and for each $\mathbf{y} \in \mathbf{Y}$, assume that $\mathbf{X}_{\mathbf{y}} \subset (\delta\cdot \mathbb{Z})\cap [0,1)$ is a $(\delta,s,\delta^{-\eta})$-set. Let
\begin{displaymath} \mathbf{Z} := \bigcup_{\mathbf{y} \in \mathbf{Y}} \mathbf{X}_{\mathbf{y}} \times \{\mathbf{y}\}. \end{displaymath}
For every $\mathbf{z} \in \mathbf{Z}$, assume that $\mathcal{T}(\mathbf{z}) \subset \mathcal{T}^{\delta}$ is a $(\delta,s,\delta^{-\eta})$-set of dyadic $\delta$-tubes such that $\mathbf{z} \in T$ for all $T \in \mathcal{T}(\mathbf{z})$. Then $|\calT| \geq \delta^{-2s - \eta}$, where $\mathcal{T} = \bigcup_{\mathbf{z} \in \mathbf{Z}} \mathcal{T}(\mathbf{z})$.
\end{proposition}

This proposition is essentially \cite[Proposition 4.36]{Orponen20}. For the sake of self-containedness, we provide a proof, which is much shorter than that of \cite[Proposition 4.36]{Orponen20}. The decrease in length is partially due to the fact that whereas \cite{Orponen20} derived Proposition \ref{productProp} from a sum-product theorem of Bourgain \cite{Bourgain10}, we now reduce matters to a slightly stronger, and newer, version of Bourgain's theorem, due to He \cite{MR4148151}.

\begin{proof}[Proof of Proposition \ref{productProp}]
In this proof only, we use the notation $A\lessapprox B$ to denote $A\le C \delta^{-C\eta} B$ for some absolute constant $C\ge 1$, and similarly for $A\gtrapprox B$, $A\approx B$. Likewise, a $(\delta,u)$-set stands for a $(\delta,u, C\delta^{-C\eta})$-set.

We argue by contradiction: suppose $|\mathcal{T}|\le \delta^{-2s-\eta}$. By Lemma \ref{lem:thin-delta-subset}, without loss of generality we may assume that $|\mathbf{X}_{\mathbf{y}}|\le \delta^{-s}$ for $\mathbf{y}\in\mathbf{Y}$ and $|\mathcal{T}(\mathbf{z})| \le \delta^{-s}$ for $\mathbf{z}\in\mathbf{Z}$.

Fix $\mathbf{y}\in\mathbf{Y}$. We claim that
\[
\mathcal{T}(\mathbf{y}) := \bigcup_{\mathbf{x}\in\mathbf{X}_{\mathbf{y}} } \mathcal{T}(\mathbf{x},\mathbf{y})
\]
is a $(\delta,2s)$-set. First, since $\mathbf{z}\in T$ for $T\in\mathcal{T}(\mathbf{z})$ and the tubes in $\mathcal{T}$ are far from horizontal, the union in the definition of $\mathcal{T}(\mathbf{y})$ has bounded overlap, and therefore  $|\mathcal{T}(\mathbf{y})| \gtrapprox \delta^{-2s}$, using that $\mathbf{X}_{\mathbf{y}}$, $\mathcal{T}(\mathbf{x},\mathbf{y})$ are $(\delta,s)$-sets.

We write $\pi_{\mathbf{y}}(a,b):=a\mathbf{y} + b$ (this differs from the previous convention in the appendix, but these projections are most useful now). We identify $T=\mathbf{D}([a,a+\delta)\times[b,b+\delta))$ with $(a,b)$ in the sequel. Under this identification, $\mathcal{T}\subset (\delta\cdot \mathbb{Z}^2)\cap [-1,2)^2$. Using that $(\mathbf{x},\mathbf{y})\in T$ for $T\in \mathcal{T}(\mathbf{x},\mathbf{y})$, we will next see that the $(\delta,s)$-set $\mathcal{T}(\mathbf{x},\mathbf{y})$ (under the identification above) satisfies
\begin{equation}\label{rev11} \pi_{\mathbf{y}}(\mathcal{T}(\mathbf{x},\mathbf{y})) \subset B(\mathbf{x},2\delta), \qquad \mathbf{y} \in \mathbf{Y}, \, \mathbf{x} \in \mathbf{X}_{\mathbf{y}}. \end{equation}
Indeed, if $(\mathbf{x},\mathbf{y}) \in \mathbf{D}(a',b')$ with $(a',b') \in [a,a + \delta) \times [b,b + \delta)$, then $\mathbf{x} = a'\mathbf{y} + b'$, recalling \eqref{duality-appendix}, and consequently $|\pi_{\mathbf{y}}(a,b) - \mathbf{x}| = |\pi_{\mathbf{y}}(a,b) - \pi_{\mathbf{y}}(a',b')| \leq 2\delta$. This proves \eqref{rev11}. As an immediate corollary of \eqref{rev11},
\begin{equation}\label{rev12} \pi_{\mathbf{y}}(\mathcal{T}(\mathbf{y})) \subset \mathbf{X}_{\mathbf{y}}(2\delta), \qquad \mathbf{y} \in \mathbf{Y}. \end{equation}
Finally, to show the $(\delta,2s)$-set property of $\mathcal{T}(\mathbf{y})$, fix an $r$-square $R\subset [0,1)^2$, $r\in [\delta,1]$. Then $\pi_{\mathbf{y}}(R)$ is an interval of length $\lesssim r$, and we let $I$ be the $(2\delta)$-neighbourhood of $\pi_{\mathbf{y}}(R)$. From \eqref{rev11} we see that if $\mathbf{x} \in \mathbf{X}_{\mathbf{y}}$ is such that $\mathcal{T}(\mathbf{x},\mathbf{y}) \cap R \neq \emptyset$, then $B(\mathbf{x},2\delta) \cap \pi_{\mathbf{y}}(R) \neq \emptyset$, and consequently $\mathbf{x} \in \mathbf{X}_{\mathbf{y}} \cap I$. Thus, using that $\mathbf{X}_{\mathbf{y}}$ and $\mathcal{T}(\mathbf{x},\mathbf{y})$ are both $(\delta,s)$-sets of cardinality $\le \delta^{-s}$,
\begin{align*}
|\mathcal{T}(\mathbf{y})\cap R| & = \bigg| \bigcup_{\mathbf{x} \in \mathbf{X}_{\mathbf{y}} \cap I} \mathcal{T}(\mathbf{x},\mathbf{y}) \cap R \bigg|\\
& \leq |\mathbf{X}_{\mathbf{y}}\cap I|\cdot \max_{ \mathbf{x}\in \mathbf{X}_{\mathbf{y}}}  |\mathcal{T}(\mathbf{x},\mathbf{y})\cap R|\\
& \lessapprox r^{2s}\delta^{-2s} \lessapprox r^{2s}|\mathcal{T}(\mathbf{y})|. \end{align*}

The idea in the remainder of the argument is the following. The $(\delta,2s)$-set $\mathcal{T}(\mathbf{y})$ has the property \eqref{rev12}, which in particular implies $|\pi_{\mathbf{y}}(\mathcal{T}(\mathbf{y}))|_{\delta} \lesssim \delta^{-s}$. Since we (counter-)assumed that $|\mathcal{T}| \lessapprox \delta^{-2s}$, this means that for each $\mathbf{y} \in \mathbf{Y}$, the set $\mathcal{T}$ has a "substantial" subset $\mathcal{T}(\mathbf{y}) \subset \mathcal{T}$ whose $\pi_{\mathbf{y}}$-projection has only $\tfrac{1}{2}$ the dimension of $\mathcal{T}$. This would immediately contradict a projection theorem of Bourgain \cite{Bourgain10} if we knew that $\mathcal{T}$ is a $(\delta,2s)$-set. This is true after to passing to a suitable refinement $\overline{\mathcal{T}}$, defined next.

Since $|\mathcal{T}(\mathbf{y})|\approx |\mathcal{T}|$, $\mathbf{y}\in\mathbf{Y}$, there is a set $\overline{\mathcal{T}}\subset\mathcal{T}$ such that $|\overline{\mathcal{T}}|\approx |\mathcal{T}|$ and each $T\in\overline{\mathcal{T}}$ belongs to $\approx |\mathbf{Y}|$ of the sets $\mathcal{T}(\mathbf{y})$. This implies that $\overline{\mathcal{T}}$ is a $(\delta,2s)$-set: indeed, if $R$ is an $r$-square, $r\in [\delta,1]$, then
\[
|\overline{\mathcal{T}}\cap R| \approx \sum_{T \in \overline{\mathcal{T}} \cap R} \frac{1}{|\mathbf{Y}|} \sum_{\mathbf{y} \in \mathbf{Y}} \mathbf{1}_{\mathcal{T}(\mathbf{y})}(T)  \leq \frac{1}{|\mathbf{Y}|} \sum_{y \in \mathbf{Y}} |\mathcal{T}(\mathbf{y}) \cap R| \lessapprox r^{2s} \delta^{-2s}\lessapprox r^{2s}|\overline{\mathcal{T}}|,
\]
using that $\mathcal{T}(\mathbf{y})$ is a $(\delta,2s)$-set for each $\mathbf{y} \in \mathbf{Y}$.

Now, since $\sum_{\mathbf{y} \in \mathbf{Y}} |\overline{\mathcal{T}} \cap \mathcal{T}(\mathbf{y})| = \sum_{T \in \overline{\mathcal{T}}} |\{\mathbf{y} \in \mathbf{Y} : T \in \mathcal{T}(\mathbf{y})\}| \approx |\overline{\mathcal{T}}||\mathbf{Y}|$, there is a subset $\overline{\mathbf{Y}}\subset\mathbf{Y}$ such that $|\overline{\mathbf{Y}}|\approx |\mathbf{Y}|$ and
\[
|\overline{\mathcal{T}}(\mathbf{y}) | := |\overline{\mathcal{T}}\cap \mathcal{T}(\mathbf{y})| \approx \delta^{-2s}\approx |\overline{\mathcal{T}}|, \quad \mathbf{y}\in \overline{\mathbf{Y}}.
\]
Then $\overline{\mathbf{Y}}$ is a $(\delta,\tau)$-set, with the following property: for all $\mathbf{y}\in \overline{\mathbf{Y}}$, there is a subset $\overline{\mathcal{T}}(\mathbf{y})$ of the $(\delta,2s)$-set $\overline{\mathcal{T}}$, with comparable cardinality (in the $\approx$ sense), and such that
\[
|\pi_{\mathbf{y}}(\overline{\mathcal{T}}(\mathbf{y}))|_\delta \le |\pi_{\mathbf{y}}(\mathcal{T}(\mathbf{y}))|_\delta \stackrel{\eqref{rev12}}{\lesssim} |\mathbf{X}_{\mathbf{y}}| \le \delta^{-s}.
\]
However, these facts contradict Bourgain's discretized projection theorem \cite{Bourgain10}, in the refined form presented by W. He \cite[Theorem 1]{MR4148151}, if $\eta$ is sufficiently small in terms of $\tau, s$ only. More precisely, in our setting $n=2$, $m=1$, $A$ in \cite[Theorem 1]{MR4148151} corresponds to our $\overline{\mathcal{T}}$ while the measure $\mu$ is $|\overline{\mathbf{Y}}|^{-1}\sum_{\mathbf{y}\in \overline{\mathbf{Y}}} \delta^{-1}\mathcal{L}^{1}|_{[\mathbf{y},\mathbf{y}+\delta)}$ (and we identify $\pi_{\mathbf{y}}$ with an orthogonal projection in the standard way). We conclude that the counter-assumption $|\mathcal{T}| \le \delta^{-2s-\eta}$ cannot hold. \end{proof}

To conclude the proof of Theorem \ref{mainAppendix}, we apply Proposition \ref{productProp} to the set $\mathbf{Z}$ defined in \eqref{ZZ}, with $\tau := t  - s > 0$, and at scale "$\Delta$" in place of "$\delta$". As noted in \eqref{formA53}, the sets $\mathbf{X}_{\mathbf{y}} = \Delta^{-1}\Pi_{\mathbf{y}}$ are $(\Delta,s)$-subsets of $[0,3]$, and the difference between $[0,1]$ and $[0,3]$ is irrelevant. The hypotheses of Proposition \ref{productProp} will be valid if the initial parameter $\epsilon=\epsilon(s,t) > 0$ from the counter assumption \eqref{counterAss} was chosen so small that $\mathbf{Y}$ is a $(\Delta,t - s,\Delta^{-\eta})$-set, each $\mathbf{X}_{\mathbf{y}}$ is a $(\Delta,s,\Delta^{-\eta})$-set, and the families $\mathcal{T}(\mathbf{z})$ constructed in the previous section are $(\Delta,s,\Delta^{-\eta})$-sets, with $\eta = \eta(s,t - s) > 0$. Then, the conclusion of Proposition \ref{productProp} will contradict the upper bound $|\mathcal{T}(\mathbf{Z})| \lessapprox \Delta^{-2s}$, established at the end of the previous section, if $\epsilon > 0$ is sufficiently small. This contradiction shows that our initial counter assumption \eqref{counterAss} has to fail for some $\epsilon = \epsilon(s,t) > 0$ sufficiently small, and the proof of Theorem \ref{mainAppendix} is complete.

\def\cprime{$'$}


\end{document}